\documentclass[a4paper,12pt,leqno]{amsart}
\usepackage{amsmath,amssymb,amsxtra,comment,graphicx,psfrag}
\usepackage{bm,mathrsfs}
\usepackage{mathtools}
\usepackage{xspace}
\usepackage{color}

\usepackage{enumerate}

\usepackage{hyperref}
\usepackage{pdfsync}

\usepackage{hhline}

\allowdisplaybreaks

\definecolor{blue}{rgb}{0,0.0,0.9}

 at 8. truept

\def\R{{\mathbb R}}
\def\P{{\mathbb P}}
\def\N{{\mathbb N}}

\def\T{T}
\def\TT{{\mathcal T}}

\def\d{{\mathrm d}}
\def\e{{\rm e}}
\def\i{{\rm i}}
\def\tr|{|\!|\!|}

\DeclareMathOperator\Real {Re}

\DeclareMathOperator\diver {div}
\DeclareMathOperator\kernel {ker}
\DeclareMathOperator\supp {supp}

\theoremstyle{plain}
\newtheorem{theorem}{Theorem}[section]
\newtheorem{lemma}{Lemma}[section]
\newtheorem{proposition}{Proposition}[section]

\theoremstyle{definition}

\newtheorem{remark}{Remark}[section]

\numberwithin{equation}{section}
\numberwithin{table}{section}
\numberwithin{figure}{section}

\newcommand{\diff}{\frac{\d}{\d t}}

\def \d {\mathrm{d}}

\newcommand{\nb}{\nabla}
\newcommand{\Om}{\varOmega}
\newcommand{\pa}{\partial}
\def \t {(t)}
\newcommand{\vphi}{\varphi}

\def\half{\frac{1}{2}}

\newcommand\andquad{\quad\hbox{ and }\quad}

\newcommand{\be}{\bm{e}}
\newcommand{\de}{\dot{\be}}

\newcommand{\bd}{\bm{d}}
\newcommand{\bi}{\bm{i}}
\newcommand{\bj}{\bm{j}}
\newcommand{\m}{\bm{m}}
\newcommand{\wm}{\widehat{\bm{m}}}
\newcommand{\ms}{\bm{m}_{\star}}
\newcommand{\dms}{\dot{\bm{m}}_{\star}}
\newcommand{\wms}{\widehat{\bm{m}}_{\star}}
\newcommand{\bq}{\bm{q}}
\newcommand{\br}{\bm{r}}
\newcommand{\bs}{\bm{s}}
\newcommand{\bv}{\bm{v}}
\newcommand{\bw}{\bm{w}}
\newcommand{\bphi}{\bm{\vphi}}

\newcommand{\beh}{\bm{e}_h}
\newcommand{\deh}{\dot{\bm{e}}_h}
\newcommand{\weh}{\widehat{\bm{e}}_h}
\newcommand{\bfEh}{{\mathbf E}_h}
\newcommand{\bdh}{\bm{d}_h}
\newcommand{\mh}{\bm{m}_h}
\newcommand{\dmh}{\dot{\bm{m}}_h}
\newcommand{\wmh}{\widehat{\bm{m}}_h}
\newcommand{\msh}{\bm{m}_{\star,h}}
\newcommand{\dmsh}{\dot{\bm{m}}_{\star,h}}
\newcommand{\wmsh}{\widehat{\bm{m}}_{\star,h}}
\newcommand{\bph}{\bm{p}_h}
\newcommand{\bqh}{\bm{q}_h}
\newcommand{\brh}{\bm{r}_h}

\newcommand{\bphih}{\bm{\vphi}_h}
\newcommand{\bpsih}{\bm{\psi}_h}

\newcommand{\bfE}{{\mathbf E}}

\newcommand{\bophi}{\mathbf{\phi}}

\def\Id{{\mathbf I}}
\def\P{{\mathbf P}}

\def\PPh{\varPi_h}
\def\PPPh{{\mathbf \Pi}_h}
\def\RRh{R_h}
\def\RRRh{{\mathbf R}_h}

\newcommand{\mnull}{\m_h}

\renewcommand{\leq}{\leqslant}
\renewcommand{\geq}{\geqslant}

\begin{document}

	\title[Higher-order discretization of the LLG equation]
	{Higher-order linearly implicit full discretization of  the Landau--Lifshitz--Gilbert equation}

	\author[Georgios Akrivis]{Georgios Akrivis}
	\address{Department of Computer Science \& Engineering, University of Ioannina, 451$\,$10
		Ioannina, Greece, and Institute of Applied and Computational Mathematics, 
		FORTH, 700$\,$13 Heraklion, Crete, Greece} 
	\email {\href{mailto:akrivis@cse.uoi.gr}{akrivis{\it @\,}cse.uoi.gr}} 
	
	\author[Michael Feischl]{Michael Feischl}
	\address{Institute for Analysis and Scientific Computing (E 101), 
		Technical University Wien, Wiedner Hauptstrasse 8-10, 1040 Vienna, Austria}
	\email {\href{mailto:michael.feischl@kit.edu}{michael.feischl{\it @\,}kit.edu}} 
	\email {\href{mailto:michael.feischl@tuwien.ac.at}{michael.feischl{\it @\,}tuwien.ac.at}}

	\author[Bal\'azs Kov\'acs]{Bal\'azs Kov\'acs}
	\address{Mathematisches Institut, Universit\"at T\"ubingen, Auf der Morgenstelle, 
		D-72076 T\"ubingen, Germany}
	\email {\href{mailto:kovacs@na.uni-tuebingen.de}{kovacs{\it @\,}na.uni-tuebingen.de}}

	\author[Christian Lubich]{Christian Lubich}
	\address{Mathematisches Institut, Universit\"at T\"ubingen, Auf der Morgenstelle, 
		D-72076 T\"ubingen, Germany}
	\email {\href{mailto:lubich@na.uni-tuebingen.de}{lubich{\it @\,}na.uni-tuebingen.de}}


	\date{\today}

	\begin{abstract}
		For the Landau--Lifshitz--Gilbert (LLG) equation of micromagnetics
		we study  linearly implicit backward difference 
		formula (BDF) time discretizations up to order $5$ combined with higher-order non-conforming finite element space discretizations, 
		which are based on the weak formulation due to Alouges but use approximate tangent spaces that are defined by  $L^2$-averaged
		instead of nodal orthogonality constraints. 
		We prove stability and optimal-order error bounds in the situation of a sufficiently regular solution.
		For the BDF methods of orders $3$ to~$5$, 
		this requires 
		that the damping parameter in the LLG equations be above a positive threshold;  this condition is  not needed 
		for the A-stable methods of orders $1$ and $2$, for which furthermore
		a discrete energy inequality irrespective of solution regularity is proved.
	\end{abstract}
	
	\maketitle

	\section{Introduction}\label{Sec:intr}

	\subsection{Scope}
	In this paper we study the convergence of higher-order time and space discretizations of the Landau--Lifshitz--Gilbert (LLG) 
	equation, which is the basic model for phenomena in micromagnetism, such as in recording media \cite{GuoD08,Pro01}.

	The main novelty of the paper lies in the construction and analysis of what is apparently the first numerical method for the 
	LLG equation that is second-order convergent in both space and time to sufficiently regular solutions and that satisfies, 
	as an important robustness property irrespective of regularity, a discrete energy inequality analogous to 
	that of the continuous problem.
	
	We study discretization in time by linearly implicit backward difference formulae (BDF) up to order $5$ and discretization 
	in space by  finite elements of arbitrary polynomial degree.
	For the BDF methods up to order $2$ we prove optimal-order error bounds in the situation of a sufficiently 
	regular solution and a discrete energy inequality irrespective of solution regularity under very weak regularity assumptions on the data. 
	For the BDF methods of orders $3$ to $5$, we prove optimal-order error bounds in the situation of a sufficiently regular solution
	under the additional condition that the damping parameter in the LLG equation be above a method-dependent positive 
	threshold. However, no discrete energy inequality irrespective of solution regularity is obtained for the BDF methods of orders $3$ to $5$.

	%
	%
	
	The discretization in space is done by a higher-order non-conforming finite element method based on the approach 
	of Alouges~\cite{Al08,AlJ06}, which uses a projection to an approximate tangent space to the normality constraint. Contrary to the pointwise 
	orthogonality constraints in the nodes, which define the approximate tangent space in those papers and yield only first-order convergence also for finite elements with higher-degree polynomials, we here enforce 
	orthogonality averaged over the finite element basis functions. 
	With these modified approximate tangent spaces we prove $H^1$-convergence of optimal 
	order in space and time under the assumption of a sufficiently regular solution.

	Key issues in the error analysis  are the properties of the orthogonal projection onto the approximate tangent space, 
	the higher-order consistency error analysis, and the proof of stable error propagation, which is based on non-standard 
	energy estimates and uses both $L^2$ and maximum norm finite element analysis.

	\subsection{The Landau--Lifshitz--Gilbert equation}
	
	The standard phenomenological model for micromagnetism is  provided by the Landau--Lifshitz (LL) equation
	\begin{equation}
	\label{LLG1}
	\partial_t \bm{m} =  - \bm{m}\times \bm{H}_{\text{eff}} - \alpha\, \bm{m}\times (\bm{m}\times \bm{H}_{\text{eff}})
	\end{equation}
	where the unknown magnetization field $\bm{m}=\bm{m}(x,t)$ takes values on the unit sphere~$\mathbb{S}^2$, $\alpha>0$ is a dimensionless 
	damping parameter, and the effective magnetic field $\bm{H}_{\text{eff}}$  depends on the unknown $\bm{m}$. 
	The  Landau--Lifshitz equation \eqref{LLG1}
	can be equivalently written in the  Landau--Lifshitz--Gilbert form 
	\begin{equation}
	\label{LLG2}
	\alpha\, \partial_t\bm{m} +\bm{m}\times \partial_t\bm{m}=(1+\alpha^2)\big [\bm{H}_{\text{eff}}-\big (\bm{m}\cdot \bm{H}_{\text{eff}}\big )\bm{m}\big ].
	\end{equation}
	Indeed, in view of the vector identity $\bm{a}\times (\bm{b}\times \bm{c})=(\bm{a}\cdot \bm{c})\bm{b}-(\bm{a}\cdot \bm{b})\bm{c},$
	for $\bm{a}, \bm{b}, \bm{c}\in \R^3,$ we have $-\bm{m}\times \big ( \bm{m}\times \bm{H}_{\text{eff}}\big )=
	\bm{H}_{\text{eff}}-\big (\bm{m}\cdot \bm{H}_{\text{eff}}\big )\bm{m},$
	and taking the vector product of \eqref{LLG1} with $\bm{m}$ and adding $\alpha$ times \eqref{LLG1} then yields  \eqref{LLG2}.
	%

	Since $\bm{m}\times \bm{a}$ is orthogonal to $\bm{m},$ for any  $\bm{a}\in \R^3,$ 
	it is obvious from \eqref{LLG1} that $\partial_t\bm{m}$ is orthogonal to $\bm{m}$:
	$\bm{m}\cdot \partial_t\bm{m}=0;$ we infer that
	the Euclidean norm satisfies $|\bm{m}(x,t)|=1$ for all $x$ and for all $t$, provided this is satisfied for the initial data.

	The term in square brackets on the right-hand side in \eqref{LLG2} can be rewritten as $\P(\bm{m}) \bm{H}_{\text{eff}}$,
	where (with $\Id$ the $3\times 3$ unit matrix)
	\[
	\P(\bm{m})= \Id - \bm{m}\bm{m}^T
	\] 
	is the orthogonal projection onto the tangent plane to the unit sphere $\mathbb{S}^2$ at $\bm{m}$.
	
	In this paper we consider the situation 
	\begin{equation} \label{Heff}
	\bm{H}_{\text{eff}}=\frac1{1+\alpha^2}\bigl( \varDelta \bm{m}+ \bm{H}\bigr), 
	\end{equation}
	where  $\bm{H}=\bm{H}(x,t)$ is a given external magnetic field. The factor $1/(1+\alpha^2)$ is chosen for convenience of presentation, 
	but is inessential for the theory; it can be replaced by any positive constant factor.
	
	With this choice of $\bm{H}_{\text{eff}}$, we arrive at the Landau--Lifshitz--Gilbert (LLG) equation 
	in the form
	\begin{equation}
	\label{llg-projection}
	\alpha\, \partial_t\bm{m} + \bm{m}\times \partial_t\bm{m}=\P(\bm{m})(  \varDelta \bm{m}+\bm{H}).
	\end{equation}
	We consider this equation as an initial-boundary value problem on a bounded domain $ \varOmega\subset \R^3$  and a time interval 
	$0\leqslant t \leqslant  \bar t$, with homogeneous Neumann boundary conditions and initial data $\bm{m}_0$ taking values on the unit 
	sphere, i.e., the Euclidean norm $|\bm{m}_0(x)|$ equals $1$ for all $x\in \varOmega$.
	
	We consider the following
	weak formulation, first proposed by  Alouges \cite{Al08,AlJ06}: Find the solution $\bm{m}: \varOmega\times[0,\bar t\,]\to \mathbb{S}^2$ 
	with $\m(\cdot,0)=\m_0$ by determining, at $\m(t)\in H^1(\Om)^3$, the time derivative $\partial_t \bm{m}$ (omitting here and in 
	the following the argument $t$) as that function  in the tangent space 
	\[
	\T(\bm{m}) :=\big\{ \bm\varphi\in L^2(\varOmega)^3 \,:\, \m\cdot\bm\varphi = 0\ \text{ a.e.} \big\} =
	\big\{ \bm\varphi\in L^2(\varOmega)^3 \,:\, \P(\m)\bm{\varphi} = \bm{\varphi} \}
	\] 
	that satisfies, for all $\bm{\varphi}\in \T(\bm{m})\cap H^1(\varOmega)^3$,
	\begin{equation}
	\label{weak}
	\begin{aligned}
	\alpha \bigl( \partial_t\bm{m},\bm\varphi\bigr) + \bigl(\bm{m}\times \partial_t\bm{m}, \bm{\varphi}\bigr) 
	+  \bigl( \nabla \bm{m}, \nabla \bm{\varphi} \bigr) = \bigl( \bm{H},\bm{\varphi} \bigr), \\
	\end{aligned}
	\end{equation}
	where the brackets $(\cdot,\cdot)$ denote the $L^2$ inner product over the domain $\varOmega$.
	The numerical methods studied in this paper are based on this weak formulation.
	%
	%
	
	\subsection{Previous work}
	There is a rich literature on numerical methods for Landau--Lifshitz(--Gilbert) equations; for the numerical literature 
	up to $2007$ see  the review by Cimr\'ak~\cite{Cim08}.
	
	Alouges \& Jaisson \cite{Al08,AlJ06} propose linear finite element discretizations in space and linearly implicit backward Euler in time for 
	the LLG equation in the weak formulation \eqref{weak} and prove convergence \emph{without} rates  towards nonsmooth weak solutions, 
	using a discrete energy inequality and compactness arguments.  Convergence of this type was previously 
	shown by Bartels \& Prohl \cite{BP} for fully implicit methods that are based on a different formulation of the Landau--Lifshitz equation \eqref{LLG1}. 
	In \cite{AKST14}, convergence without rates towards weak solutions is shown for a method that is (formally) of   ``almost" 
	order~$2$ in time, based on the midpoint rule, for the LLG equation with an effective magnetic field of a more general type than \eqref{Heff}.
	
	In a complementary line of research, convergence \emph{with} rates has been studied under sufficiently strong regularity assumptions, 
	which can, however, not be guaranteed over a given time interval, since solutions of the LLG equation may develop singularities.
	A first-order error bound for a linearly implicit  time discretization of the Landau--Lifshitz equation \eqref{LLG1} was proved by Cimr\'ak  \cite{Cim05}.
	Optimal-order error bounds for  linearly implicit  time discretizations based on the backward Euler and Crank--Nicolson 
	methods combined with finite element full discretizations  for  a different version of the Landau--Lifshitz equation \eqref{LLG1} 
	were obtained under sufficient regularity assumptions by Gao \cite{Gao14} and An \cite{RongAn16}, respectively. 
	In contrast to \cite{Al08,AlJ06,AKST14,BP},  these methods do not satisfy an energy inequality irrespective of the solution regularity.
	
	Numerical discretizations for the coupled system of the LLG equation \eqref{weak} with the eddy current approximation 
	of the Maxwell equations are studied by Feischl \& Tran \cite{FeiT17}, with first-order error bounds in space and time under 
	sufficient regularity assumptions. This also yields the first result of first-order convergence of the method of Alouges \& Jaisson \cite{Al08,AlJ06}.
	%
	%
	
	There are several methods for the LLG equations that are of formal order $2$ in time (though only of order $1$ in space), 
	e.g.,~\cite{PRS18,alouges14,DiFPPRS17}, but none of them comes with an error analysis. 
	Fully implicit BDF time discretizations for LLG equations have been used successfully 
	in the computational physics literature \cite{SueTSetal02}, though without giving any error analysis. 
	
	To the authors' knowledge, the second-order linearly implicit method proposed and studied here is thus the first 
	numerical method for the LLG (or LL) equation that has rigorous a priori error estimates of order $2$ in both space and time under high regularity assumptions and that 
	satisfies a discrete energy inequality irrespective of regularity. 

	We conclude this brief survey of the literature with a remark:
	The existing convergence results either give convergence of a subsequence without rates to a 
	weak solution (without imposing strong regularity assumptions), or they show convergence with rates towards sufficiently regular 
	solutions (as we do here). Both approaches yield insight into the numerical methods and have their merits, and they 
	complement each other. Clearly, neither approach is fully satisfactory, because convergence without rates 
	of some subsequence is nothing to observe in actual computations, and on the other hand 
	high regularity is at best provable for close to constant initial conditions~\cite{FeiT17b} or over short time intervals.
	We regard the situation as analogous to the development of numerical methods and their analysis in other fields such 
	as nonlinear hyperbolic conservation laws: second-order methods are highly popular 
	in that field, even though they can only be shown to converge with very low order ($1/2$ or less or only without rates) for available regularity properties; 
	see, e.g., \cite[Chapter~3]{Kr97}.
	Nevertheless, second-order methods are favored over first-order methods in many applications, especially if they enjoy some 
	qualitative properties that give them robustness in non-regular situations. 
	A similar situation occurs with the LLG equation, where the most important qualitative property appears to be the energy inequality.

	\subsection{Outline}
	
	In Section \ref{Se:LLG-discr} we describe the numerical methods studied in this paper. They use time discretization by linearly implicit 
	BDF methods of orders up to $5$ and space discretization by finite elements of arbitrary polynomial degree in a numerical 
	scheme that is based on the weak formulation \eqref{weak}, with an approximate tangent space that enforces the 
	orthogonality constraint approximately in an $L^2$-projected sense.

	In Section \ref{Se:main-res} we state our main results:\\
	$\bullet\ $  For the full discretization of \eqref{weak} by linearly implicit BDF methods of orders $1$ and $2$ and finite element 
	methods of arbitrary polynomial degree we give optimal-order error bounds in the $H^1$ norm,
	under very mild mesh conditions,  in the case of sufficiently 
	regular solutions (Theorem~\ref{theorem:err-bdf-full - BDF 1 and 2}).
	For these methods we also show a discrete energy inequality that requires only very weak regularity assumptions on the data 
	(Proposition~\ref{prop:energy}). This discrete energy inequality is of the same type as the one used in \cite{AlJ06,BP} 
	for proving convergence without rates to a weak solution.
	\\
	$\bullet\ $ For the linearly implicit BDF methods of orders $3$ to $5$ and finite element methods with polynomial degree at least~$2$, 
	we have
	optimal-order error bounds in the $H^1$ norm only if the damping parameter $\alpha$ is larger than some positive threshold, 
	which depends on the order of the BDF method (Theorem~\ref{theorem:err-bdf-full - BDF 3+}). Moreover, a stronger 
	(but still mild) CFL condition $\tau\leqslant c h$ is required. A discrete energy inequality under very weak regularity conditions 
	is not available for the BDF methods of orders $3$ to~$5$, in contrast to the A-stable BDF methods of orders $1$ and $2$.
	
	In Section \ref{Se:cont-perturb} we prove a perturbation result for the continuous problem by energy techniques, as a preparation for the proofs 
	of our error bounds for the discretization.
	
	In Section~\ref{section:discrete orthogonal projection}  
	we study properties of the $L^2$-orthogonal projection onto the discrete tangent space,
	which are needed to ensure consistency of the full order and  stability of the space discretization with the higher-order 
	discrete tangent space.  
	
	In Section~\ref{Se:full-discr} we study consistency properties of the methods and present the error equation.
	
	In Sections~\ref{Se:orders 1 and 2}  and~\ref{Se:orders 3 to 5}  we prove Theorems~\ref{theorem:err-bdf-full - BDF 1 and 2} 
	and~\ref{theorem:err-bdf-full - BDF 3+}, respectively.
	The higher-order convergence proofs are separated into consistency (Section~\ref{Se:full-discr}) and stability estimates.
	The stability proofs use the technique of energy estimates, in an unusual version where the error equation is tested with 
	a projection of the discrete time derivative of the error onto the discrete tangent space. These proofs are different for the 
	A-stable BDF methods of orders $1$ and $2$ and for the BDF methods of orders $3$ to $5$. For the control of nonlinearities, 
	the stability proofs also require pointwise error bounds, which are obtained with the help of finite element inverse 
	inequalities from the $H^1$ error bounds of previous time steps.
	
	In Section \ref{Se:numer-exp} we illustrate our results by numerical experiments.
	
	In an Appendix we collect basic results on energy techniques for BDF methods that are needed for our stability proofs.
	
	\pagebreak[3]
	
	\section{Discretization of the LLG equation}\label{Se:LLG-discr}
	
	We now describe the time and space discretization that is proposed and studied in this paper.
	
	\subsection{Time discretization by linearly implicit BDF methods}\label{SSe:LLG-discr-time}
	We shall discretize the LLG equation \eqref{weak} in time by the linearly implicit $k$-step BDF methods, 
	$1 \leqslant k \leqslant 5$, described by the polynomials $\delta$ and $\gamma,$
	\[\delta(\zeta)=\sum_{ \ell=1}^k\frac 1 \ell (1-\zeta)^ \ell=\sum_{j=0}^k\delta_j\zeta^j,\quad
	\gamma (\zeta)=\frac 1  \zeta\big [1-(1-\zeta)^k\big ]=\sum_{i=0}^{k-1} \gamma_i\zeta^i.\]

	We let $t_n=n\tau,\ n=0,\dotsc,N,$ be a uniform partition of the interval $[0,\bar t\,]$ with time step $\tau=\bar t/N.$
	For the $k$-step method we require $k$ starting values $\bm{m}^i$ for $i=0,\dotsc,k-1$. For $n\geqslant k$, we determine the approximation 
	$\bm{m}^n$ to $\bm{m}(t_n)$ as follows. We first extrapolate the known values $\bm{m}^{n-k},\dotsc,\bm{m}^{n-1}$ to a preliminary 
	normalized  approximation $\widehat{\bm{m}}^n$ at $t_n$,
	\begin{equation}
	\label{shn1}
	\widehat{\bm{m}}^n:=  {\displaystyle{\sum_{j=0}^{k-1}\gamma_j\bm{m}^{n-j-1}}} \Big/
	{\Big |\displaystyle{\sum_{j=0}^{k-1}\gamma_j\bm{m}^{n-j-1}}\Big |}.
	\end{equation}
	To avoid potentially undefined quantities, we define $\widehat{\bm{m}}^n$ to be an arbitrary fixed unit vector if the denominator in the above formula is zero.
	
	The derivative approximation $\dot{\bm{m}}^n$ and the solution approximation $ {\bm{m}}^n$ are related by the backward difference formula
	\begin{equation}
	\label{shn1n}
	\dot{\bm{m}}^n= 
	\frac 1\tau \sum_{j=0}^k\delta_j\bm{m}^{n-j}, \ \text{i.e., }\
	{\bm{m}}^n = \Bigl( - \sum_{j=1}^k\delta_j\bm{m}^{n-j} + \tau \dot{\bm{m}}^n \Bigr)/\delta_0.
	\end{equation}
	We  determine ${\bm{m}}^n$
	by requiring that for all $\bm{\varphi}\in \T(\widehat{\bm{m}}^n) \cap H^1(\varOmega)^3$,
	\begin{equation}
	\label{BDF1}
	\begin{aligned}
	& \alpha \bigl( \dot{\bm{m}}^n,\bm{\varphi}\bigr)+ \bigl(\widehat{\bm{m}}^n\times \dot{\bm{m}}^n,\bm{\varphi}\bigr) 
	+ \bigl(\nabla {\bm{m}}^n,\nabla \bm{\varphi}\bigr) = \bigl( \bm{H}(t_n),\bm{\varphi} \bigr)  
	\\
	& \dot{\bm{m}}^n \in \T(\widehat{\bm{m}}^n), \ \text{ i.e., }\  \widehat{\bm{m}}^n \cdot \dot{\bm{m}}^n  =0.
	\end{aligned}
	\end{equation}
	Here we note that on inserting the formula in \eqref{shn1n} for $ {\bm{m}}^n $ in the third term of \eqref{BDF1}, we obtain a linear constrained 
	elliptic equation for $\dot{\bm{m}}^n  \in \T(\widehat{\bm{m}}^n)\cap H^1(\varOmega)^3$ of the form
	\[
	\alpha \bigl( \dot{\bm{m}}^n,\bm{\varphi}\bigr)+ \bigl(\widehat{\bm{m}}^n\times \dot{\bm{m}}^n,\bm{\varphi}\bigr) 
	+ \frac\tau{\delta_0}\bigl(\nabla \dot{\bm{m}}^n,\nabla \bm{\varphi}\bigr) = \bigl( \bm{f}^n,\bm{\varphi} \bigr) \quad\ \forall \bm{\varphi}\in 
	\T(\widehat{\bm{m}}^n)\cap H^1(\varOmega)^3,
	\]
	where $ \bm{f}^n$ consists of known terms. The bilinear form on the left-hand side is $H^1(\varOmega)^3$-{coercive} 
	on $\T(\widehat{\bm{m}}^n)\cap H^1(\varOmega)^3$, and hence the above linear equation has a unique solution 
	$\dot{\bm{m}}^n \in\T(\widehat{\bm{m}}^n)\cap H^1(\varOmega)^3$ by the Lax--Milgram lemma.
	Once this elliptic equation is solved for $\dot{\bm{m}}^n$, we obtain the approximation $ {\bm{m}}^n\in H^1(\varOmega)^3$ to $ {\bm{m}}(t_n)$
	from the second formula in \eqref{shn1n}.

	%
	
	\subsection{Full discretization by BDF and higher-order finite elements}
	\label{SSe:LLG-discr-full} 
	For a family of regular and quasi-uniform finite element triangulations of $\Om$ with maximum meshwidth 
	$h>0$  we form the Lagrange finite element spaces $V_h \subset H^1(\Om)$ with piecewise 
	polynomials of degree 
	$r\geqslant 1$. 	
	We denote the $L^2$-orthogonal projections onto the finite element space by $\PPh\colon L^2(\Om)\to V_{h}$ and
	$\PPPh=\Id\otimes\PPh\colon L^2(\Om)^3 \to V_h^3$.
	With a function $\m\in H^1(\Om)^3$ that vanishes nowhere on $\varOmega$, we associate the 
	discrete tangent space
	\begin{equation}\label{discrete tangent space}
	\begin{aligned} 
	\T_h(\m) &=  \{\bm\varphi_h \in V_h^3:\, (\m\cdot\bm\varphi_h,v_h) = 0\ \ \forall\: v_h\in V_h\}  
	\\[1mm]
	&=\{\bm\varphi_h \in V_h^3\,:\,\PPh (\m\cdot\bm\varphi_h) = 0\}.
	\end{aligned}
	\end{equation}
	This space is different from the discrete tangent space used in \cite{Al08,AlJ06}, where the orthogonality constraint 
	$\m\cdot\bm\varphi_h=0$ is required to hold pointwise at the finite element nodes. Here, the constraint is enforced weakly 
	on the finite element space, as is done in various saddle point problems for partial differential equations, for example for 
	the divergence-free constraint in the Stokes problem \cite{BreF,GirR}. In contrast to that example, here the bilinear form 
	associated with the linear constraint, i.e.,
	$b(\m; \bm\varphi_h,v_h) = (\m\cdot\bm\varphi_h,v_h)$, depends on the state $\m$. This dependence substantially affects both 
	the implementation and the error analysis.

	Following  the general approach of \cite{Al08,AlJ06} with this modified discrete tangent space, we discretize \eqref{weak} in space by
	determining the time derivative $\partial_t\bm{m}_h(t) \in \T_h(\bm{m}_h(t))$ such that (omitting the argument $t$)
	\begin{equation}
	\label{weak-h}
	\alpha \bigl( \partial_t\bm{m}_h,\bm\varphi_h\bigr) + \bigl(\bm{m}_h\times \partial_t\bm{m}_h, \bm{\varphi}_h\bigr) 
	+  \bigl( \nabla \bm{m}_h, \nabla \bm{\varphi}_h \bigr) = \bigl( \bm{H},\bm{\varphi}_h \bigr) 
	\quad \forall \bm{\varphi}_h\in \T_h(\bm{m}_h),
	\end{equation}
	where the brackets $(\cdot,\cdot)$ denote again the $L^2$ inner product over the domain $\varOmega$.
	
	The full discretization with the linearly implicit BDF method is then readily obtained from \eqref{BDF1}: determine
	$\dot{\bm{m}}_h^n \in \T_h(\widehat{\bm{m}}^n_h)$ such that
	\begin{equation}
	\label{BDF1-h}
	\alpha \bigl( \dot{\bm{m}}_h^n,\bm{\varphi}_h\bigr)+ \bigl(\widehat{\bm{m}}_h^n\times \dot{\bm{m}}_h^n,\bm{\varphi}_h\bigr) 
	+ \bigl(\nabla {\bm{m}}_h^n,\nabla \bm{\varphi}_h\bigr) = \bigl( \bm{H}^n,\bm{\varphi}_h \bigr)  
	\quad \forall \bm{\varphi}_h\in \T_h(\widehat{\bm{m}}_h^n),
	\end{equation}
	where $\widehat{\bm{m}}_h^n$ and $ \dot{\bm{m}}_h^n$ are related to $\m_h^{n-j}$ for $j=0,\dotsc,k$ in the same 
	way as in \eqref{shn1} and \eqref{shn1n} above  
	with ${\bm{m}}_h^{n-j}$ in place of $\m^{n-j}$, viz.,
	\begin{equation}\label{mnh-hat-dot}
	\dot{\bm{m}}^n_h= 
	\frac 1\tau \sum_{j=0}^k\delta_j\bm{m}^{n-j}_h, 
	\qquad 
	\widehat{\bm{m}}^n_h=  {\displaystyle{\sum_{j=0}^{k-1}\gamma_j\bm{m}^{n-j-1}_h}} \Big/
	{\Big |\displaystyle{\sum_{j=0}^{k-1}\gamma_j\bm{m}^{n-j-1}_h}\Big |}.
	\end{equation}
	To avoid potentially undefined quantities, we define $\widehat{\bm{m}}^n_h$ to be an arbitrary fixed unit vector if the denominator 
	in the above formula is zero. (We will, however, show that this does not occur in the situation of sufficient regularity.)

	To implement the discrete tangent space $\T_h(\widehat{\bm{m}}^n_h)$, there are at least two options: using the constraints 
	$\PPh(\m\cdot\bm\varphi_h)=0$ or 
	constructing a local basis of $T_h(\m)$.
	
	(a) \emph{Constraints}:\/ Let $\phi_i$ for $i=1,\dotsc, N:=\text{dim} V_h$ denote the nodal basis of $V_h$ and denote the basis functions of $V_h^3$ by 
	$\bophi_{\bi}=\be_k \otimes \phi_i$ for $\bi=(i,k)$, where $\be_k$ for $k=1,2,3$ are the standard unit vectors of $\R^3$. We 
	denote by $M$ and $A$ the usual mass and stiffness matrices, respectively,
	with entries $m_{ij}=(\phi_i,\phi_j)_{L^2(\Om)}$ and $a_{ij}=(\nabla \phi_i,\nabla\phi_j)_{L^2(\Om)^3}$.
	We further introduce the sparse skew-symmetric matrix $S^n=(s_{\bi,\bj}^n) \in \R^{3N\times 3N}$ with entries
	$s_{\bi,\bj}^n= (\widehat{\bm{m}}^n_h \times \phi_{\bi},\phi_{\bj})_{L^2(\Om)^3}$ and the sparse
	constraint matrix $C^n = (c_{\bi,j}^n) \in \R^{3N\times N}$ by 
	$c_{\bi,j}^n=(\widehat{\bm{m}}^n_h\cdot \phi_{\bi},\phi_j)_{L^2(\Om)}$. Finally, we denote the matrix of the unconstrained time-discrete problem as
	\[K^n = \alpha \mathbf{I} \otimes M  +\frac\tau{\delta_0} \mathbf{I} \otimes  A + S^n.\]
	%
	Let $\dot m^n \in \R^{3N}$ denote the nodal vector of $\dot{\bm{m}}_h^n \in \T_h(\widehat{\bm{m}}_h^n)$. In this setting, 
	\eqref{BDF1-h} yields a system of linear equations of saddle point type
	\begin{equation*}
	\begin{alignedat}{2}
	&K^n \dot m^n + (C^n)^T \lambda^n &&{}= f^n,\\
	& C^n \dot m^n  &&{}= 0,
	\end{alignedat}
	\end{equation*}
	where $ \lambda^n\in \R^N$ is the unknown vector of Lagrange multipliers and $f^n\in\R^{3N}$ is a known right-hand side.
	
	(b) \emph{Local basis}:\/  It is possible to compute a local basis of $T_h(\m)$ by solving small local problems. 
	To see that, let $\omega\subset \Om$ denote a collection of elements of the mesh and let $\overline\omega\supset \omega$ 
	denote the same set plus the layer of elements touching $\omega$  (the patch of $\omega$). A sufficient (and necessary) 
	condition for $\bm\varphi_h\in V_h^3$ with $\supp(\bm\varphi_h)\subseteq \omega$ to belong to $T_h(\m)$ is 
	\begin{equation}\label{eq:locprob}
	(\m\cdot \bm\varphi_h,\psi_h) =0 \quad\text{for all } \psi_h\in V_h\text{ with } {\rm supp}(\psi_h)\subseteq \overline\omega.
	\end{equation}
	If we denote by $\#\omega$ the number of generalized hat functions of $V_h$ supported in $\omega$,
	the space of functions in $V_h^3$ with support in $\omega$ is $3\#\omega$-dimensional. On the other hand, 
	the space of test functions 
	in~\eqref{eq:locprob} is $\#\overline\omega$-dimensional. We may choose $\omega$ sufficiently large 
	(depending only on shape regularity) such that $3\#\omega>\#\overline\omega$ and hence~\eqref{eq:locprob} 
	has at least one solution which is then a local basis function of $T_h(\m)$. Choosing different $\omega$ to cover 
	$\Om$ yields a full basis of $T_h(\m)$. 
	
	Let us denote the so obtained basis of $\T_h(\widehat{\bm{m}}_h^n)$ by $(\psi_\ell^n)$, given  via
	$\psi_\ell^n = \sum_{\bi} \phi_{\bi} b_{ \bi \ell}^n $, and the sparse basis matrix by $B^n=(b_{ \bi \ell}^n)$. 
	Then, the nodal vector $\dot m^n= B^n x^n$ is obtained by solving
	the linear system 
	\[ (B^n)^T K^n B^n x^n = (B^n)^T f^n.\]
	An advantage of this approach is that the dimension is roughly halved compared to the formulation with constraints. 
	However, the efficiency of one approach versus the other depends heavily on the numerical linear algebra used. 
	Such comparisons are outside the scope of this paper.

	\begin{remark}
		The algorithm described above does not enforce the norm constraint $|\m|=1$ at the nodes.
		The user might add a normalization step in the definition of $\bm{m}^n$ in~\eqref{shn1n}. However, 
		here we do not consider this normalized variant of the method, whose convergence properties are not obvious to derive. 
	\end{remark}
	
	\begin{remark} \label{rem:implementation}
		Differently to~\cite{Al08}, we do not use the pointwise discrete tangent space 
		\begin{align*}
		T_h^{\rm pw} (\m) &{}=   \{\bm\varphi_h \in V_h^3:\, \m\cdot\bm\varphi = 0\ \text{ in every node}\}\\
		&{}=\{\bm\varphi_h \in V_h^3\,:\, I_h (\m\cdot\bm\varphi_h) = 0\} = \mathbf{I}_h \P(\m) V_h^3,
		\end{align*}
		where $I_h:C(\bar\varOmega)\to V_h$ denotes finite element interpolation and $\mathbf{I}_h= \Id \otimes I_h:C(\bar\varOmega)^3\to V_h^3$.
		It is already reported in~\cite[Section~4]{Al08} that an improvement of the order with higher-degree finite elements could not be 
		observed in numerical experiments when using the pointwise tangent spaces in the discretization \eqref{weak-h}. 
		Our analysis shows a lack of consistency of optimal order in the discretization with $T_h^{\rm pw} (\m)$, 
		which originates from the fact that $\mathbf{I}_h \P(\m)$ is not self-adjoint.
		The order reduction can, 
		however, be cured by adding a correction term: in the $n$th time step, determine 
		$\dot{\bm{m}}_h^n\in \T_h^{\rm pw}(\widehat{\bm{m}}_h^n)$
		such that for all 
		$\bm{\varphi}_h\in \T_h^{\rm pw}(\widehat{\bm{m}}_h^n)$,
		\begin{equation}
		\label{BDF1-h-pw}
		\begin{aligned}
		\alpha \bigl( \dot{\bm{m}}_h^n,\bm{\varphi}_h\bigr)&{}+ \bigl(\widehat{\bm{m}}_h^n\times \dot{\bm{m}}_h^n,\bm{\varphi}_h\bigr) 
		+ \bigl(\nabla {\bm{m}}_h^n,\nabla \bm{\varphi}_h\bigr) \\
		&{}- \bigl(\nabla \widehat{\bm{m}}_h^n,\nabla (\mathbf{I}-\mathbf{P}(\widehat{\bm{m}}_h^n))\bm{\varphi}_h\bigr)
		= \bigl( \mathbf{P}(\widehat{\bm{m}}_h^n)\bm{H}(t_n),\bm{\varphi}_h \bigr)  ,
		\end{aligned}
		\end{equation}
		with notation $\widehat{\bm{m}}_h^n$ and $ \dot{\bm{m}}_h^n$ as in \eqref{mnh-hat-dot}.
		With the techniques of the present paper, it can be shown that like \eqref{BDF1-h}, 
		also this discretization converges with optimal order in the $H^1$ norm under sufficient regularity conditions. 
		Since this paper is already rather long, we do not include 
		the proof of this result.
		In contrast to \eqref{BDF1-h} for the first- and second-order BDF methods, the method \eqref{BDF1-h-pw} does not admit an $h$- 
		and $\tau$-independent bound of the energy that is irrespective of the smoothness of the solution.  
	\end{remark}

	\section{Main results}\label{Se:main-res}
	
	\subsection{Error bound and energy inequality for BDF of orders 1 and 2}\label{SSe:main-res}

	For the full discretization with first- and second-order BDF methods and finite elements of arbitrary polynomial degree $r\geqslant 1$ 
	we will prove the following optimal-order error bound in Sections \ref{section:discrete orthogonal projection}  to \ref{Se:orders 1 and 2}. 
	
	\begin{theorem}[Error bound for orders $k=1, 2$]
		\label{theorem:err-bdf-full - BDF 1 and 2}
		Consider the full discretization \eqref{BDF1-h}
		of the LLG equation \eqref{llg-projection} by
		the  linearly implicit $k$-step BDF time discretization for $ k \leqslant 2$ 
		and finite elements of polynomial degree $r \geqslant 1$ from a family of regular and 
		quasi-uniform  triangulations of $\Om$.
		Suppose that the solution $\m$ of the LLG equation is sufficiently regular. Then, there exist $\bar \tau>0$ and $\bar h>0$ such that
		for numerical solutions obtained with step sizes $\tau\leqslant\bar\tau$ and meshwidths $h\leqslant\bar h$,
		which are restricted by the very mild CFL-type condition
		\[\tau^k \leqslant \bar c h^{1/2} \]
		%
		with a sufficiently small constant $\bar c$ $($independent of $h$ and $\tau$$)$, the errors
		are bounded by
		\begin{equation}
		\label{err-bdf12}
		\| \mh^n - \m(t_n) \|_{H^1(\Om)^3} \leqslant C (\tau^k+h^r) \quad \text{ for }\ t_n=n\tau\leqslant \bar t,
		\end{equation}
		where $C$ is independent of $h, \tau$ and $n$ $($but depends on $\alpha$ and exponentially on $\bar t$\,$)$, 
		provided that the errors of the starting values also satisfy such a bound.
	\end{theorem}
	
	The precise regularity requirements are as follows: 
	\begin{equation}
	\label{eq:assumptions on regularity}
	\begin{aligned}
	&\m \in  C^{k+1}([0,\bar t\,], L^\infty(\Om)^3) \cap C^1([0,\bar t\,], W^{r+1,\infty}(\Om)^3), \\
	&\varDelta \bm{m} +\bm{H} \in C([0,\bar t\,], W^{r+1,\infty}(\Om)^3) .
	\end{aligned} 
	\end{equation}
	
	%
	\begin{remark}[\emph{Discrepancy from normality}]\label{rem:normality}
		Since $\m(x,t_n)$ are unit vectors, an immediate consequence of the error estimate \eqref{err-bdf12} is that
		\begin{equation} \label{unit-discr}
		\| 1- |\m^n_h| \|_{L^2(\Om)} \leqslant C (\tau^k+h^r) \quad \text{ for }\ t_n=n\tau\leqslant \bar t,
		\end{equation}
		with a constant $C$ independent of $n, \tau$ and $h$. The proof of Theorem~\ref{theorem:err-bdf-full - BDF 1 and 2} 
		also shows that the denominator in the definition of the normalized extrapolated value ${\widehat\m}^n_h$ satisfies
		\[\Bigl\| 1- \bigl|\sum_{j=0}^{k-1}\gamma_j \m^{n-j-1}_h\bigr| \Bigr\|_{L^\infty(\Om)} \leqslant 
		C h^{-1/2} (\tau^k+h^r) \leqslant \tfrac12 \quad \text{ for }\ t_n=n\tau\leqslant \bar t,\]
		which in particular ensures that ${\widehat\m}^n_h$ is unambiguously defined.
	\end{remark}
	
	\vspace*{0.1cm}
	Testing with $\bphi=\partial_t\m\in T(\m)$ in \eqref{weak},  we obtain (only formally, if $\partial_t\m$ is not in $H^1(\varOmega)^3$)
	\[\alpha (\partial_t\m,\partial_t\m) + (\nb\m,\partial_t \nb \m) = (\bm{H},\partial_t\m),\]
	which, by integration in time and the Cauchy--Schwarz and Young inequalities, implies the energy inequality
	\[ \| \nb\m(t) \|_{L^2}^2 + \tfrac12 \alpha \int_0^t \| \partial_t\m(s) \|_{L^2}^2 \,\d s  \leqslant  \| \nb\m(0) \|_{L^2}^2 
	+ \frac1{2\alpha}\int_0^t \| \bm{H}(s) \|_{L^2}^2 \,\d s.\]

	Similarly, we test with $\bphi_h={\dot\m}^n_h\in \T_h({\widehat\m}^n_h)$ in \eqref{BDF1-h}. Then we can prove
	the following discrete energy inequality, which holds under very weak regularity assumptions on the data.
	
	\begin{proposition}[Energy inequality for orders $k=1, 2$]
		\label{prop:energy}
		Consider the full discretization \eqref{BDF1-h}
		of the LLG equation \eqref{llg-projection} by the linearly implicit $k$-step BDF time discretization for $ k \leqslant 2$ 
		and finite elements of polynomial degree $r \geqslant 1$. Then, the numerical solution satisfies the following discrete 
		energy inequality\emph{:}  for $n\geqslant k$ with $n\tau \leqslant \bar t$,
		\[\gamma_k^- \| \nabla \mh^n \|^2_{L^2} + \tfrac12 \alpha\tau \sum_{j=k}^n \| {\dot \m}_h^j \|^2_{L^2} \leqslant 
		\gamma_k^+ \sum_{i=0}^{k-1} \| \nabla \mh^i \|^2_{L^2} +
		\frac\tau{2\alpha} \sum_{j=k}^n \| \bm{H}(t_j) \|^2_{L^2},\]
		where $\gamma_1^\pm=1$ and $\gamma_2^\pm=(3\pm 2\sqrt2)/4$.
	\end{proposition}
	
	This energy inequality is an important robustness indicator of the numerical method.
	In \cite{AlJ06,BP}, such energy inequalitys are used to prove convergence without rates (for a subsequence $\tau_n\to 0$ and $h_n\to 0$) 
	to a weak solution of the LLG equation for the numerical schemes considered there (which have $\gamma^\pm=1$, but this 
	is inessential in the proofs).
	
	As the proof of Proposition~\ref{prop:energy} is short, we give it here.
	
	\begin{proof} The proof relies on the A-stability of the first- and second-order BDF methods via Dahlquist's G-stability theory 
		as expressed in Lemma~\ref{lemma:Dahlquist} of the Appendix, used with $\delta(\zeta)=\sum_{\ell=1}^k (1-\zeta)^\ell / \ell$ 
		and $\mu(\zeta)=1$. The positive definite symmetric  matrices $G=(g_{ij})_{i,j=1}^k$
		are known to be $G=1$ for $k=1$ and (see \cite[p.\,309]{HW})
		\[G=\frac14\begin{pmatrix} 1 & -2 \\
		-2 & 5
		\end{pmatrix} 
		\quad\text{ for }\ k=2,\]
		which has the eigenvalues $\gamma^\pm=(3\pm 2\sqrt2)/4$.
		
		We test with $\bm{\varphi}_h= {\dot \m}_h^n \in T_h(\widehat\m_h^n)$ in \eqref{BDF1-h} and note 
		$\bigl(\widehat{\bm{m}}_h^n\times \dot{\bm{m}}_h^n,{\dot \m}_h^n\bigr)=0$, so that
		\[\alpha \| {\dot \m}_h^n \|^2_{L^2} + ( \nabla \mh^n, \nabla {\dot \m}_h^n) = (\bm{H}^n, {\dot \m}_h^n).\]
		The right-hand side is bounded by
		\[(\bm{H}^n, {\dot \m}_h^n) \leqslant \frac\alpha 2 \| {\dot \m}_h^n \|^2_{L^2} + \frac1{2\alpha} \| \bm{H}^n \|^2_{L^2}.\] 
		Recalling the definition of ${\dot \m}_h^n$, we have by Lemma~\ref{lemma:Dahlquist}
		\[( \nabla \mh^n, \nabla {\dot \m}_h^n) \geqslant \frac1\tau \sum_{i,j=1}^k g_{ij}  ( \nabla \mh^{n-i+1}, \nabla {\m}_h^{n-j+1}) -
		\frac1\tau \sum_{i,j=1}^k g_{ij}  ( \nabla \mh^{n-i}, \nabla {\m}_h^{n-j}).\]
		We fix $\bar n$ with $k\leqslant \bar n \leqslant \bar t/\tau$ and sum from $n=k$ to $\bar n$ to obtain
		\begin{align*}
		&\sum_{i,j=1}^k g_{ij}  ( \nabla \mh^{\bar n-i+1}, \nabla {\m}_h^{\bar n-j+1}) + 
		\tfrac12 \alpha\tau \sum_{n=k}^{\bar n} \| {\dot \m}_h^n \|^2_{L^2} 
		\\
		&\qquad \leqslant
		\sum_{i,j=1}^k g_{ij}  ( \nabla \mh^{k-i}, \nabla {\m}_h^{k-j}) +
		\frac\tau{2\alpha} \sum_{n=k}^{\bar n} \| \bm{H}^n \|^2_{L^2}.
		\end{align*}
		Noting that 
		\begin{align*}
		\gamma^- \| \nabla \m_h^{\bar n} \|^2_{L^2}
		&\leqslant \sum_{i,j=1}^k g_{ij}  ( \nabla \mh^{\bar n-i+1}, \nabla {\m}_h^{\bar n-j+1}),
		\\
		%
		%
		\sum_{i,j=1}^k g_{ij}  ( \nabla \mh^{k-i}, \nabla {\m}_h^{k-j}) &\leqslant
		\gamma^+ \sum_{i=0}^{k-1} \| \nabla \mh^i \|^2_{L^2},
		\end{align*}
		we obtain the stated result.
	\end{proof}

	\subsection{Error bound for BDF of orders $3$ to $5$}
	
	For the BDF methods of orders $3$ to $5$ we prove the following result in Section~\ref{Se:orders 3 to 5}. Here we require a stronger, 
	but still moderate stepsize restriction in terms  of the meshwidth. More importantly, we must impose a positive lower bound on the 
	damping parameter $\alpha$ of \eqref{LLG1}.
	
	\begin{theorem} [Error bound for orders $k=3, 4, 5$]
		\label{theorem:err-bdf-full - BDF 3+}
		Consider the full discretization \eqref{BDF1-h}
		of the LLG equation \eqref{llg-projection} by
		the   linearly implicit $k$-step BDF time discretization for $3 \leqslant k \leqslant 5$ 
		and finite elements of polynomial degree $r \geqslant 2$ from a family of regular and 
		quasi-uniform triangulations of $\Om$.
		Suppose that the solution $\m$ of the LLG equation has the regularity \eqref{eq:assumptions on regularity}, 
		and that the damping parameter $\alpha$ satisfies
		\begin{equation}
		\label{alpha-k} 
		\begin{gathered}
		\alpha > \alpha_k \quad \text{with } \\
		\alpha_k= 0.0913, \ 0.4041, \ 4.4348 , \quad \text{for} \ k=3, 4 , 5 , \ \text{respectively} .
		\end{gathered}
		\end{equation}
		Then, for an arbitrary constant $\bar C>0$, there exist $\bar \tau>0$ and $\bar h>0$ such that
		for numerical solutions obtained with step sizes $\tau\leqslant\bar\tau$ and meshwidths $h\leqslant\bar h$ that are restricted by
		\begin{equation}\label{CFL}
		\tau \leqslant \bar C h,
		\end{equation}
		the errors are bounded by
		\[
		\| \m^n_h - \m(t_n) \|_{H^1(\Om)^3} \leqslant C (\tau^k+h^r)\quad \text{ for }\ t_n=n\tau\leqslant \bar t,
		\]
		where $C$ is independent of $h, \tau$ and $n$ $($but depends on $\alpha$ and exponentially on $\bar C\bar t$$)$, 
		provided that the errors of the starting values also satisfy such a bound.
	\end{theorem}
	
	
	Theorem~\ref{theorem:err-bdf-full - BDF 3+} limits the use of the BDF methods of orders higher than $2$ (and more severely 
	for orders higher than $3$) to applications with a large damping parameter~$\alpha$, such as cases described in \cite{GM09,TH14}. 
	We remark, however, that in many situations $\alpha$ is of magnitude $10^{-2}$ or even smaller \cite{BCEU14}. A very small 
	damping parameter $\alpha$  affects not only the methods considered here. To our knowledge, the error analysis of any numerical 
	method proposed in the literature breaks down as $\alpha\to 0$, as does the energy inequality.

	It is not surprising that a positive lower bound on $\alpha$ arises for the methods of orders $k\geqslant 3$, since they are 
	not A-stable and a lower bound on $\alpha$ is required also for the simplified linear problem $(\alpha+\i)\pa_t u =  \varDelta u$, 
	which arises from \eqref{llg-projection} by freezing $\m$ in the term $\m\times\pa_t\m$ and diagonalizing this skew-symmetric linear 
	operator (with eigenvalues $\pm \i$ and $0$) and by omitting the projection $\P(\m)$ on the right-hand side of \eqref{llg-projection}.
	
	The proof of Theorem~\ref{theorem:err-bdf-full - BDF 3+} uses a  variant of the Nevanlinna--Odeh multiplier technique \cite{NO}, 
	which is described in the Appendix for the convenience of the reader. 
	While for sufficiently large $\alpha$  we have an optimal-order 
	error bound in the case of a smooth solution, there is apparently no discrete energy inequality under weak regularity 
	assumptions similar to Proposition~\ref{prop:energy} for the BDF methods of orders $3$ to~$5$.
	
	As in Remark~\ref{rem:normality}, the error bounds also allow us to bound the discrepancy from normality.

	\section{A continuous perturbation result}\label{Se:cont-perturb}

	In this section we present a perturbation result for the continuous problem, because we will later transfer the arguments 
	of its proof to the discretizations to prove stability and convergence of the numerical methods.

	Let $\m(t)$ be a solution of \eqref{llg-projection} for $0\leqslant t \leqslant \bar t$, and let $\ms(t)$, also of unit length, solve the 
	same equation up to a defect $\bd(t)$ for $0\leqslant t \leqslant \bar t$:
	\begin{equation}
	\label{eq:perturbed problem}
	\begin{aligned}
	\alpha \pa_t \ms + \ms \times \pa_t \ms = {}&  \P(\ms) (\varDelta \ms +\bm H) +\bd \\
	= {}&  \P(\m) (\varDelta \ms +\bm H) + \br ,
	\end{aligned}
	\end{equation}
	with
	\[
	\br = -  \big(\P(\m) - \P(\ms)\big) (\varDelta \ms +  \bm H) + \bd.
	\]
	Then, $\ms$ also solves the perturbed weak formulation
	\[
	\alpha (\pa_t \ms,\bphi) + (\ms \times \pa_t \ms,\bphi) +  (\nb \ms,\nb \bphi) = (\br,\bphi)   \quad\ \forall \bphi \in \T(\m)
	\cap H^1(\Om)^3,
	\]
	and the error $\be = \m - \ms$ satisfies the error equation 
	\begin{equation}
	\label{eq:error equation weak}
	\begin{aligned}
	\alpha(\pa_t \be,\bphi) + (\be \times \pa_t \ms,\bphi) + (\m \times \pa_t \be,\bphi) +{}& (\nb \be,\nb \bphi) = - (\br,\bphi) \\
	& \forall \bphi \in \T(\m)\cap H^1(\Om)^3.
	\end{aligned}
	\end{equation}

	Before we turn to the perturbation result, we need  Lipschitz-type bounds for the orthogonal projection 
	$\P(\m)=\Id-\m \m^T$ applied to sufficiently regular functions. 
	\begin{lemma}
		\label{lemma:projection errors}
		The projection $\P(\cdot)$ satisfies the following estimates, for  functions $\m,\ms,\bv : \Om \to \R^3$, 
		where $\m$ and $\ms$ take values on the unit sphere and $\ms\in W^{1,\infty}(\Om)^3$\emph{:}
		\begin{align*}
		\| (\P(\m) - \P(\ms)) \bv \|_{L^2(\Om)^3} \leqslant {}& 2\, \|\bv\|_{L^\infty(\Om)^3} \|\m-\ms\|_{L^2(\Om)^3}, \\
		\big\| \nb \big((\P(\m) - \P(\ms)) \bv\big) \big\|_{L^2(\Om)^{3\times 3}} \leqslant {}& 2\,\| \ms\|_{W^{1,\infty}(\Om)^3} \|\bv\|_{W^{1,\infty}(\Om)^3} 
		\|\m-\ms\|_{L^2(\Om)^3} \\
		&{} \qquad\qquad\!+ 6 \, \|\bv\|_{L^\infty(\Om)^3} \|\nabla(\m-\ms)\|_{L^2(\Om)^{3\times 3}}.
		\end{align*}
	\end{lemma}
	\begin{proof}
		Setting $\be=\m-\ms$, we start by rewriting 
		\[
		(\P(\m) - \P(\ms)) \bv 
		=  - (\m \m^T - \ms \ms^T) \bv \\
		=  - ( \m \be^T + \be \ms^T) \bv .
		\]
		
		The first inequality then follows immediately by taking the $L^2$ norm of both sides of the above equality, 
		using the fact that $\m$ and $\ms$ are of unit length.
		The second inequality is proved similarly, using the product rule
		\begin{align*}
		\partial_i(\P(\m) - \P(\ms)) \bv 
		= & - \partial_i  ( \be\be^T +\ms \be^T + \be \ms^T) \bv \\
		= & - ( \partial_i \be \be^T  + \be \partial_i \be^T + \partial_i \ms \be^T + \ms \partial_i\be^T + \partial_i\be \ms^T +\be \partial_i\ms^T ) \bv 
		\\ 
		& +  ( \m \be^T + \be \ms^T) \partial_i \bv ,
		\end{align*}
		the $L^\infty$ bound of $\partial_i \ms$, and the fact that $\| \be \|_{L^\infty} \leqslant \| \m \|_{L^\infty} + \| \ms \|_{L^\infty} \leqslant 2$.
	\end{proof}
	
	We have the following perturbation result.
	\begin{lemma}
		\label{lemma:perturbation result}
		Let $\m\t$ and $\ms\t$ be solutions of unit length of \eqref{weak} and \eqref{eq:perturbed problem}, respectively, and suppose that,      
		for $0 \leqslant t \leqslant \bar t$, we have 
		\begin{equation}
		\label{eq:assumptions on m star}
		\begin{aligned}
		\| \ms\t\|_{W^{1,\infty}(\Om)^3} + \|\pa_t  \ms\t\|_{W^{1,\infty}(\Om)^3} &\leqslant R \\
		\andquad \|\varDelta \ms\t + {\bm H}(t)\|_{L^\infty(\Om)^3} &\leqslant K.
		\end{aligned}
		\end{equation}
		Then, the error $\be(t) = \m\t - \ms\t$ satisfies, for $0 \leqslant t \leqslant \bar t$,
		\begin{equation}\label{cont-pert-est}
		\|\be(t)\|_{H^1(\Om)^3}^2
		\leqslant  C \Bigl( \| \be(0) \|_{H^1(\Om)^3}^2 + \int_0^t \|\bd(s)\|_{L^2(\Om)^3}^2 \,\d s \Bigr),
		\end{equation}
		where the constant $C$ depends only on $\alpha,R,K$, and $\bar t$.
	\end{lemma}
	
	\begin{proof} Let us first assume that $\pa_t\m(t) \in H^1(\Om)^3$ for all $t$. Following \cite{FeiT17}, we test in
		the error equation \eqref{eq:error equation weak}  with  $\bphi= \P(\m) \pa_t\be \in \T(\m)$. By the following argument, 
		this test function is then indeed in $H^1(\Om)^3$ and can be viewed as a perturbation of~$\pa_t\be$:
		\begin{align*}
		\bphi = \P(\m) \pa_t\be = {}& \P(\m) \pa_t\m - \P(\m) \pa_t \ms \\
		= {}&  \P(\m) \pa_t\m  -  \P(\ms) \pa_t \ms - (\P(\m) - \P(\ms)) \pa_t \ms \\
		= {}& \pa_t\m - \pa_t \ms - (\P(\m) - \P(\ms)) \pa_t \ms ,
		\end{align*}
		and so we have
		\begin{equation}
		\label{eq:dot e perturbed}
		\bphi = \P(\m) \pa_t\be = \pa_t\be + \bq \qquad \textnormal{ with } \quad \bq = -(\P(\m) - \P(\ms)) \pa_t \ms .
		\end{equation}
		By Lemma~\ref{lemma:projection errors} and using \eqref{eq:assumptions on m star} we have
		\begin{equation}
		\label{eq:q bound}
		\|\bq\|_{L^2} \leqslant 2 R \|\be\|_{L^2} \andquad \|\nb \bq\|_{L^2} \leqslant C R \|\be\|_{H^1} .
		\end{equation}
		
		Testing the error equation \eqref{eq:error equation weak} with $\bphi = \pa_t\be + \bq$, we obtain
		\begin{align*}
		\alpha (\pa_t \be,\pa_t\be + \bq) + (\be \times \pa_t \ms,\pa_t\be + \bq) 
		{}& + (\m \times \pa_t \be,\pa_t\be + \bq)  \\
		{}& +  (\nb \be,\nb (\pa_t\be + \bq))  = - (\br,\pa_t\be + \bq) ,
		\end{align*}
		where, by \eqref{eq:perturbed problem} and Lemma~\ref{lemma:projection errors} with \eqref{eq:assumptions on m star}, $\br$ is bounded as
		\begin{equation}\label{r-est}
		\begin{aligned}
		\|\br\|_{L^2}
		\leqslant {}&  \|\big(\P(\m) - \P(\ms)\big)( \varDelta \ms+ \bm H)\|_{L^2} + \|\bd\|_{L^2} \\
		\leqslant {}& 2K \|\be\|_{L^2} + \|\bd\|_{L^2} .
		\end{aligned}
		\end{equation}
		By collecting terms, and using the fact that $(\m \times \pa_t \be,\pa_t\be)$ vanishes, we altogether obtain
		\begin{align*}
		\alpha\|\pa_t\be\|_{L^2}^2 +  \half \diff \|\nb \be\|_{L^2}^2 
		= &{} -\alpha (\pa_t \be, \bq) - (\be \times \pa_t \ms,\pa_t\be + \bq) - (\m \times \pa_t \be,\bq) \\
		&{} - (\nb \be,\nb \bq)  - (\br,\pa_t\be + \bq) .
		\end{align*}
		For the right-hand side, the Cauchy--Schwarz inequality and $\|\m\|_{L^\infty}=1$ yield
		\begin{align*}
		& \alpha\|\pa_t\be\|_{L^2}^2 +  \half \diff \|\nb \be\|_{L^2}^2 
		\leqslant  \alpha\|\pa_t \be\|_{L^2} \|\bq\|_{L^2} + R \|\be\|_{L^2} (\|\pa_t\be\|_{L^2} + \|\bq\|_{L^2})  \\
		& \quad {}+ \|\pa_t \be\|_{L^2} \|\bq\|_{L^2} + \|\nb \be\|_{L^2} \|\nb \bq\|_{L^2}  + \|\br\|_{L^2} (\|\pa_t\be\|_{L^2} + \|\bq\|_{L^2}) .
		\end{align*}
		Young's inequality and absorptions, together with the bounds in \eqref{eq:q bound} and \eqref{r-est}, yield
		\begin{equation*}
		\alpha \half \|\pa_t\be\|_{L^2}^2 +  \half \diff \|\nb \be\|_{L^2}^2 
		\leqslant  c \|\be\|_{H^1}^2 + c\|\bd\|_{L^2}^2 .
		\end{equation*}
		Here, we note  that
		\[
		\half \diff \| \be\|_{L^2}^2 = (\pa_t\be,\be) \leqslant \tfrac 12  \|\pa_t\be\|_{L^2}^2 + \tfrac12  \|\be\|_{L^2}^2,\quad \text{so that}\quad
		\|\pa_t\be\|_{L^2}^2 \geqslant \diff \| \be\|_{L^2}^2  - \|\be\|_{L^2}^2.
		\]
		Combining these inequalities and integrating in time, we obtain
		\[
		\|\be(t) \|_{H^1}^2 \leqslant  c\|\be(0) \|_{H^1}^2 + c\int_0^t \|\be(s)\|_{H^1}^2 \d s + c \int_0^t \|\bd(s)\|_{L^2}^2 \d s .
		\]
		By Gronwall's inequality, we then obtain the stated error bound.
		
		Finally, if $\pa_t\m(t)$ is not in $H^1(\Om)^3$ for some $t$, then a regularization and density argument, which we do not 
		present here, yields the result, since the error bound does not depend on the $H^1$ norm of $\pa_t \m$.
	\end{proof}

	\section{Orthogonal projection onto the discrete tangent space}
	\label{section:discrete orthogonal projection}
	
	For  consistency and stability of the full discretization, we need to study properties of the $L^2(\Om)$-orthogonal 
	projection onto the discrete tangent space $T_h(\m)$, which we denote by 
	\[\P_h(\m)\colon V_h^3\to T_h(\m). \]
	We do not have an explicit expression for this projection, but the properties stated in Lemmas~\ref{lem:gal} to \ref{lem:stab} 
	will be used for proving consistency and stability. We recall that we consider a quasi-uniform, shape-regular 
	family $\mathcal{T}_h$ of  triangulations with Lagrange finite elements of polynomial degree~$r$.
	
	The first lemma states that the projection $\P_h(\m)$ approximates the orthogonal projection $\P(\m)=\mathbf{I} - \m\m^T$ 
	onto the tangent space $T(\m)$ with optimal order. It will be used in the consistency error analysis of Section~\ref{Se:full-discr}.
	\begin{lemma}\label{lem:gal}
		For $\m\in W^{r+1,\infty}(\Om)^3$ with $|\m|=1$ almost everywhere we have 
		\begin{align*}
		\|(\P_h(\m)-\P(\m))\bv\|_{L^2(\Om)^3} \leqslant {}& C h^{r + 1} \, \| \bv \|_{H^{r+1}(\Om)^3}, \\
		\|(\P_h(\m)-\P(\m))\bv\|_{H^1(\Om)^3} \leqslant {}&  C h^r \, \| \bv \|_{H^{r+1}(\Om)^3} ,
		\end{align*}
		for all $\bv\in H^{r+1}(\Om)^3$,
		where $C$ depends on a bound of $\|\m\|_{ W^{r+1,\infty}(\Om)^3}$.
	\end{lemma}
	
	The second lemma states that the projection $\P_h(\m)$ has Lipschitz bounds of the same type 
	as those of the orthogonal projection $\P(\m)$ given in Lemma~\ref{lemma:projection errors}. 
	It will be used in the stability analysis of Sections~\ref{Se:orders 1 and 2}  and~\ref{Se:orders 3 to 5}.
	\begin{lemma}\label{lem:diff}
		Let $\m \in  W^{1,\infty}(\Om)^3$ and $\widetilde\m \in H^1(\Om)^3$ with $|\m|=|\widetilde\m|=1$ 
		almost everywhere and $\| \m \|_{W^{1,\infty}} \leqslant R$. 
		There exist  $C_R>0$ and $h_R>0$  such that for $h\leqslant h_R$, for all ${\bv_h}\in V_h^3$,
		\begin{align*}
		\text{$(i)$} \qquad \|(\P_h(\m)-\P_h(\widetilde \m)){\bv_h}\|_{L^2(\Om)^3} \leqslant 
		{}& C_R \|\m-\widetilde \m\|_{L^{p}(\Om)^3}\|\bv_h\|_{L^{q}(\Om)^3}, \\
		\intertext{for $(p,q) \in \{(2,\infty),(\infty,2)\}$, and}
		\text{$(ii)$} \qquad \|(\P_h(\m)-\P_h(\widetilde \m)){\bv_h}\|_{H^1(\Om)^3} 
		\leqslant {}& C_R \|\m-\widetilde \m\|_{H^{1}(\Om)^3}\|\bv_h\|_{L^{\infty}(\Om)^3}\\
		&\ + C_R \|\m-\widetilde \m\|_{L^{2}(\Om)^3}\|\bv_h\|_{W^{1,\infty}(\Om)^3} .
		\end{align*} %
	\end{lemma}

	The next lemma shows the $W^{s,p}$-stability of the projection. It is actually used for $p=2$ in the proof of Lemmas~\ref{lem:gal} 
	and~\ref{lem:diff} and will be used for $p=2$ in Section~\ref{Se:full-discr}  and for $p=\infty$ in Sections~\ref{Se:orders 1 and 2}  
	and~\ref{Se:orders 3 to 5}.
	
	\begin{lemma}\label{lem:stab} 
		There exists a constant depending only on $p\in [1,\infty]$ and the shape regularity of the mesh such that
		for all $\m\in W^{1,\infty}(\Om)^3$ with $|\m|=1$ almost everywhere,
		\begin{equation*}
		\| \P_h(\m)\bv_h\|_{W^{s,p}(\Om)^3}\leqslant C \|\m\|_{W^{1,\infty}(\Om)^3}^2\|\bv_h\|_{W^{s,p}(\Om)^3}
		\end{equation*}
		for all $\bv_h\in V_h^3$ and $s\in \{-1,0,1\}$.
	\end{lemma}

	These three lemmas will be proved in the course of this section, in which we formulate also three more lemmas that 
	are of independent interest but will not be used in the following sections.
	%
	%
	
	In the following, we use the dual norms 
	\[\|v\|_{W^{-1,q}}:=\sup_{w\in W^{1,p} }\frac{(v, w)}{\|w\|_{W^{1,p}}}\quad \text{for}\quad 1/p+1/q=1.\]
	The space $W^{-1,1}(\Om)$ is not the dual space of $W^{1,\infty}(\Om)$ but rather defined as the closure of 
	$L^2(\Om)$ with respect to the norm $\|\cdot\|_{W^{-1,1} }$.
	We also recall that $\varPi_h\colon W^{s,p}(\Om)$ $\to W^{s,p}(\Om)$ is uniformly bounded for $s\in\{0,1\}$ and $p\in [1,\infty]$ 
	(see, e.g.,~\cite{L2proj} for proofs in a much more general setting). By duality, we also obtain uniform boundedness
	for $s=-1$ and $p\in[1,\infty]$. A useful consequence is that for $v_h\in V_h$,
	\begin{align*}
	\|v_h\|_{W^{-1,q}} &{}=\sup_{w\in W^{1,p}}\frac{(v_h, \varPi_h w)}{\|w\|_{W^{1,p}}} 
	\\
	&{}\leqslant
	\sup_{w\in W^{1,p} }\frac{(v_h, \varPi_h w)}{\|\varPi_h w\|_{W^{1,p}}}\, 
	\sup_{w\in W^{1,p} }\frac{\|\varPi_h w\|_{W^{1,p}}}{\|w\|_{W^{1,p}}} \lesssim 
	\sup_{w_h\in V_h}\frac{(v_h, w_h)}{\|w_h\|_{W^{1,p}}}.
	\end{align*}

	\begin{lemma}
		There holds $\|v\|_{W^{s,p}(\Om)}\simeq\sup_{w\in W^{-s,q}(\Om)}\frac{(v,w)}{\|w\|_{W^{-s,q}(\Om)}}$ 
		with $1/p+1/q=1$ for $p\in[1,\infty]$ and $s\in\{-1,0,1\}$.
	\end{lemma}
	
	\begin{proof}
		The interesting case is $(s,p)=(1,\infty)$ since all other cases follow by duality. For $v\in W^{1,\infty}(\Om)$, 
		there exists a sequence of functions $\bq_n\in C^\infty_0(\Om)^3$
		with $\|\bq_n\|_{L^1 }=1$ such that
		\begin{equation*}
		\|\nabla v\|_{L^{\infty}}=\lim_{n\to\infty} (\nabla v,\bq_n)=\lim_{n\to\infty} -(v, \diver\bq_n) 
		\leqslant \sup_{\bq\in W^{1,1}}\frac{(v,\diver\bq)}{\|\bq\|_{L^1}} .
		\end{equation*}
		Moreover, there holds
		\begin{equation*}
		\|\diver\bq\|_{W^{-1,1}}\leqslant \sup_{w\in W^{1,\infty}}\frac{(\bq,\nabla w)}{\|\nabla w\|_{L^{\infty}}}\leqslant \|\bq\|_{L^1}.
		\end{equation*}
		Combining the last two estimates shows 
		\[\|\nabla v\|_{L^{\infty}}\leqslant \sup_{w\in W^{-1,1} }\frac{(v,w)}{\|w\|_{W^{-1,1}}}.\] 
		%
		Since 
		\[\| v\|_{L^\infty}=\sup_{w\in L^1 }\frac{(v,w)}{\|w\|_{L^1}}
		\leqslant \sup_{w\in W^{-1,1} }\frac{(v,w)}{\|w\|_{W^{-1,1}}},\]
		we conclude the proof.
	\end{proof}

	Let the discrete normal space $N_h(\m):= V_h^3\ominus T_h(\m)$ be given as the $L^2$-orthogonal
	complement of $T_h(\m)$ in $V_h^3$. We note that  
	\begin{equation}\label{normal space}
	N_h(\m)=\{\PPPh(\m\psi_h)\,:\,\psi_h\in V_h\}
	\end{equation}
	by the definition of $T_h(\m)$. The functions in the discrete normal space are bounded from below as follows.
	
	\begin{lemma}\label{lemma:definite} For every $R>0$,
		there exist $h_R>0$ and $c>0$  such that for all $\m\in W^{1,\infty}(\Om)^3$ with $|\m|=1$ almost everywhere and 
		$\| \m \|_{W^{1,\infty}(\Om)} \leqslant R$ and for all $h\leqslant h_R$,
		\[ \|\PPPh(\m \psi_h)\|_{W^{s,p}(\Om)^3}\geqslant c\, \|\psi_h\|_{W^{s,p}(\Om)}\]
		for all $\psi_h\in V_h$ and $(s,p)\in \{-1,0,1\}\times [1,\infty]$. 
	\end{lemma}
	\begin{proof} (a) We first prove the result for $s\in\{-1,0\}$.
		Let $\mathbf{I}_h\colon C(\Om)\to V_h^3$ denote the nodal interpolation operator and define $\mnull :=\mathbf{I}_h\m\in V_h^3$. 
		
		
		There holds
		\begin{equation*}
		\|\PPPh(\mnull  \psi_h)\|_{L^p}\geqslant \|\mnull  \psi_h\|_{L^p}-\|(\mathbf{I}-\PPPh)(\mnull  \psi_h)\|_{L^p}.
		\end{equation*}
		Moreover, stability of $\PPPh$ in $L^p(\Om)^3$, 
		for $1\leqslant p \leqslant \infty$, see \cite{L2proj}, 
		implies the estimate
		\begin{equation*}
		\|(\mathbf{I}-\PPPh)(\mnull  \psi_h)\|_{L^p} \leqslant (1+C) \inf_{\bv_h\in V_h^3} \|\mnull  \psi_h - \bv_h\|_{L^p} .
		\end{equation*}
		In turn, this implies 
		\begin{align*}
		\|(\mathbf{I}-\PPPh)(\mnull  \psi_h)\|_{L^p}&{}\lesssim \|(\mathbf{I}-\mathbf{I}_h)(\mnull  \psi_h)\|_{L^p}\\
		&{}=\Big(\sum_{T\in\TT_h}\|(\mathbf{I}-\mathbf{I}_h)(\mnull  \psi_h)\|_{L^p(T)^3}^p\Big)^{1/p}.
		\end{align*}
		For each element, the approximation properties of $\mathbf{I}_h$ show 
		\begin{align*}
		\|(\mathbf{I}-\mathbf{I}_h)(\mnull  \psi_h)\|_{L^p(T)^3}&{}\lesssim h^{r+1}\|\nabla^{r+1}(\mnull \psi_h)\|_{L^p(T)^3}\\
		&{}\leqslant h^{r+1}\sum_{i+j=r+1}\|\nabla^{\min\{i,r\}}\mnull \|_{L^\infty(T)^3}\|\nabla^{\min\{j,r\}}\psi_h\|_{L^p(T)^3}.
		\end{align*}
		Thus, multiple inverse estimates yield
		\begin{equation*}
		\|(\mathbf{I}-\mathbf{I}_h)(\mnull  \psi_h)\|_{L^p(T)^3}\lesssim 
		h\|\mnull \|_{W^{1,\infty} }\|\psi_h\|_{L^p(T)^3}.
		\end{equation*}
		Moreover, we  have
		\[\|\mnull  \psi_h\|_{L^p}\geqslant \|\m\psi_h\|_{L^p} - \|(\m-\mnull )\psi_h\|_{L^p} \geqslant \tfrac12 \|\psi_h\|_{L^p}\]
		provided that $\| \m-\mnull \|_{L^\infty}\leqslant \tfrac12$, which in view of
		\[\| \m-\mnull \|_{L^\infty}= \| (\mathbf{I}-\mathbf{I}_h)\m\|_{L^\infty} \lesssim h\|\nabla \m\|_{L^\infty}\]
		is satisfied for $h\leqslant h_R$ with a sufficiently small $h_R>0$ that depends only on $R$.
		Altogether, this shows
		\begin{equation*}
		\|\PPPh(\mnull  \psi_h)\|_{L^p}\gtrsim\| \psi_h\|_{L^p}
		\end{equation*}
		for $h\leqslant h_R$. Similarly we estimate
		\begin{equation*}
		\|\PPPh((\m-\mnull ) \psi_h)\|_{L^p}\lesssim \|\m-\mnull \|_{L^\infty}\| \psi_h\|_{L^p}\lesssim
		h\|\nabla\m\|_{L^\infty}\| \psi_h\|_{L^p}.
		\end{equation*}
		%
		Altogether, we obtain
		\begin{equation*}
		\|\PPPh(\m \psi_h)\|_{L^p}\gtrsim \|\PPPh(\mnull  \psi_h)\|_{L^p}-\|\PPPh((\mnull -\m) \psi_h)\|_{L^p}\gtrsim \| \psi_h\|_{L^p}
		\end{equation*}
		for $h\leqslant h_R$.
		This concludes the proof for $s= 0$.
		Finally, for $s=-1$ we note that by using the result for $s=0$ and an inverse inequality,
		\begin{align*}
		\|(\mathbf{I}-\PPPh)(\m \psi_h)\|_{W^{-1,p}}&{}\lesssim h \|\psi_h\|_{L^p}\\
		&{}\lesssim h\|\PPPh(\m \psi_h)\|_{L^p}\lesssim \|\PPPh(\m \psi_h)\|_{W^{-1,p}}.
		\end{align*}
		Since $ \|\m \psi_h\|_{W^{-1,p}} \gtrsim \|\m\|_{W^{1,\infty}}^{-1}\| \psi_h\|_{W^{-1,p}}$, this concludes the proof for $s\in\{-1,0\}$. 
		
		(b) It remains to prove the result for $s=1$.
		Note that the result follows from duality if we show
		\begin{equation}\label{eq:dualgoal}
		\|\varPi_h(\m\cdot\bw_h)\|_{W^{-1,q}} \gtrsim \|\bw_h\|_{W^{-1,q}}
		\end{equation}
		for all $\bw_h\in N_h(\m)$.  To see this, note that~\eqref{eq:dualgoal} implies
		\begin{align*}
		\|\PPPh&(\m \psi_h)\|_{W^{1,p}}\geqslant \sup_{\bw_h\in N_h(\m)}\frac{ (\psi_h, \varPi_h(\m\cdot \bw_h))}{\|\bw_h\|_{W^{-1,q}}}\\
		&\gtrsim
		\sup_{\bw_h\in N_h(\m)}\frac{ (\psi_h, \varPi_h(\m\cdot \bw_h))}{\|\varPi_h(\m\cdot \bw_h)\|_{W^{-1,q}}}
		=\sup_{\omega_h\in V_h}\frac{ (\psi_h, \omega_h)}{\|\omega_h\|_{W^{-1,q}}} \simeq \|\psi_h\|_{W^{1,p}},
		\end{align*}
		where we used in the second to last equality that part (a) 
		for $s=0$ already shows that $\dim (N_h(\m)) = \dim (V_h)$ 
		and since~\eqref{eq:dualgoal} implies that the map $N_h(\m)\to V_h,\,\bw_h\mapsto \varPi_h(\m\cdot \bw_h)$
		is injective, it is already bijective.
		It remains to prove~\eqref{eq:dualgoal}.
		To that end, we first show for $\bw_h=\PPPh(\m \omega_h)\in N_h(\m)$ for some $\omega_h\in V_h$, 
		using the reverse triangle inequality, that
		\begin{align*}
		\|\m\cdot \bw_h\|_{W^{-1,q}}&{} \geq \|\omega_h\|_{W^{-1,q}}-\|\m\cdot(\mathbf{I} -\PPPh)(\m \omega_h)\|_{W^{-1,q}}\\
		&{}\gtrsim  \|\m\|_{W^{1,\infty}}^{-1}\|\bw_h\|_{W^{-1,q}}-\|\m\cdot(\mathbf{I} -\PPPh)(\m \omega_h)\|_{W^{-1,q}}.
		\end{align*}
		With $\m_h:= \mathbf{I}_h(\m)\in V_h^3$, the last term satisfies
		\begin{align*}
		\|\m\cdot(\mathbf{I}& -\PPPh)(\m \omega_h)\|_{W^{-1,q}}\lesssim h \|\m\|_{W^{1,\infty}}\|(\mathbf{I} -\PPPh)(\m \omega_h)\|_{L^q}\\
		&\lesssim h \|\m\|_{W^{1,\infty}}(\|\m-\m_h\|_{L^\infty}\|\omega_h\|_{L^q} + h\|\m_h\|_{W^{1,\infty}}\|\omega_h\|_{L^q}),
		\end{align*}
		where we used the same arguments as in the proof of part (a) 
		to get the estimate $\|(\mathbf{I}-\PPPh)(\m_h \omega_h)\|_{L^q}
		\lesssim h\|\m_h\|_{W^{1,\infty}}\|\omega_h\|_{L^q}$. The fact $\|\m_h\|_{W^{1,\infty}} \lesssim \|\m\|_{W^{1,\infty}}$, 
		the approximation property 
		$\|\m-\m_h\|_{L^\infty}\lesssim h\|\m\|_{W^{1,\infty}}$, and an inverse inequality conclude
		\begin{equation}\label{eq:def1}
		\|\m\cdot \bw_h\|_{W^{-1,q}}\gtrsim \|\bw_h\|_{W^{-1,q}}
		\end{equation}
		with (hidden) constants depending only on $\|\m\|_{W^{1,\infty}}$ and shape regularity of the mesh.
		
		To prove~\eqref{eq:dualgoal}, it remains to bound the left-hand side above by $\|\varPi_h(\m\cdot \bw_h)\|_{W^{-1,q}}$. To that end, we note
		\begin{align*}
		\|(I&-\varPi_h)(\m\cdot \bw_h)\|_{W^{-1,q}}\lesssim h\|\bw_h\|_{L^q}=h\sup_{\bv\in L^p}\frac{(\bw_h,\bv)}{\|\bv\|_{L^p}}\\
		&\lesssim
		h\sup_{\bv\in N_h(\m)}\frac{(\bw_h,\bv)}{\|\bv\|_{L^p}}= h\sup_{v\in V_h}\frac{(\varPi_h(\m\cdot \bw_h),v)}{\|\PPPh(\m v)\|_{L^p}}
		\lesssim 
		h \|\varPi_h(\m\cdot \bw_h)\|_{L^q},
		\end{align*}
		where we used part (a) 
		for $s=0$ for the last inequality. An inverse inequality and the combination with~\eqref{eq:def1} imply~\eqref{eq:dualgoal} for $h>0$ sufficiently small
		in terms of $\|\m\|_{W^{1,\infty} }^{-1}$.
		This concludes the proof.
	\end{proof}

	\begin{lemma}\label{lem:matrixinv}
		Define the matrix $M\in\R^{N\times N}$, where $N$ denotes the dimension of~$V_h$, by $M_{ij}:=h^{-3}(\PPPh(\m \phi_j),\PPPh(\m \phi_i))$. 
		Under the assumptions of Lemma~\ref{lemma:definite},
		there exists $C>0$ such that for $h\leqslant h_R$,
		\[ \|M\|_p + \|M^{-1}\|_p \leqslant C \quad\text{ for } 1 \leqslant p \leqslant \infty,\]
		where $C$ depends only on the shape regularity. 
	\end{lemma}
	\begin{proof}
		Lemma~\ref{lemma:definite} shows for $x\in\R^N$
		\begin{equation}\label{eq:elliptic}
		Mx\cdot x = h^{-3}\|\PPPh(\m \sum_{i=1}^N x_i\phi_i)\|_{L^2}^2\gtrsim h^{-3}\|\sum_{i=1}^N x_i\phi_i\|_{L^2}\simeq |x|^2,
		\end{equation}
		where $|\cdot|$ denotes the Euclidean norm on $\R^N$.
		Let $d(i,j):= {\rm dist}(z_i,z_j)h^{-3}$ denote the metric which (approximately) measures the number of elements 
		between the supports of $\phi_i$ and $\phi_j$, 
		corresponding to the nodes $z_i$ and $z_j$, and let $B_d(z)$ denote the corresponding ball. 
		In the following, we use a localization property of the $L^2$-projection, i.e.,
		there exist $a,b>0$ such that for all $\ell\in\N$,
		\begin{equation}\label{eq:expdecay}
		\|\PPPh(\m\phi_i)\|_{L^2(\varOmega\setminus B_\ell(z_i))^3}\leqslant a \e^{-b\ell}\|\m\phi_i\|_{L^2}.
		\end{equation}
		The proof of this bound is essentially contained in the proof of~\cite[Lemma~3.1]{BAYS14}. 
		Since we use the very same arguments below, we briefly recall the strategy: 
		First, one observes that the mass matrix $\widetilde M\in \R^{N\times N}$ with entries $\widetilde M_{ij}:=h^{-3}(\phi_j,\phi_i)$ 
		is banded in the sense that $d(i,j)\gtrsim 1$ implies $\widetilde M_{ij}=0$,
		and it satisfies $\widetilde M x\cdot x\gtrsim |x|^2$.
		As shown below, this implies that the inverse matrix $\widetilde M^{-1}$ satisfies $|(\widetilde M^{-1})_{ij}|\lesssim \e^{-bd(i,j)}$ 
		for some $b>0$ independent of $h>0$.
		Note that each entry of the vector field $\PPPh(\m\phi_i)\in V_h^3$ can be represented by $\sum_{j=1}^N x_{k,j}\phi_j, k=1,2,3,$ and
		is computed by solving $\widetilde Mx_k = g_k\in\R^N$ with $\m=(m_1,m_2,m_3)^T$ and $g_{k,j}:= (m_k\phi_i,\phi_j)$. 
		Hence, the exponential decay of $\widetilde M^{-1}$ directly
		implies~\eqref{eq:expdecay}.
		
		
		From the decay property~\eqref{eq:expdecay}, we immediately obtain
		\begin{equation*}
		|M_{ij}|\leqslant \widetilde a \e^{-\widetilde b d(i,j)}
		\end{equation*}
		for all $1\leqslant i,j\leqslant N$ and some $\widetilde a,\widetilde b>0$. This already proves $\|M\|_p  \leqslant C$. 
		We follow the arguments from~\cite{Jaf90} to show that also $M^{-1}$
		decays exponentially. To that end, note that~\eqref{eq:elliptic} implies the existence of $c>0$ such that $\|I-cM\|_2=:q<1$ and hence
		\begin{equation}\label{eq:neumann}
		M^{-1} = c(I-(I-cM))^{-1} = c\sum_{k=0}^\infty (I-cM)^k.
		\end{equation}
		Clearly, $I-cM$ inherits the decay properties from $M$ and therefore
		\begin{align*}
		|((I-cM)^{k+1})_{ij}|&{}\leqslant \widetilde a^{k+1}\sum_{r_1,\dotsc,r_k=1}^N \e^{-\widetilde b (d(i,r_1)+\dotsb + d(r_k,j))}\\
		&{}\leqslant \widetilde a^{k+1}\Big(\max_{s=1,\dotsc,N}\sum_{r=1}^N \e^{-\widetilde b d(s,r)/2}\Big)^k \e^{-\widetilde b d(i,j)/2}.
		\end{align*}
		The value of $\max_{s=1,\dotsc,N}\sum_{r=1}^N \e^{-\widetilde b d(s,r)/2}$ depends only on the shape regularity 
		of the triangulation and on $\widetilde b$, but is independent of $h$ (it just
		depends on the number of elements contained in an annulus of thickness $\approx h$).
		This implies the existence of $\widetilde c\geqslant 1$ such that
		\begin{equation*}
		|((I-cM)^{k+1})_{ij}|\leqslant \min\{q^{k+1},\widetilde c^{k+1}\e^{-\widetilde b d(i,j)/2}\}.
		\end{equation*}
		Thus, for $\widetilde c^{k+1}\leqslant \e^{\widetilde b d(i,j)/4}$, we have $|((I-cM)^{k+1})_{ij}|\leqslant  \e^{-\widetilde b d(i,j)/4}$, 
		whereas for 
		$\widetilde c^{k+1}> \e^{\widetilde b d(i,j)/4}$, we have $|((I-cM)^{k+1})_{ij}|\leqslant q^{k+1}<q^{\widetilde b d(i,j)/(4\log(\widetilde c))}$. 
		Altogether, we find some $\widetilde b>0$
		(we reuse the symbol), independent of $h$ such that
		\begin{equation*}
		|((I-cM)^{k+1})_{ij}|\leqslant q^{(k+1)/2} |((I-cM)^{k+1})_{ij}|^{1/2}\lesssim q^{(k+1)/2}\e^{-\widetilde b d(i,j)}.
		\end{equation*}
		Plugging this into~\eqref{eq:neumann}, we obtain
		\begin{equation*}
		|(M^{-1})_{ij}|\lesssim \sum_{k=0}^\infty q^{(k+1)/2}\e^{-\widetilde b d(i,j)}\lesssim \e^{-\widetilde b d(i,j)}.
		\end{equation*}
		This yields the stated result. 
	\end{proof}
	
	We are now in a position to prove Lemma~\ref{lem:stab}.
	
	\begin{proof}[Proof of Lemma~\ref{lem:stab}]
		(a) We first consider the case $s=0$.
		In view of \eqref{normal space}, we write $(\mathbf{I}-\P_h(\m))\bv_h \in N_h(\m)$ as
		\[(\mathbf{I}-\P_h(\m))\bv_h= h^{-3/2}\sum_{i=1}^N x_i\PPPh(\m \phi_i)\]
		for some coefficient vector $x\in \R^N$ and let $b_i:= h^{-3/2}(\bv_h,\m \phi_i)$ for $i=1,\dotsc,N$. Then, there holds $Mx=b$ with the matrix $M$
		from Lemma~\ref{lem:matrixinv}. This lemma and the $L^p$-stability of the $L^2$-orthogonal projection $\Pi_h$ \cite{L2proj} imply that for $p\in[1,\infty]$,
		\begin{align*}
		&\| (\mathbf{I}-\P_h(\m))\bv_h\|_{L^p}= \| \PPPh h^{-3/2}\sum_{i=1}^N x_i\m \phi_i\|_{L^p}\lesssim 
		\| h^{-3/2}\sum_{i=1}^N x_i\m \phi_i\|_{L^p}\\
		&{}\lesssim
		h^{-3/2}\Big(\sum_{i=1}^N h^3|x_i|^p\Big)^{1/p}
		=  h^{3/p-3/2}|x|_{p} =  h^{3/p-3/2}|M^{-1} b|_{p}\lesssim   h^{3/p-3/2}| b|_{p}.
		\end{align*}
		With $|b_i|\leqslant h^{-3/2}\|\bv_h\|_{L^p(\supp(\phi_i))^3}h^{3(1-1/p)}=\|\bv_h\|_{L^p(\supp(\phi_i))^3}h^{3/2-3/p}$, this shows
		\[\| \P_h(\m)\bv_h\|_{L^p} \lesssim  \|\bv_h\|_{L^p}.\]

		(b) We now turn to the cases $s=\pm1$. Define the operator 
		\[\widetilde \P_h^\perp(\m)\bv_h:=\PPPh(\m\varPi_h(\m\cdot\bv_h))\]
		and note that $\widetilde \P_h^\perp(\m)\bv_h\in N_h(\m)$ as well as
		$\kernel \widetilde \P_h^\perp(\m) = T_h(\m)$ (due to Lemma~\ref{lemma:definite}). However, $\widetilde \P_h^\perp(\m)$ is no projection. 
		We observe for $\bv_h=\PPPh(\m\psi_h) \in N_h(\m)$ that
		\begin{align*}
		\|(\mathbf{I}-\widetilde \P_h^\perp(\m))\bv_h\|_{W^{-1,p}} &{}= \|\PPPh\m\psi_h -\PPPh(\m\varPi_h(\m\cdot\PPPh(\m\psi_h))) \|_{W^{-1,p}} \\
		&{}\lesssim 
		\|\m\|_{W^{1,\infty}}\|\psi_h -\m\cdot\PPPh(\m\psi_h) \|_{W^{-1,p}}\\
		&{}=
		\|\m\|_{W^{1,\infty}}^2\|(\mathbf{I}-\PPPh)(\m\psi_h) \|_{W^{-1,p}}\\
		&{}\lesssim
		\|\m\|_{W^{1,\infty}}^2 \,h\|\psi_h \|_{L^p}.
		\end{align*}
		With Lemma~\ref{lemma:definite} we conclude 
		\[ \|(\mathbf{I}-\widetilde \P_h^\perp(\m))\bv_h\|_{W^{-1,p}}\lesssim  \|\m\|_{W^{1,\infty}}^2h\|\bv_h\|_{L^p}.\]
		Since $\widetilde \P_h^\perp(\m) \P_h(\m) =0$ by definition of $T_h(\m)$, we obtain  with part (a) and an inverse inequality that for all $\bv_h\in V_h^3$,
		\begin{align*}
		\|(\mathbf{I}-\P_h(\m)-\widetilde \P_h^\perp(\m))\bv_h\|_{W^{-1,p}} &{}= \|(\mathbf{I}-\widetilde \P_h^\perp(\m))(\mathbf{I}-\P_h(\m))\bv_h\|_{W^{-1,p}}\\
		&{}\lesssim  \|\m\|_{W^{1,\infty}}^2h\|(\mathbf{I}-\P_h(\m))\bv_h\|_{L^p}\\
		&{}\lesssim  \|\m\|_{W^{1,\infty}}^2h\|\bv_h\|_{L^p}\\
		&{}\lesssim \|\m\|_{W^{1,\infty}}^2\|\bv_h\|_{W^{-1,p}}.
		\end{align*}
		The $W^{-1,p}(\Om)$-stability of $\varPi_h$ implies
		$\|\widetilde \P_h^\perp(\m)\bv_h\|_{W^{-1,p}}\lesssim \|\m\|_{W^{1,\infty}}^2\|\bv_h\|_{W^{-1,p}}$ and
		the triangle inequality concludes the proof for $s=-1$. The case $s=1$ follows by duality.
	\end{proof}

	\begin{proof}[Proof of Lemma~\ref{lem:diff}]
		(a) ($s=0$)
		The projection $\bv_h:=\P_h(\m)\bv$ is given by the equation
		\[(\bv_h,\bphi_h) = (\bv,\bphi_h) \qquad \forall\,\bphi_h\in T_h(\m),\]
		which in view of the definition of $T_h(\m)$ is equivalent to the solution of the saddle point problem 
		(with the Lagrange multiplier $\lambda_h\in V_h$)
		\begin{equation*}
		\begin{alignedat}{3}
		&(\bv_h,\bw_h) + (\m\cdot\bw_h,\lambda_h) &&{}= (\bv,\bw_h) \quad &&\forall \bw_h \in V_h^3, \\
		&(\m\cdot\bv_h, \mu_h)  &&{}= 0  &&\forall \mu_h \in V_h.
		\end{alignedat}
		\end{equation*}
		By the first equation, we also obtain the identity $\PPPh(\m\lambda_h)=(\mathbf{I}-\P_h(\m))\bv_h$, which will be used below.
		Furthermore, $\widetilde \bv_h :=\P_h(\widetilde\m)\bv$ is given by the same system with $\widetilde \m$ in place of $\m$, 
		yielding a corresponding Lagrange multiplier $\widetilde \lambda_h$. Hence, the differences $\be_h := \bv_h - \widetilde \bv_h$ and 
		$\delta_h := \lambda_h - \widetilde\lambda_h$ satisfy 
		\begin{equation*}
		\begin{alignedat}{3}
		&(\be_h,\bw_h) + (\m\cdot\bw_h,\delta_h) &&{}= -(\bw_h,(\m-\widetilde\m)\widetilde\lambda_h) \quad &&\forall \bw_h \in V_h^3, \\
		&(\m\cdot\be_h, \mu_h)  &&{}= - ((\m-\widetilde\m)\cdot\widetilde\bv_h,\mu_h) \quad &&\forall \mu_h \in V_h.
		\end{alignedat}
		\end{equation*}
		The classical results on saddle-point problems (see \cite[Proposition~2.1]{brezzi}) require two inf-sup conditions to be satisfied.
		First, 
		\[\inf_{q_h\in V_h}\sup_{\bv_h\in V_h^3}\frac{(\m\cdot\bv_h,q_h)}{\|\bv_h\|_{H^s }\|q_h\|_{H^{-s}}}>0\]
		holds uniformly in $h$
		due to Lemma~\ref{lemma:definite}. 
		Second, 
		\[\inf_{\bw_h\in T_h(\m)}\sup_{\bv_h\in T_h(\m)}\frac{(\bv_h,\bw_h)}{\|\bv_h\|_{H^s }\|\bw_h\|_{H^{-s} }}>0\]
		holds uniformly in $h$ due to the stability estimates from Lemma~\ref{lem:stab} 
		(noting that $\bv_h=\P_h(\m)\bv_h$ and $\bw_h=\P_h(\m)\bw_h$ for $\bv_h,\bw_h\in T_h(\m)$).
		For the above saddle-point problems, these bounds for $s=0$ give us an $L^2$ bound for 
		$\be_h=\P_h(\m)\bv-\P_h(\widetilde\m)\bv$: From \cite{brezzi} we obtain
		\[\| \widetilde\bv_h \|_{L^2} + \| \widetilde\lambda_h \|_{L^2} \lesssim \| \bv \|_{L^2}\]
		and
		\begin{equation*}
		\| \be_h  \|_{L^2} + \| \delta_h \|_{L^2} \lesssim \| (\m-\widetilde\m)\widetilde\lambda_h \|_{L^2} +
		\| (\m-\widetilde\m)\cdot\widetilde\bv_h \|_{L^2}.
		\end{equation*}
		With the stability from Lemma~\ref{lem:stab} and Lemma~\ref{lemma:definite}, we also obtain
		\begin{equation*}
		\|\widetilde\bv_h\|_{L^\infty} + \|\widetilde\lambda_h\|_{L^\infty}\lesssim \|\P_h(\widetilde\m)\bv\|_{L^\infty} 
		+ \|(\mathbf{I}-\P_h(\widetilde\m))\bv\|_{L^\infty}
		\lesssim\|\bv\|_{L^\infty}.
		\end{equation*}
		Altogether, this implies
		\begin{equation*}
		\| \be_h  \|_{L^2} + \| \delta_h \|_{L^2}\lesssim \| \m-\widetilde\m\|_{L^p}\|\bv \|_{L^q}
		\end{equation*}
		for $(p,q)\in\{(2,\infty),(\infty,2)\}$.

		(b) ($s=1$) For the $H^1(\Om)$-estimate, we introduce the Riesz mapping $J_h$ between $V_h\subset H^1(\Om)$ 
		and its dual $V_h\subset H^{1}(\Om)'$, i.e., the isometry defined by
		\[ (v_h, J_h \psi_h)_{H^1 } = \langle v_h, \psi_h \rangle \qquad \forall v_h\in {V_h},\ \psi_h \in {V_h}.\]
		By $\mathbf{J}_h:= \mathbf{I}\otimes J_h$ we denote the corresponding vector-valued mapping on $V_h^3$.
		We consider the bilinear form on  ${V_h^3}\times {V_h^3}$ defined by 
		\[a_h(\bv_h,\bw_h)= \langle \bv_h,  \mathbf{J}_h^{-1}\bw_h \rangle, \quad \bv_h,\bw_h\in {V_h^3},\]
		and reformulate the saddle-point problem for $(\bv_h,\lambda_h)\in V_h^3\times V_h \subset H^1(\Om)^3\times H^1(\Om)'$ as
		\begin{equation*}
		\begin{alignedat}{3}
		&a_h(\bv_h,\bw_h) + \langle\m\cdot {\mathbf{J}_h^{-1}\bw_h},\lambda_h\rangle &&{}= a(\bv,\bw_h)\quad &&\forall \bw_h \in V_h^3, \\
		&\langle\m\cdot\bv_h, {J_h^{-1}\mu_h} \rangle &&{}= 0 \quad  &&\forall \mu_h \in V_h.
		\end{alignedat}
		\end{equation*}
		{As in the case $s=0$ (algebraically it is the same system), we have $\bv_h=\P_h(\m)\bv$ and $\PPPh(\m \lambda_h ) 
			= (\mathbf{I}-\P_h(\m))\bv$.} The system for $\be_h = \bv_h - \widetilde \bv_h$ and
		$\delta_h = \lambda_h - \widetilde\lambda_h$ reads
		\begin{equation*}
		\begin{alignedat}{3}
		&a_h(\be_h,\bw_h) + \langle\m\cdot{\mathbf{J}_h^{-1}\bw_h},\delta_h\rangle 
		&&{}= -\langle (\m-\widetilde\m)\cdot{\mathbf{J}_h^{-1}\bw_h},\widetilde\lambda_h\rangle 
		\quad &&\forall \bw_h \in V_h^3, \\
		&\langle\m\cdot \be_h, {J_h^{-1}\mu_h}\rangle &&{}= - \langle(\m-\widetilde\m)\cdot\widetilde\bv_h,{J_h^{-1}\mu_h}\rangle \quad  &&\forall \mu_h \in V_h.
		\end{alignedat}
		\end{equation*}
		The above inf-sup bounds for $s=1$ and $s=-1$ 
		are precisely the inf-sup conditions that need to be satisfied for these generalized saddle-point problems (see~\cite[Theorem~2.1]{ciarlet}), 
		whose right-hand sides are bounded by 
		\[|a_h(\bv,\bw_h)| \leqslant \| \bv \|_{H^1} \, \| \mathbf{J}_h^{-1} \bw_h \|_{H^{-1}} \simeq \| \bv \|_{H^1} \, \| \bw_h \|_{H^{1}}\]
		and
		\begin{align*}
		&| \langle (\m-\widetilde\m)\cdot{\mathbf{J}_h^{-1}\bw_h},\widetilde\lambda_h\rangle | \lesssim 
		\| (\m-\widetilde\m)\widetilde \lambda_h \|_{H^1}\, \| \bw_h \|_{H^1},\\
		& | \langle(\m-\widetilde\m)\cdot\widetilde\bv_h,{J_h^{-1}\mu_h}\rangle | \leqslant 
		\| (\m-\widetilde\m)\cdot \widetilde \bv_h \|_{H^1}\,\| \mu_h  \|_{H^{1}}.
		\end{align*}
		{
			As in the case $s=0$, we obtain from Lemma~\ref{lem:stab} and Lemma~\ref{lemma:definite} that
			\begin{align*}
			\|\widetilde\bv_h\|_{W^{1,\infty}} + \|\widetilde\lambda_h\|_{W^{1,\infty}}
			&{}\lesssim \|\P_h(\widetilde\m)\bv\|_{W^{1,\infty}} + \|(\mathbf{I}-\P_h(\widetilde\m))\bv\|_{W^{1,\infty}}
			\\
			&{} \lesssim\|\bv\|_{W^{1,\infty}}.
			\end{align*}
			Hence, we obtain from~\cite[Theorem~2.1]{ciarlet}, for $(p,q)\in \{(2,\infty),(\infty,2)\}$,
			\begin{align*}
			\| \be_h \|_{H^1} &{}\lesssim \| (\m-\widetilde\m)\widetilde \lambda_h \|_{H^1} +
			\| (\m-\widetilde\m)\cdot \widetilde\bv_h \|_{H^1}\\
			&{}\lesssim \sum_{s'=0}^1\Big(\| \m-\widetilde\m \|_{H^1}\,  \| \widetilde\lambda_h  \|_{W^{1-s',q}} +
			\| \m-\widetilde\m \|_{W^{s',p}}\, \| \widetilde\bv_h \|_{W^{1-s',q}}\Big)
			\\
			&{} \lesssim \sum_{s'=0}^1\| \m-\widetilde\m \|_{W^{s',p}} \, \| \bv \|_{W^{1-s',q}}.
			\end{align*}
			This implies the $H^1(\Om)^3$ estimate and hence concludes the proof. }
	\end{proof}

	\begin{proof}[Proof of Lemma~\ref{lem:gal}]
		Since $\P_h(\m)\bv$ is the Galerkin approximation of the saddle point problem for $\P(\m)\bv$ (as in the previous proof),
		the C\'ea lemma for saddle-point problems (see~\cite[Theorem~2.1]{brezzi}) shows in $L^2$
		\begin{align*}
		&{} \|(\P_h(\m)-\P(\m))\bv\|_{L^2} \\
		&{}\lesssim \inf_{(\bw_h,\mu_h)\in V_h^3\times V_h}\Big(\|\P(\m)\bv - \bw_h\|_{L^2} +\|\m\cdot\bv - \mu_h\|_{L^2} \Big)\\
		&{}\lesssim
		h^{r+1 }\| \m\|_{W^{r+1,\infty}}\|\bv\|_{H^{r+1}}
		\end{align*}
		and similarly in $H^1$, using \cite[Theorem~2.1]{ciarlet},
		\begin{align*}
		&{} \|(\P_h(\m)-\P(\m))\bv\|_{H^1} \\
		&{}\lesssim \inf_{(\bw_h,\mu_h)\in V_h^3\times V_h}\Big(\|\P(\m)\bv - \bw_h\|_{H^1} +\|\m\cdot\bv - \mu_h\|_{H^1} \Big)\\
		&{}\lesssim
		h^r\| \m\|_{W^{r+1,\infty}}\|\bv\|_{H^{r+1}}.
		\end{align*} %
		This concludes the proof.
	\end{proof}

	\section{Consistency error and error equation}\label{Se:full-discr}
	
	To study the consistency errors, we find it instructive to separate the issues of consistency for the time and space discretizations. 
	Therefore, we first show defect estimates for the semidiscretization in time, and then turn to the full discretization. 
	
	\subsection{Consistency error of the semi-discretization in time} \label{subsec:cons}
	
	The order of both the fully implicit $k$-step BDF method, described by the coefficients
	$\delta_0,\dotsc,\delta_k$ and $1,$ and the explicit $k$-step BDF method,
	that is the method  described by the coefficients
	$\delta_0,\dotsc,\delta_k$ and $\gamma_0,\dotsc,\gamma_{k-1},$ is $k,$ i.e.,
	\begin{equation}
	\label{order}
	\sum_{i=0}^k(k-i)^\ell \delta_i=\ell k^{\ell-1}=
	\ell \sum_{i=0}^{k-1}(k-i-1)^{\ell-1} \gamma_i,\quad \ell=0,1,\dotsc,k.  
	\end{equation}

	We first rewrite the linearly implicit $k$-step BDF method \eqref{BDF1} in strong form,
	\begin{equation}
	\label{BDF1str}
	\alpha \dot{\bm{m}}^n+\widehat{\bm{m}}^n\times \dot{\bm{m}}^n=\P(\widehat{\bm{m}}^n)( \varDelta \bm{m}^n+\bm{H}^n),
	\end{equation}
	with Neumann boundary conditions.
	
	The consistency error $\bm{d}^n$ 
	of the linearly implicit $k$-step BDF method \eqref{BDF1str} for the solution $\bm{m}$ is the defect 
	by which the exact solution misses satisfying \eqref{BDF1str}, and is given by 
	\begin{equation}
	\label{cons-err}
	\bm{d}^n =\alpha \dot{\bm{m}}^n_\star+\widehat{\bm{m}}^n_\star\times \dot{\bm{m}}^n_\star
	-\P(\widehat{\bm{m}}^n_\star)( \varDelta \bm{m}^n_\star+\bm{H}^n)
	\end{equation}
	for $n=k,\dotsc,N$, where we use the notation $\bm{m}^n_\star = \m(t_n)$ and
	\begin{equation}
	\label{eq:extrapolation and discr derivative for exact solution}
	\begin{aligned}
	&\widehat{\bm{m}}^n_\star = \sum_{j=0}^{k-1} \gamma_j \ms^{n-j-1}\Big/\Big| \sum_{j=0}^{k-1} \gamma_j \ms^{n-j-1}\Big|, 
	\\
	&\dot{\bm{m}}^n_\star = \P(\widehat{\bm{m}}^n_\star) \frac1\tau \sum_{j=0}^{k} \delta_j \ms^{n-j} \in T(\widehat{\bm{m}}^n_\star).
	\end{aligned}
	\end{equation}
	Note that the definition of $\dms^n$ contains the projection $\P(\wms^n)$, while $\dot\m^n$ was defined without a projection 
	(see the first formula in \eqref{shn1n}), since $\dot\m^n = \P(\wm^n)\dot\m^n$ is automatically satisfied due to the constraint 
	in \eqref{BDF1}.  
	
	The consistency error is bounded as follows.
	
	\begin{lemma}\label{lemma:cons}
		If the solution of the LLG equation \eqref{llg-projection} has the regularity
		\[
		\m \in C^{k+1}([0,\bar t\,], L^2(\Om)^3) \cap C^1([0,\bar t\,], L^\infty(\Om)^3) \ \ \text{and} \ \  \varDelta \bm{m} +\bm{H} \in C([0,\bar t\,], L^\infty(\Om)^3),
		\]
		then the consistency error \eqref{cons-err} is bounded by
		\[
		\|\bm{d}^n \|_{L^2(\Om)^3}\leqslant C\tau^k
		\]
		for $n=k,\dotsc,N$.
	\end{lemma}
	%
	
	\begin{proof} We begin by rewriting the equation for the defect as
		\begin{equation}
		\label{cons2}
		\begin{aligned}
		\bm{d}^n
		&{}=\alpha \dot{\bm{m}}^n_\star+\widehat{\bm{m}}^n_\star\times \dot{\bm{m}}^n_\star
		-\P(\bm{m}^n_\star)( \varDelta \bm{m}^n_\star+\bm{H}^n)\\
		&{}\qquad-\big (\P(\widehat{\bm{m}}^n_\star)-\P(\bm{m}^n_\star)\big )( \varDelta \bm{m}^n_\star+\bm{H}^n).
		\end{aligned}
		\end{equation}
		In view of \eqref{llg-projection}, we have
		\[\P(\bm{m}^n_\star)( \varDelta \bm{m}^n_\star+\bm{H}^n)=
		\alpha\, \partial_t\bm{m}(t_n) + \bm{m}^n_\star\times \partial_t\bm{m}(t_n),\]
		and can rewrite \eqref{cons2} as
		\begin{equation}
		\nonumber
		\begin{aligned}
		\bm{d}^n&{}=\alpha \big (\dot{\bm{m}}^n_\star-\partial_t\bm{m}(t_n)\big )
		+\big (\widehat{\bm{m}}^n_\star\times \dot{\bm{m}}^n_\star-\ms^n\times \partial_t\bm{m}(t_n)\big )\\
		&{}-\big (\P(\widehat{\bm{m}}^n_\star)-\P(\bm{m}^n_\star)\big )( \varDelta \bm{m}^n_\star+\bm{H}^n),
		\end{aligned}
		\end{equation}
		i.e.,
		\[\begin{aligned}
		\bm{d}^n&{}=\alpha\big (\dot{\bm{m}}^n_\star-\partial_t\bm{m}(t_n)\big )+(\widehat{\bm{m}}^n_\star-\bm{m}^n_\star)\times \dot{\bm{m}}^n_\star+
		\bm{m}^n_\star\times \big (\dot{\bm{m}}^n_\star-\partial_t\bm{m}(t_n)\big )\\
		&{}-\big (\P(\widehat{\bm{m}}^n_\star)-\P(\bm{m}^n_\star)\big )( \varDelta \bm{m}^n_\star+\bm{H}^n).
		\end{aligned}\]
		Therefore,
		\begin{equation}
		\label{split1}
		\bm{d}^n=\alpha\dot{\bm{d}}^n+\widehat{\bm{d}}^n\times \dot{\bm{m}}^n_\star+\bm{m}^n_\star\times \dot{\bm{d}}^n
		-\big (\P(\widehat{\bm{m}}^n_\star)-\P(\bm{m}^n_\star)\big )( \varDelta \bm{m}^n_\star+\bm{H}^n),
		\end{equation}
		with
		\begin{equation}
		\label{split2}
		\dot{\bm{d}}^n:=\dot{\bm{m}}^n_\star-\partial_t\bm{m}(t_n),\quad \widehat{\bm{d}}^n:=\widehat{\bm{m}}^n_\star-\bm{m}^n_\star.
		\end{equation}
		%
		%

		Now, in view of the first estimate in Lemma \ref{lemma:projection errors}, we have
		\[\| \big (\P(\widehat{\bm{m}}^n_\star)-\P(\bm{m}^n_\star)\big ) ( \varDelta \bm{m}^n_\star+\bm{H}^n) \|_{L^2 } 
		\leqslant C \|\widehat{\bm{m}}^n_\star-\bm{m}^n_\star\|_{L^2 },\]
		i.e.,
		\begin{equation}
		\label{cons2nn}
		\| \big (\P(\widehat{\bm{m}}^n_\star)-\P(\bm{m}^n_\star)\big ) ( \varDelta \bm{m}^n_\star+\bm{H}^n) \|_{L^2 }
		\leqslant C \|\widehat{\bm{d}}^n\|_{L^2 }. 
		\end{equation}
		Therefore, it suffices to estimate $\dot{\bm{d}}^n$ and $\widehat{\bm{d}}^n$. 
		
		To estimate $\widehat{\bm{d}}^n$, we shall proceed in two steps. First we shall estimate the extrapolation error
		\begin{equation}
		\label{cons5}
		\sum\limits^{k-1}_{j=0}\gamma_j\bm{m}_\star^{n-j-1}-\bm{m}^n_\star
		\end{equation}
		and then $\widehat{\bm{d}}^n.$
		
		By Taylor expanding about $t_{n-k},$ the leading terms of order up to $k-1$ cancel,
		due to the second equality in \eqref{order}, and we obtain 
		\begin{equation}
		\label{eq:Peano kernel for extrapolation}
		\begin{aligned}
		\sum\limits^{k-1}_{i=0}\gamma_i\bm{m}_\star^{n-i-1}-\bm{m}^n_\star
		=\frac 1{(k-1)!}\Bigg [&{} \sum\limits^{k-1}_{j=0}\gamma_j \int_{t_{n-k}}^{t_{n-j-1}}(t_{n-j-1}-s)^{k-1}\bm{m}^{(k)}(s) \d s\\
		&{}-\int_{t_{n-k}}^{t_n}(t_n-s)^{k-1}\bm{m}^{(k)}(s) \d s\Bigg ],
		\end{aligned}
		\end{equation}
		with $\bm{m}^{(\ell)}:=\frac {\partial^\ell \bm{m}} {\partial t^\ell},$ whence
		\begin{equation}
		\label{cons6}
		\Big \|\sum\limits^{k-1}_{i=0}\gamma_i\bm{m}_\star^{n-i-1}-\bm{m}^n_\star\Big \|_{L^2 } \leqslant C\tau^k.
		\end{equation}
		%
		Now, for a normalized vector $\bm{a}$ and a non-zero vector $\bm{b},$ we have
		\[\bm{a}-\frac {\bm{b}}{|\bm{b}|}=(\bm{a}-\bm{b})+\frac 1{|\bm{b}|}(|\bm{b}|-|\bm{a}|)\bm{b},\]
		whence
		\[\big |\bm{a}-\frac {\bm{b}}{|\bm{b}|}\big |\leqslant 2|\bm{a}-\bm{b}|.\]
		Therefore, \eqref{cons6} yields
		\begin{equation}
		\label{cons7}
		\|\widehat{\bm{d}}^n \|_{L^2 }\leqslant C\tau^k.
		\end{equation}
		
		To bound $\dot{\bm{d}}^n,$ we use the fact that $\P(\m(t_n))\partial_t\m(t_n)= \partial_t\m(t_n)\in T(\m(t_n))$, 
		so that we have
		\begin{align*}
		\dot{\bm{d}}^n &{}= \P(\widehat{\bm{m}}^n_\star) \frac1\tau \sum_{j=0}^{k} \delta_j \m(t_{n-j}) - \partial_t\m(t_n) 
		\\
		&{}= \P(\widehat{\bm{m}}^n_\star) \Bigl( \frac1\tau \sum_{j=0}^{k} \delta_j \m(t_{n-j}) - \partial_t\m(t_n) \Bigr)
		+ \bigl( \P(\widehat{\bm{m}}^n_\star) - \P(\m(t_n)) \bigr) \partial_t\m(t_n).
		\end{align*}
		By Lemma~\ref{lemma:projection errors} and \eqref{cons7}, we have for the last term
		\[
		\| \bigl( \P(\widehat{\bm{m}}^n_\star) - \P(\m(t_n)) \bigr) \partial_t\m(t_n) \|_{L^2 }\leqslant C\tau^k.
		\]
		By Taylor expanding the first term about $t_{n-k},$ we see that, due to the order conditions of the implicit BDF method,
		i.e., the first equality in \eqref{order}, the leading  terms of order up to $k-1$ cancel, and we obtain
		\begin{equation}
		\label{cons3}
		\begin{aligned}
		\frac1\tau \sum_{j=0}^{k} \delta_j \m(t_{n-j}) - \partial_t\m(t_n)=\frac 1{k!}\Bigg [ \frac1\tau\sum\limits^k_{j=0} 
		&{}\delta_j\!\int_{t_{n-k}}^{t_{n-j}}(t_{n-j}-s)^k\bm{m}^{(k+1)}(s)\d s\\
		&{}-k \int_{t_{n-k}}^{t_n}(t_n-s)^{k-1}\bm{m}^{(k+1)}(s) \d s\Bigg ],
		\end{aligned}
		\end{equation}
		%
		whence
		\begin{equation}
		\label{cons4}
		\| \dot{\bm{d}}^n \|_{L^2 }\leqslant C\tau^k,
		\end{equation}
		provided the solution $\bm{m}$ is sufficiently regular.  
		Now, \eqref{split1}, \eqref{cons2nn}, \eqref{cons4},  and \eqref{cons7} yield 
		\begin{equation}
		\label{cons8}
		\|\bm{d}^n \|_{L^2 }\leqslant C\tau^k.
		\end{equation}
		This is the desired consistency estimate, which is valid for BDF methods of arbitrary order $k$.
	\end{proof}	
	
	\subsection{Consistency error of the full discretization}
	
	We define the Ritz projection $\RRh\colon H^1(\varOmega)\to V_h$ corresponding to the Poisson--Neumann problem 
	via
	\begin{equation*}
	\bigl(\nabla \RRh \varphi,\nabla \psi\bigr) + \bigl(\RRh\varphi,1\bigr)\bigl(\psi,1\bigr)= 	\bigl(\nabla \varphi,\nabla \psi\bigr)
	+\bigl(\varphi,1\bigr)\bigl(\psi,1\bigr)
	\end{equation*}
	for all $\psi\in V_h$, and we denote $\RRRh =\Id\otimes \RRh \colon H^1(\varOmega)^3\to V_h^3$. We denote again 
	the $L^2$-orthogonal projections onto the finite element space by $\PPh\colon L^2(\Om)\to V_{h}$ and
	$\PPPh=\Id\otimes\PPh\colon L^2(\Om)^3 \to V_h^3$.
	As in the previous section, we write $\P_h({\bm{m}})$ 
	for the $L^2$-orthogonal projection onto the discrete tangent space at $\bm{m}$.
	We insert the following quantities, which are related to the exact solution,
	\begin{align}
	\nonumber
	&\bm{m}^n_{\star,h} = \RRRh\m(t_n), \\ 
	\label{wmshn}
	&\widehat{\bm{m}}^n_{\star,h} = \sum_{j=0}^{k-1} \gamma_j \bm{m}_{\star,h}^{n-j-1}\Big/\Big| \sum_{j=0}^{k-1} \gamma_j \bm{m}_{\star,h}^{n-j-1}\Big|, \\
	\nonumber
	&\dot{\bm{m}}^n_{\star,h} = \P_h(\widehat{\bm{m}}_{\star,h}^n)\frac1\tau \sum_{j=0}^{k} \delta_j \bm{m}_{\star,h}^{n-j} \in T_h(\widehat{\bm{m}}_{\star,h}^n),
	\end{align}
	into the linearly implicit $k$-step BDF method \eqref{BDF1-h}
	and obtain a defect $\bm{d}_h^n\in T_h(\widehat{\bm{m}}_{\star,h}^n)$ from
	\begin{equation}
	\label{BDF1str-full}
	\alpha \bigl( \dot{\bm{m}}_{\star,h}^n , \bm{\varphi}_h \bigr) 
	+ \bigl( \widehat{\bm{m}}_{\star,h}^n \times \dot{\bm{m}}_{\star,h}^n , \bm{\varphi}_h\bigr) 
	= - \bigl( \nabla\bm{m}_{\star,h}^n , \nabla\bm{\varphi}_h\bigr) 
	+ \bigl(\bm{H}^n,\bm{\varphi}_h\bigr) 
	+ \bigl(\bm{d}_h^n,\bm{\varphi}_h\bigr)
	\end{equation}
	for all $\bm{\varphi}_h\in T_h(\widehat{\bm{m}}_{\star,h}^n)$. By definition, there holds $(\RRh\varphi,1)=(\varphi,1)$ (this can be seen by testing with $\psi=1$) and hence
	\begin{equation*}
	\bigl(\nabla \bm{m}_{\star,h}^n,\nabla\bm{\varphi}\bigr) = \bigl(\nabla \bm{m}(t_n),\nabla\bm{\varphi}\bigr)=-\bigl(\varDelta \bm{m}(t_n),\bm{\varphi}\bigr).
	\end{equation*}
	Thus,  we obtain the consistency error for the full discretization by
	\begin{equation}
	\label{cons-err-full}
	\bm{d}^n_h =\P_h(\widehat{\bm{m}}_{\star,h}^n)\bm{D}_h^n \quad\text{with}\quad
	\bm{D}_h^n=\alpha \dot{\bm{m}}_{\star,h}^n+\widehat{\bm{m}}_{\star,h}^n\times \dot{\bm{m}}_{\star,h}^n-\varDelta \bm{m}(t_n)-\bm{H}(t_n)
	\end{equation}
	for $n=k,\dotsc,N$.
	The consistency error is bounded as follows.
	\begin{lemma}\label{lemma:cons-full}
		If the solution of the LLG equation \eqref{llg-projection} has the regularity
		\begin{align*}
		&\m \in C^{k+1}([0,\bar t\,], L^2(\Om)^3) \cap C^1([0,\bar t\,], W^{{r+1},\infty}(\Om)^3) \quad\text{and}\quad  \\
		&\varDelta \bm{m} +\bm{H} \in C([0,\bar t\,], W^{r+1,\infty}(\Om)^3),
		\end{align*}
		then the consistency error \eqref{cons-err-full} is bounded by
		\[
		\|\bm{d}^n_h \|_{L^2(\Om)^3}\leqslant C(\tau^k+h^r)
		\]
		for $n$ with $k\tau\leqslant n\tau\leqslant \bar t$.
	\end{lemma}
	%
	
	\begin{proof} We begin by defining
		\begin{equation*}
		\bm{D}^n:=\alpha \partial_t \bm{m}(t_n)+\bm{m}(t_n)\times \partial_t \bm{m}(t_n)-\varDelta \bm{m}(t_n)-\bm{H}(t_n)
		\end{equation*}
		and note that $\P(\bm{m}_\star^n)\bm{D}^n=0$. Here we denote again $\bm{m}_\star^n=\m(t_n)$ and in the following 
		we use also the notations $\dot{\bm{m}}_{\star}^n$ and $\widehat{\bm{m}}_{\star}^n$ as defined 
		in \eqref{eq:extrapolation and discr derivative for exact solution}.
		With this, we rewrite the equation for the defect as
		\begin{align*}
		\bm{d}_h^n = {}& \P_h(\widehat{\bm{m}}_{\star,h}^n)\bm{D}_h^n - \P(\bm{m}_\star^n)\bm{D}^n \\
		= {}& \P_h(\widehat{\bm{m}}_{\star,h}^n) \bigl(\bm{D}_h^n - \bm{D}^n\bigr) 
		+ \bigl( \P_h(\widehat{\bm{m}}_{\star,h}^n) - \P_h(\widehat{\bm{m}}_{\star}^n) \bigr) \bm{D}^n \\ 
		&{}+ \bigl( \P_h(\widehat{\bm{m}}_{\star}^n) - \P(\widehat{\bm{m}}_{\star}^n) \bigr) \bm{D}^n 
		+ \bigl( \P(\widehat{\bm{m}}_{\star}^n) - \P({\bm{m}}_{\star}^n) \bigr) \bm{D}^n \\
		\equiv {}& \text{\emph{ I $+$ II $+$ III $+$ IV}} .
		\end{align*}

		For the term \emph{IV} we have by Lemma~\ref{lemma:projection errors}
		\[
		\| \text{\emph{IV}} \|_{L^2} \leqslant 2 \| \widehat{\bm{m}}_{\star}^n - {\bm{m}}_{\star}^n \|_{L^2} \, \| \bm{D}^n \|_{L^\infty},
		\]
		where the last term $\widehat{\bm{m}}_{\star}^n - {\bm{m}}_{\star}^n$ has been bounded in the $L^2$ norm by $C\tau^k$ in the proof of Lemma~\ref{lemma:cons}. 
		
		The term \emph{III} is estimated using the first bound from Lemma~\ref{lem:gal}, under our regularity assumptions, as 
		\begin{equation*}
		\| \text{\emph{III}}\|_{L^2} \leqslant C h^r .
		\end{equation*}
		
		For the bound on \emph{II} we use Lemma~\ref{lem:diff} ($i$)  (with $p=2$ and $q=\infty$), to obtain
		\begin{equation*}
		\| \text{\emph{II}} \|_{L^2} \leqslant C_R \| \widehat{\bm{m}}_{\star,h}^n - \widehat{\bm{m}}_{\star}^n \|_{L^2}\|\bm{D}^n\|_{L^\infty} ,
		\end{equation*}
		where, using \eqref{ab-vectors}, we obtain 
		\[
		\| \widehat{\bm{m}}_{\star,h}^n - \widehat{\bm{m}}_{\star}^n \|_{L^2} \leqslant
		\frac{2 \| \sum_{i=1}^k \gamma_i (\RRRh-\mathbf{I}) \m_*^{n-i} \|_{L^2}}{\min \bigl| \sum_{i=1}^k \gamma_i \m_*^{n-i} \bigr|}
		\leqslant Ch^r.
		\]
		The denominator is bounded from below by $1 - C \tau^k$, because $|\m_*^n|=1$ and 
		$| \sum_{i=1}^k \gamma_i \m_*^{n-i} - \m_*^n | \leqslant C\tau^k$. 
		For the first term we have 
		\begin{align*}
		\| \text{\emph{I}} \|_{L^2} &{}\leqslant \|\bm{D}^n-\bm{D}_h^n\|_{L^2} 
		\\
		&{}\leqslant
		\alpha\|\partial_t\bm{m}(t_n)-\dot{\bm{m}}_{\star,h}^n\|_{L^2}
		+\|\bm{m}(t_n)\times \partial_t \bm{m}(t_n)-\widehat{\bm{m}}_{\star,h}^n\times \dot{\bm{m}}_{\star,h}^n\|_{L^2}.
		\end{align*} 
		The terms $\|\partial_t\bm{m}(t_n)-\dot{\bm{m}}_{\star}^n\|_{L^2}$ and $\|\bm{m}_{\star}^n\times \partial_t \bm{m}(t_n)
		-\widehat{\bm{m}}_{\star}^n\times \dot{\bm{m}}_{\star}^n\|_{L^2}$ can be handled as in the proof of Lemma~\ref{lemma:cons}.
		Standard error estimates for the Ritz projection $\RRh$ (we do not exploit the Aubin--Nitsche duality here) imply
		\begin{equation*}
		\|(\mathbf{I}-\RRRh)\dot{\bm{m}}_{\star}^n\|_{L^2}\leqslant c\, h^r \|\dot{\bm{m}}_\star^n\|_{H^{r+1}}.
		\end{equation*}
		Together this yields, under the stated regularity assumption,
		\[
		\| \text{\emph{I}} \|_{L^2} \leqslant C(\tau^k+h^r),
		\]
		and the result follows.
	\end{proof}

	\subsection{Error equation}
	
	We recall, from \eqref{BDF1-h}, the fully discrete problem with the linearly implicit BDF method: find $\dmh^n\in \T_h(\wmh^n)$ 
	such that for all $\bphih \in \T_h(\wmh^n)$,
	\begin{equation}
	\label{BDF1-h - recalled}
	\alpha ( \dmh^n,\bphih ) + (\wmh^n \times \dmh^n,\bphih) 
	+ (\nb \mh^n,\nb \bphih ) = ( \bm{H}(t_n),\bphih) .
	\end{equation}
	%
	
	Then, similarly as we have done in Section~\ref{Se:cont-perturb}, we first rewrite \eqref{BDF1str-full}: for all $\bphih \in \T_h(\wmh^n)$,
	\begin{equation}
	\label{eq:BDF1str-full - rewritten}
	\alpha ( \dmsh^n , \bphih) 
	+ ( \wmsh^n \times \dmsh^n , \bphih ) 
	+ ( \nb \msh^n , \nb \bphih) \\
	= (\brh^n,\bphih) 
	\end{equation}
	with
	\begin{equation}\label{rhn}
	\brh^n = - (\P_h(\wmh^n) - \P_h(\wmsh^n)) (\varDelta \ms(t_n) + \bm{H}(t_n)) + \bdh^n.
	\end{equation}
	
	The error $\beh^n = \mh^n - \msh^n$ satisfies the error equation that is obtained by subtracting \eqref{eq:BDF1str-full - rewritten} 
	from \eqref{BDF1-h - recalled}. We use the notations 
	\begin{align}
	\label{wehn}
	\weh^n &{}= \wmh^n - \wmsh^n , 
	\\
	\label{eq:s definition}
	\deh^n &{}= \dot {\bm m}_h^n - \dmsh^n = \frac1\tau \sum_{j=0}^k \delta_j \beh^{n-j} + \bs_h^n , \\ 
	\nonumber &{} \qquad\qquad \qquad \qquad \text{with }\quad 
	\bs_h^n=(\mathbf{I} - \P_h(\wmsh^n)) \frac1\tau \sum_{j=0}^k \delta_j \msh^{n-j}.
	\end{align}
	We have the following bound for $\bs_h^n$.
	
	\begin{lemma}\label{lemma:sn}
		Under the regularity assumptions of Lemma~\ref{lemma:cons-full}, we have
		\begin{equation}
		\label{sn2}
		\| \bs_h^n \|_{H^1(\Om)^3} \leqslant C(\tau^k+ h^r).
		\end{equation}
	\end{lemma}
	
	\begin{proof}
		We use Lemmas~\ref{lem:gal}  and~\ref{lem:stab}, and the bounds in the proof of 
		Lemma~\ref{lemma:cons-full}. We start by subtracting 
		$(\mathbf{I} - \P(\wmsh^n)) \partial_t \ms^{n} = 0$, and obtain (with $\pa^\tau \msh^n := \frac1\tau \sum_{j=0}^k \delta_j \msh^{n-j}$)
		\begin{align*}
		\bs_h^n 
		= {}& (\mathbf{I} - \P_h(\wmsh^n)) \pa^\tau \msh^n - (\mathbf{I} - \P(\wmsh^n)) \partial_t \ms^{n} \\
		= {}& (\pa^\tau \msh^n - \partial_t \ms^{n}) - \big( \P_h(\wmsh^n) \pa^\tau \msh^n - \P(\wmsh^n) \partial_t \ms^{n} \big) .
		\end{align*}
		The first term above is bounded as $O(\tau^k + h^r)$ via the techniques of the consistency proofs, Lemma~\ref{lemma:cons} 
		and \ref{lemma:cons-full}. For the second term we have 
		\begin{align*}
		&{} \P_h(\wmsh^n) \pa^\tau \msh^n - \P(\wmsh^n) \partial_t \ms^{n} \\
		&{} =  \P_h(\wmsh^n) (\pa^\tau \msh^n - \partial_t \ms^{n}) + \big( \P_h(\wmsh^n) - \P(\wmsh^n) \big) \partial_t \ms^{n} ,
		\end{align*} 
		where the first term is bounded as $O(\tau^k + h^r)$, using Lemma~\ref{lem:stab} and the previous estimate, while
		the second term is bounded as $O(h^r)$ by the $H^1$ estimate from Lemma~\ref{lem:gal}.
		Altogether, we obtain the stated $H^1$ bound for $\bs_h^n$.
	\end{proof}

	We then have the error equation
	\begin{equation}
	\label{eq:full error equation weak}
	\alpha (\deh^n,\bphih) + (\weh^n \times \dmsh^n,\bphih) +  (\wmh^n \times \deh^n,\bphih) + (\nb \beh^n ,\nb \bphih) = - (\brh^n,\bphih) ,
	\end{equation}
	for all $\bphih \in \T_h(\wmh^n)$, which is to be taken together with \eqref{rhn}--\eqref{eq:s definition}. 
	
	\section{Stability of the full discretization for BDF of orders 1 and 2}\label{Se:orders 1 and 2} 
	For the A-stable BDF methods (those of orders 1 and 2) we obtain the following stability estimate, which is analogous 
	to the continuous perturbation result Lemma~\ref{lemma:perturbation result}.
	\begin{lemma}[Stability for orders $k=1, 2$]
		\label{lemma:stability-full - BDF 1 and 2}
		Consider the  linearly implicit $k$-step BDF discretization~\eqref{BDF1-h} for $k \leqslant 2$ with finite elements of polynomial degree
		$r \geqslant 1$.  Let $\mh^n$ and $\msh^n=\RRRh\m(t_n)$ satisfy equations \eqref{BDF1-h} and \eqref{BDF1str-full}, respectively, 
		and suppose that the exact solution $\m(t)$ is bounded by \eqref{eq:assumptions on m star}  and  $\|\mathbf{H}(t)\|_{L^\infty}\leqslant M$
		for $0 \leqslant t \leqslant \bar t$. 
		%
		%
		Then, for sufficiently small $h\leqslant \bar h$ and $\tau\leqslant \bar\tau$, the error $\beh^n = \mh^n - \msh^n$ 
		satisfies the following bound, for $k\tau \leqslant n\tau \leqslant \bar t$,
		\begin{equation}
		\label{eq:stability bound full - k=1,2}
		\|\beh^n\|_{H^1(\Om)^3}^2
		\leqslant C \Big( \sum_{i=0}^{k-1} \|\beh^i\|_{H^1(\Om)^3}^2 
		+ \tau \sum_{j=k}^n \|\bdh^j\|_{L^2(\Om)^3}^2 + \tau \sum_{j=k}^n \|\bs_h^j\|_{H^1(\Om)^3}^2\Big) ,
		\end{equation}
		where the constant $C$ is independent of $h,\tau$ and $n$, but depends on $\alpha, R, K, M$, and~$\bar t$.
		This estimate holds under the smallness condition that the right-hand side is bounded by $\hat c h$ 
		with a sufficiently small constant $\hat c$ $($note that the right-hand side is of size
		$O((\tau^{k} + h^{r})^2)$ in the case of a sufficiently regular solution$)$.
	\end{lemma}
	
	Combining Lemmas~\ref{lemma:stability-full - BDF 1 and 2}, \ref{lemma:cons-full} and~\ref{lemma:sn} yields the 
	\emph{proof of Theorem~\ref{theorem:err-bdf-full - BDF 1 and 2}}: These lemmas imply the estimate 
	\begin{equation*}
	\|\beh^n\|_{H^1(\Om)^3} \leqslant \widetilde C (\tau^k + h^r) 
	\end{equation*} 
	in the case of a sufficiently regular solution. Since then $\|\RRRh \m(t_n) - \m(t_n) \|_{H^1(\Om)^3} \leqslant C h^r$ 
	and because of $\mh^n-\m(t_n) = \beh^n + (\RRRh \m(t_n) - \m(t_n))$, this implies the error bound \eqref{err-bdf12}.
	
	The smallness condition imposed in Lemma~\ref{lemma:stability-full - BDF 1 and 2} is satisfied under the very mild CFL condition,
	for a sufficiently small $\bar c > 0$ (independent of $h,\tau$ and $n$),
	\[		\tau^k \leqslant \bar c h^{1/2} .\]
	Taken together, this proves Theorem~\ref{theorem:err-bdf-full - BDF 1 and 2}.

	\begin{proof} (a) \emph{Preparations.}
		The proof of this lemma transfers the arguments of the proof of Lemma~\ref{lemma:perturbation result} to the fully discrete situation, using
		energy estimates obtained by testing with (essentially) the discrete time derivative of the error, as presented in the Appendix, 
		which is based on~Dahlquist's $G$-stability theory. 
		
		However, testing the error equation \eqref{eq:full error equation weak} directly with $\deh^n$ is not possible, 
		since $\deh^n$  is not in the tangent space $\T_h(\wmh^n)$.
		Therefore, as in the proof of Lemma~\ref{lemma:perturbation result}, we again start by showing that the test function 
		$\bphi_h = \P_h(\wmh^n) \deh^n \in \T_h(\wmh^n) \cap H^1(\varOmega)^3$ is a perturbation of $\deh^n$ itself:
		\begin{align*}
		\bphi_h = \P_h(\wmh^n) \deh^n = {}& \P_h(\wmh^n) \dot\m_h^n - \P_h(\wmh^n) \dmsh^n \\
		= {}&  \P_h(\wmh^n) \dot\m_h^n  -  \P_h(\wmsh^n) \dmsh^n + (\P_h(\wmsh^n) - \P_h(\wmh^n)) \dmsh^n .
		\end{align*}
		Here we note that $\P_h(\wmh^n) \dot\m_h^n = \dot\m_h^n \in T_h(\wmh^n)$ by construction of the method 
		\eqref{BDF1-h}, and $\P_h(\wmsh^n) \dmsh^n = \dmsh^n \in T_h(\wmsh^n)$ by the definition of $\dmsh^n$ 
		in \eqref{eq:extrapolation and discr derivative for exact solution}.
		So we have
		\[
		\bphi_h = \dot\m_h^n - \dmsh^n - (\P_h(\wmh^n) - \P(\wmsh^n)) \dmsh^n ,
		\]
		and hence
		\begin{equation}
		\label{eq:de_h^n perturbed - fully discrete}
		\bphi_h = \P_h(\wmh^n) \de^n = \deh^n + \bq_h^n 
		\qquad	 \textnormal{with } \quad \bq_h^n = - (\P_h(\wmh^n) - \P(\wmsh^n)) \dmsh^n.
		\end{equation}
		
		The proof now transfers the proof of the continuous perturbation result 
		Lemma~\ref{lemma:perturbation result} to the discrete situation with some notable differences, which are emphasized here: 
		
		(i) Instead of using the continuous quantities it uses their spatially discrete counterparts, in particular 
		the discrete projections $\P_h(\wmh^n)$ and $\P_h(\wmsh^n)$, defined and studied in Section~\ref{section:discrete orthogonal projection}. 
		In view of the definition \eqref{shn1} and \eqref{wmshn} of  $\wmh^n$ and $\wmsh^n$, respectively, this requires that $\sum_{j=0}^{k-1} \gamma_j \mh^{n-j-1}(x)$ and
		$\sum_{j=0}^{k-1} \gamma_j \msh^{n-j-1}(x)$
		are bounded away from zero uniformly for all $x\in\varOmega$.
		
		(ii) Instead of Lemma~\ref{lemma:projection errors} we use Lemma~\ref{lem:diff} 
		(with $\wmh^n$ and $\wmsh^n$ in the role of $\widetilde{\m}$ and $\m$, respectively) to bound the quantity 
		$\bm{q}_h^n$. This requires that 
		$\wmsh^n$ and $\dmsh^n$ are bounded in $W^{1,\infty}$ independently of $h$.

		\emph{Ad} (i):
		In order to show that $|\sum_{j=0}^{k-1} \gamma_j \mh^{n-j-1}(x)|$ stays close to $1$ for all $x\in\varOmega$, we need 
		to establish an $L^\infty$ bound for the errors  $\beh^{n-j-1} = \mh^{n-j-1} - \msh^{n-j-1}$. 
		
		We use an induction argument and assume that for some time step number $\bar n$ with $\bar n\tau\leqslant \bar t$ we have 
		\begin{equation}
		\label{eq:assumed bounds}
		\|\beh^n\|_{L^\infty} 
		\leqslant \rho,  \quad\  \text{ for } 0 \leqslant n< \bar n ,
		\end{equation}
		where we choose $\rho$ sufficiently small independent of $h$ and $\tau$. (In this proof it suffices to choose $\rho\leqslant 1/(4C_\gamma)$,
		where $C_\gamma = \sum_{j=0}^{k-1} |\gamma_j| = 2^k - 1$.)
		
		Note that the smallness condition of the lemma implies that \eqref{eq:assumed bounds} is satisfied for $\bar n=k$, 
		because for the $L^\infty$ errors of the starting values  we have by an inverse inequality, for $i=0,\dotsc,k-1$,
		\[\| \beh^i \|_{L^\infty} \leqslant Ch^{-1/2} \| \beh^i \|_{H^1} \leqslant Ch^{-1/2}\, (\hat c h)^{1/2} = C\hat c^{1/2} \leqslant \rho,\]
		provided that $\hat c$ is sufficiently small (independent of $\tau$ and $h$), as is assumed.
		
		We will show in part (b) of the proof that with the induction hypothesis \eqref{eq:assumed bounds} we obtain also 
		$\|\beh^{\bar n}\|_{L^\infty} \leqslant \rho$ so that finally we obtain \eqref{eq:assumed bounds} for \emph{all} $\bar n$ 
		with $\bar n \tau \leqslant \bar t$.
		
		Using  reverse and ordinary triangle inequalities,  the  error bound of \cite[Corollary~8.1.12]{BrennerScott} 
		(noting that $\m(t)\in W^{2,\infty}(\varOmega)$ under our assumptions) and the $L^\infty$ boundedness of $\partial_t \m$, 
		and the bound \eqref{eq:assumed bounds}, we estimate
		\begin{equation}
		\label{eq:extrapolation L infty control - pre}
		\begin{aligned}
		&\bigg\| \Big|\sum_{j=0}^{k-1} \gamma_j \mh^{n-j-1}\Big| - 1 \bigg\|_{L^\infty} \!\! =
		\bigg\| \Big|\sum_{j=0}^{k-1} \gamma_j \mh^{n-j-1}\Big| - |\ms^n| \bigg\|_{L^\infty} 
		\!\! \leqslant \bigg\| \sum_{j=0}^{k-1} \gamma_j \mh^{n-j-1} - \ms^n \bigg\|_{L^\infty} 
		\\
		\leq	&\ \Big\|\sum_{j=0}^{k-1} \gamma_j \beh^{n-j-1}\Big\|_{L^\infty}  \!
		+ \bigg\| \sum_{j=0}^{k-1} \gamma_j( \RRRh\ms^{n-j-1} \! -\ms^{n-j-1}) \bigg\|_{L^\infty}  \! + 
		\bigg\| \sum_{j=0}^{k-1} \gamma_j (\ms^{n-j-1} \!  - \ms^{n}) \bigg\|_{L^\infty} 
		\\
		\leqslant &\	\Big\|\sum_{j=0}^{k-1} \gamma_j \beh^{n-j-1}\Big\|_{L^\infty}  +Ch + C\tau
		\leqslant \sum_{j=0}^{k-1} |\gamma_j| \, \cdot \,  \rho  +Ch + C\tau  \leqslant \half,
		\end{aligned}
		\end{equation}
		provided that $h$ and $\tau$ are sufficiently small.
		The same argument also yields that $\bigl\| |\sum_{j=0}^{k-1} \gamma_j \msh^{n-j-1}| - 1 \bigr\|_{L^\infty} \leqslant \frac12$,
		and so we have
		\begin{equation}
		\label{eq:extrapolations L infty control}
		\half \leqslant \Big| \! \sum_{j=0}^{k-1} \! \gamma_j \mh^{n-j-1}(x)\Big| \leqslant \frac32 \andquad \half 
		\leqslant \Big| \! \sum_{j=0}^{k-1} \! \gamma_j \msh^{n-j-1}(x)\Big| \leqslant \frac32  
		\end{equation}
		for all $x \in \Om$. In particular, it follows that  $\wmh^n$  and  $\wmsh^n$ are unambiguously defined.
		
		\emph{Ad} (ii):
		The required $W^{1,\infty}$ bound for $\msh^n =\RRRh\m(t_n)$ follows from the $W^{1,\infty}$-stability of the Ritz projection: 
		by   \cite[Theorem~8.1.11]{BrennerScott} and by the assumed $W^{1,\infty}$ bound~\eqref{eq:assumptions on m star} for $\m(t)$,
		\begin{equation}\label{eq:msh^n W 1,infty bound}
		\| \msh^n \|_{W^{1,\infty}} \leqslant C \| \m(t_n) \|_{W^{1,\infty}} \leqslant CR.
		\end{equation}
		The bounds \eqref{eq:extrapolations L infty control} and~\eqref{eq:msh^n W 1,infty bound} for $n\leqslant \bar n$ imply that also
		\begin{equation}\label{eq:wmsh^n W 1,infty bound}
		\| \wmsh^n \|_{W^{1,\infty}}  \leqslant CR
		\end{equation}
		for $n\leqslant \bar n$ (with a different constant $C$).
		Using this bound in Lemma~\ref{lem:stab} and the assumed $W^{1,\infty}$ bound  \eqref{eq:assumptions on m star} for $\partial_t\m(t)$, 
		we obtain with
		$\delta(\zeta)/(1-\zeta)=\sum_{\ell=1}^k (1-\zeta)^{\ell-1} /\ell =: \sum_{j=0}^{k-1} \mu_j \zeta^j$ that
		\begin{align*}
		\| \dmsh^n \|_{W^{1,\infty}} &= \| \P_h (\wmsh^n) \frac1\tau \sum_{j=0}^k \delta_j \ms^{n-j} \|_{W^{1,\infty}}
		\\
		&=  \| \P_h (\wmsh^n) \sum_{j=0}^{k-1} \mu_j \frac1\tau (\ms^{n-j} - \ms^{n-j-1})\|_{W^{1,\infty}}
		\\
		&= \| \P_h (\wmsh^n) \sum_{j=0}^{k-1} \mu_j \frac1\tau \int_{t_{n-j-1}}^{t_{n-j}} \partial_t \m(t) \,\d t \|_{W^{1,\infty}}
		\\
		&\leqslant C_R  \, \|  \sum_{j=0}^{k-1} \mu_j \frac1\tau \int_{t_{n-j-1}}^{t_{n-j}} \partial_t \m(t) \,\d t \|_{W^{1,\infty}}
		\\
		&\leqslant C_R \sum_{j=0}^{k-1} |\mu_j|\, R.
		\end{align*}
		We can now establish a bound for $\bqh^n$ as defined in \eqref{eq:de_h^n perturbed - fully discrete}, using Lemma~\ref{lem:diff}  
		together with the above $W^{1,\infty}$ bounds for $\wmsh^n$ and $\dmsh^n$ to obtain
		\begin{equation}
		\label{eq:qhn bound}
		\|\bqh^n\|_{L^2} \leqslant c \|\weh^n\|_{L^2} \andquad \|\nb \bqh^n\|_{L^2} \leqslant c \|\weh^n\|_{H^1} .
		\end{equation}
		With the $W^{1,\infty}$ bound of $\wmsh^n$ we also obtain a bound of $\brh^n$ defined in
		\eqref{rhn}. Using Lemma~\ref{lem:diff} ($i$) and recalling the $L^\infty$ bound of
		$\varDelta \m + \bm H$
		of \eqref{eq:assumptions on m star}, we find that $\brh^n$ is bounded by
		\begin{equation}
		\label{eq:rhn bound}
		\begin{aligned}
		\|\brh^n\|_{L^2} 
		\leqslant {}& \|(\P_h(\wmh^n) - \P_h(\wmsh^n)) ( \varDelta \ms^n + \bm H^n)\|_{L^2} + \|\bd_h^n \|_{L^2} \\
		\leqslant {}& c \|\weh^n\|_{L^2} + \|\bdh^n \|_{L^2} .
		\end{aligned}
		\end{equation}

		\medskip
		
		%
		%
		
		
		(b) \emph{Energy estimates.} For $n \leqslant \bar n$ with $\bar n$ of \eqref{eq:assumed bounds}, 
		we test the error equation \eqref{eq:full error equation weak}  with $\bphi_h = \deh^n + \bqh^n$ and obtain
		\begin{align*}
		\alpha (\deh^n,\deh^n + \bqh^n) 
		+ (\weh^n \times \dmsh^n,\deh^n + \bqh^n) 
		{}& + (\wmh^n \times \deh^n,\deh^n + \bqh^n) \\
		{}& +  (\nb \beh^n,\nb (\deh^n + \bqh^n)) = - (\brh^n,\deh^n + \bqh^n) .
		\end{align*}
		%
		%
		By collecting the terms, and using the fact that $(\wmh^n \times \deh^n,\deh^n)=0$, we altogether obtain
		\begin{align*}
		\alpha \|\deh^n\|_{L^2}^2 + (\nb \beh^n,\nb \deh^n) 
		{}& = - \alpha (\deh^n,\bqh^n) - (\weh^n \times \dmsh^n,\deh^n + \bqh^n) \\
		{}& - (\wmh^n \times \deh^n,\bqh^n) - (\nb \beh^n,\nb \bqh^n) 
		- (\brh^n,\deh^n + \bqh^n) .
		\end{align*}
		We now estimate the term $(\nb \beh^n,\nb \deh^n)$ on the left-hand side from below using Dahlquist's Lemma~\ref{lemma:Dahlquist}, 
		so that the ensuing relation \eqref{multiplier} yields
		\begin{equation*}
		(\nb \beh^n,\nb \deh^n) \geqslant \frac{1}{\tau} \Big(\|\nb \bfE_h^n\|_G^2 - \|\nb \bfE_h^{n-1}\|_G^2 \Big)  + (\nb \beh^n, \nb \bs_h^n),
		\end{equation*}
		where $\bfE_h^n = (\beh^{n-k+1},\dotsc,\beh^n)$ and the $G$-weighted semi-norm is given by
		\begin{equation*}
		\|\nb \bfE_h^n\|_G^2 
		= \sum_{i,j=1}^kg_{ij}(\nb \beh^{n-k+i},\nb \beh^{n-k+j}) .
		\end{equation*}
		This semi-norm satisfies the relation
		\begin{equation}
		\label{eq:G norm equivalence}
		\gamma^- \sum_{j=1}^k \|\nb \beh^{n-k+j}\|_{L^2}^2 
		\leqslant \|\nb \bfE_h^n\|_G^2
		\leqslant \gamma^+ \sum_{j=1}^k \|\nb \beh^{n-k+j}\|_{L^2}^2 ,
		\end{equation}
		where $\gamma^-$ and $\gamma^+$ are the smallest and largest eigenvalues 
		of the positive definite symmetric matrix $G=(g_{ij})$ from Lemma~\ref{lemma:Dahlquist}. 
		
		The remaining terms are estimated using the Cauchy--Schwarz inequality and $\|\wmh^n\|_{L^\infty}=1$; we altogether obtain
		\begin{align*}
		&\alpha \|\deh^n\|_{L^2}^2 +  \frac{1}{\tau} \Big( \! \|\nb \bfE_h^n\|_G^2 - \|\nb \bfE_h^{n-1}\|_G^2 \! \Big) 
		\leqslant 
		\alpha \|\deh^n\|_{L^2} \|\bqh^n\|_{L^2} 
		+ \|\weh^n\|_{L^2} ( \|\deh^n\|_{L^2} + \|\bqh^n\|_{L^2} ) \\
		&\quad  + \|\deh^n\|_{L^2} \|\bqh^n\|_{L^2} 
		+  \|\nb \beh^n\|_{L^2}( \|\nb \bq^n\|_{L^2} + \|\nb \bs_h^n\|_{L^2})
		+ \|\brh^n\|_{L^2} ( \|\deh^n\|_{L^2} + \|\bqh^n\|_{L^2} ) .
		\end{align*}
		We now show an $L^2$ error bound for $\weh^n$ in terms of $(\beh^{n-j-1})_{j=0}^{k-1}$. Using the fact that for $\bm{a},\bm{b}\in\R^3\setminus\{0\}$,
		\begin{equation}
		\label{ab-vectors}
		\left| \frac{\bm{a}}{|\bm{a}|} - \frac{\bm{b}}{|\bm{b}|} \right| = \left| \frac{(|\bm{b}|-|\bm{a}|) \bm{a} 
			+ |\bm{a}| (\bm{a}-\bm{b})}{|\bm{a}| \ |\bm{b}|} \right| \leqslant 2\, \frac{|\bm{a}-\bm{b}|}{|\bm{b}|},
		\end{equation}
		and the lower bounds in \eqref{eq:extrapolations L infty control} for both $| \sum_{j=0}^{k-1}\gamma_j\mh^{n-j-1}|$ and $| \sum_{j=0}^{k-1}\gamma_j\msh^{n-j-1} |$,
		we can estimate
		\begin{equation}
		\label{eq:hat e L2 estimate}
		\|\weh^n\|_{L^2} = \left\| \frac{\sum_{j=0}^{k-1}\gamma_j\mh^{n-j-1}} 
		{\Big |\sum_{j=0}^{k-1}\gamma_j\mh^{n-j-1}\Big |} -
		\frac{\sum_{j=0}^{k-1}\gamma_j\msh^{n-j-1}} 
		{\Big |\sum_{j=0}^{k-1}\gamma_j\msh^{n-j-1}\Big |} \right\|_{L^2}
		\leqslant C \sum_{j=0}^{k-1} \|\beh^{n-j-1}\|_{L^2}^2 .
		\end{equation}
		To show a similar bound for $\|\nb\weh^n\|_{L^2}$ we need the following two observations: First, 
		the $W^{1,\infty}$ bounds for $\msh^{n-j-1}$ from \eqref{eq:msh^n W 1,infty bound} imply $W^{1,\infty}$ boundedness for $\wmsh^n$ by
		\begin{align*}
		\left| \pa_j \left( \frac{\bm{b}}{|\bm{b}|} \right) \right| \leqslant \left|  \frac{\pa_j \bm{b}}{|\bm{b}|} \right| 
		+ \left| \frac{\bm{b}(\pa_j \bm{b},\bm{b})}{|\bm{b}|^3} \right| .
		\end{align*}
		Second, similarly  we have 
		\begin{align*}
		\left| \pa_j \left( \frac{\bm{a}}{|\bm{a}|} - \frac{\bm{b}}{|\bm{b}|} \right) \right| 
		\leqslant &\ \left| \frac{\pa_j \bm{a}}{|\bm{a}|} - \frac{\pa_j \bm{b}}{|\bm{b}|} \right| 
		+ \left| \frac{\bm{a} (\pa_j\bm{a},\bm{a})|\bm{b}|^3 - \bm{b} (\pa_j\bm{b},\bm{b})|\bm{a}|^3}{|\bm{a}|^3 \ |\bm{b}|^3} \right| \\
		\leqslant &\ \left| \frac{\pa_j \bm{a}}{|\bm{a}|} - \frac{\pa_j \bm{b}}{|\bm{b}|} \right| 
		+ \frac{||\bm{a}|^3 - |\bm{b}|^3| |\pa_j\bm{b}|}{|\bm{a}|^3 \ |\bm{b}|}   
		+ \frac{| \bm{a} (\pa_j\bm{a},\bm{a}) - \bm{b} (\pa_j\bm{b},\bm{b}) | }{|\bm{b}|^3} \\
		\leqslant &\ \left| \frac{\pa_j \bm{a}}{|\bm{a}|} - \frac{\pa_j \bm{b}}{|\bm{b}|} \right| 
		+ \frac{|\bm{a} - \bm{b}| (|\bm{b}|^2 + |\bm{b}| |\bm{a}| + |\bm{a}|^2) |\pa_j\bm{b}|}{|\bm{a}|^3 \ |\bm{b}|} \\
		&\ + \frac{|\bm{a}|^2 |\pa_j\bm{a} - \pa_j\bm{b}|}{|\bm{b}|^3} 
		+ \frac{|\bm{a}| |\pa_j\bm{b}| |\bm{a} - \bm{b}|}{|\bm{b}|^3}
		+ \frac{|\bm{a} - \bm{b}| |\pa_j\bm{b}|}{|\bm{b}|^2} .
		\end{align*}
		Combining these two observations, again with $\mh$ and $\msh$ in the role of $\bm{a}$ and $\bm{b}$, 
		respectively, and  the upper and lower bounds from \eqref{eq:extrapolations L infty control} altogether yield
		\begin{equation}
		\label{eq:hat e H1 estimate}
		\|\nb \weh^n\|_{L^2}^2 \leqslant C \sum_{j=0}^{k-1} \|\beh^{n-j-1}\|_{H^1}^2 .
		\end{equation}
		
		We estimate further using Young's inequality and absorptions into the term $\|\de^n\|_{L^2}^2$, 
		together with the bounds in \eqref{eq:qhn bound} and \eqref{eq:rhn bound}, to obtain
		\begin{equation*}
		\alpha \frac12 \|\deh^n\|_{L^2}^2 + \frac{1}{\tau} \Big(\|\nb \bfE_h^n\|_G^2 - \|\nb \bfE_h^{n-1}\|_G^2 \Big) 
		\leqslant  c \sum_{j=0}^{k} \|\beh^{n-j}\|_{H^1}^2 + c \|\bdh^n\|_{L^2}^2 + c \|\nb\bs_h^n\|_{L^2}^2.
		\end{equation*}
		Multiplying both sides by $\tau$, summing up from $k$ to $n \leqslant \bar n$, and using an absorption yield
		\begin{align*}
		&\ \alpha \frac12 \tau \sum_{j=k}^n \|\deh^j\|_{L^2}^2 +  \|\nb \bfE_h^n\|_G^2 \\
		\leqslant &\ \|\nb \bfE_h^{k-1}\|_G^2 + c \tau \sum_{j=k}^n \|\beh^j\|_{H^1}^2 
		+ c \tau \sum_{j=k}^n \big( \|\bdh^j\|_{L^2}^2 +  \|\bs_h^j\|_{H^1}^2 \big) + c \sum_{i=0}^{k-1} \|\beh^i\|_{L^2}^2 .
		\end{align*}
		We then arrive, using \eqref{eq:G norm equivalence}, at
		\begin{equation}
		\label{stab17}
		\begin{aligned}
		\alpha \frac12 \tau \sum_{j=k}^n \|\deh^j\|_{L^2}^2 +  \|\nb \beh^n\|_{L^2}^2 
		&{}\leqslant  c \tau \sum_{j=k}^n \|\beh^j\|_{H^1}^2 
		+ c \tau \sum_{j=k}^n \! \big( \|\bdh^j\|_{L^2}^2 +  \|\bs_h^j\|_{H^1}^2 \big)\\ 
		&{}+ c \sum_{i=0}^{k-1} \|\beh^i\|_{L^2}^2 ,
		\end{aligned}	
		\end{equation}
		with  $c$ depending on $\alpha$.
		
		Similarly as in the time continuous case in the proof of Lemma~\ref{lemma:perturbation result}, we connect 
		$\|\beh^n\|_{L^2}^2$ and $\tau \sum_{j=k}^n \|\deh^j\|_{L^2}^2$. We rewrite the identity
		\[
		\frac1\tau \sum_{j=0}^k \delta_j \beh^{n-j} = \deh^n - \bs_h^n, \quad n \geqslant k,
		\]
		as 
		\[
		\frac1\tau \sum_{j=k}^n \delta_{n-j} \beh^j = \dot\beh^n - \bs_h^n -\bm{g}_h^n, \quad n \geqslant k,
		\]
		with $\delta_\ell=0$ for $\ell>k$ and where
		\[
		\bm{g}_h^n := \frac1\tau \sum_{i=0}^{k-1} \delta_{n-i} \beh^i
		\]
		depends only on the starting errors and satisfies $\bm{g}_h^n=0$ for $n\geqslant 2k$.
		With the inverse power series of $\delta(\zeta)$,
		\begin{equation*}
		\kappa(\zeta) = \sum_{n=0}^\infty \kappa_n \zeta^n := \frac 1{\delta(\zeta)} ,
		\end{equation*}
		we then have, for $n \geqslant k$,
		\begin{equation*}
		\beh^n = \tau \sum_{j=k}^n \kappa_{n-j} (\deh^j -\bs_h^j -\bm{g}_h^j).
		\end{equation*}
		By the zero-stability of the BDF method of order $k \leqslant 6$, the coefficients $\kappa_n$ are uniformly bounded: 
		$|\kappa_n| \leqslant c$ for all $n\geqslant 0$. Therefore we obtain via the Cauchy--Schwarz inequality
		\begin{align*}
		\| \beh^n \|_{L^2}^2 &{} \leqslant 2\tau^2 \Bigl\| \sum_{j=k}^n \kappa_{n-j} (\dot \beh_j -\bs_h^j)\Bigr\|_{L^2}^2 + 
		2\tau^2 \Bigl\| \sum_{j=k}^{2k-1} \kappa_{n-j} \bm{g}_h^j \Bigr\|_{L^2}^2\\
		&{}\leqslant (2n\tau)\tau c^2 \sum_{j=k}^n \| \dot \beh^j- \bs_h^j \|_{L^2}^2 + 2\tau^2 c^2 k \sum_{j=k}^{2k-1} \| \bm{g}_h^j \|_{L^2}^2\\
		&{} \leqslant C\tau \sum_{j=k}^n \| \deh^j \|_{L^2}^2 + C \tau \sum_{j=k}^n \|\bs_h^j \|_{L^2}^2 + C \sum_{i=0}^k \| \beh^i \|_{L^2}^2.
		\end{align*}
		Inserting this bound into \eqref{stab17} then yields
		\begin{equation*}
		\alpha \|\beh^n\|_{L^2}^2 +  \|\nb \beh^n\|_{L^2}^2
		\leqslant  c \tau \sum_{j=k}^n \|\beh^j\|_{H^1}^2 
		+ c \tau \! \sum_{j=k}^n \! \big( \|\bdh^j\|_{L^2}^2 +  \|\bs_h^j\|_{H^1}^2 \big) + c \! \sum_{i=0}^{k-1} \|\beh^i\|_{L^2}^2 ,
		\end{equation*}
		and  a discrete Gronwall inequality  implies the stated stability result for $n \leqslant \bar n$. It then follows from this stability bound, 
		the smallness condition of the lemma  and the inverse estimate from $H^1$ to 
		$L^\infty$ that \eqref{eq:assumed bounds} is satisfied also for $\bar n+1$. This completes the induction step for \eqref{eq:assumed bounds} 
		and proves the stated error bound.
	\end{proof}

	\section{Stability of the full discretization for BDF of orders 3 to 5}\label{Se:orders 3 to 5}

	Stability for full discretizations using the BDF methods of orders $3$ to $5$ can be shown 
	under additional conditions on the damping parameter $\alpha$ and the stepsize $\tau$.
	
	\begin{lemma}[Stability for orders $k=3,4,5$]
		\label{lemma:stability-full - BDF 3,4,5}
		Consider the linearly implicit $k$-step BDF discretization~\eqref{BDF1-h} for $3 \leqslant k \leqslant 5$ 
		with finite elements of polynomial degree $r \geqslant 2$.
		Let $\mh^n$ and $\msh^n$ satisfy \eqref{BDF1-h} and \eqref{BDF1str-full}, respectively, and suppose 
		that the regularity assumptions of Lemma~\ref{lemma:stability-full - BDF 1 and 2} hold. 
		Furthermore, assume that the damping parameter $\alpha$ satisfies
		\begin{equation}\label{alpha-eta}
		\alpha > \alpha_k := \frac{\eta_k}{1-\eta_k}
		\end{equation}
		with the multiplier $\eta_k$ of Lemma~\ref{lemma:NO},
		and that $\tau $ and $h $ satisfy the mild CFL-type condition, for some $\bar c>0$,
		\begin{equation}
		\label{eq:step size restrictions}
		\tau \leqslant \bar c h.
		\end{equation}
		Then, for sufficiently small $h\leqslant \bar h$ and $\tau\leqslant \bar\tau$, the error $\beh^n = \mh^n - \msh^n$ 
		satisfies the following bound, for $k\tau \leqslant n\tau \leqslant \bar t$,
		\begin{equation}
		\label{eq:stability bound full - k=3,4,5}
		\|\beh^n\|_{H^1(\Om)^3}^2
		\leqslant C \Big( \sum_{i=0}^{k-1} \|\beh^i\|_{H^1(\Om)^3}^2 
		+ \tau \sum_{j=k}^n \|\bdh^j\|_{L^2(\Om)^3}^2+ \tau \sum_{j=k}^n \|\bs_h^j\|_{H^1(\Om)^3}^2\Big) ,
		\end{equation}
		where the constant $C$ is independent of $\tau, h$ and $n$, but depends on $\alpha, R, K, M$,  and exponentially on~$\bar c\bar t$.
		This estimate holds under the smallness condition that the right-hand side is bounded by $\hat c h^3$ 
		with a constant $\hat c$ $($note that the right-hand side is of size
		$O((\tau^{k} + h^{r})^2)$ in the case of a sufficiently regular solution$)$.
	\end{lemma}
	
	Together with the defect bounds of Section~\ref{Se:full-discr}, this stability lemma proves Theorem~\ref{theorem:err-bdf-full - BDF 3+}.
	We remark that the thresholds $\alpha_k>0$ defined here are the same as those appearing in Theorem~\ref{theorem:err-bdf-full - BDF 3+}.
	
	\begin{proof}
		The proof of this lemma combines the arguments of the proof of Lemma~\ref{lemma:stability-full - BDF 1 and 2} 
		with a nonstandard variant of the multiplier technique of Nevanlinna and Odeh, as outlined in the Appendix. 
		Since the size of the parameter $\alpha$ determines which BDF methods satisfy the stability estimate, the dependence on $\alpha$ 
		will be carefully traced all along the proof.
		%

		
		(a) \emph{Preparations.}
		As in the previous proof, we make again the induction hypothesis \eqref{eq:assumed bounds} for some $\bar n$ with $\bar n\tau \leqslant \bar t$, 
		but this time with $\rho=c_0 h$ for some positive constant $c_0$:
		\begin{equation}
		\label{eq:assumed bounds - h}
		\|\beh^n\|_{L^\infty} 
		\leqslant c_0 h,  \qquad n < \bar n .
		\end{equation}
		By an inverse inequality, this implies  that $\|\beh^n\|_{W^{1,\infty}}$ has an $h$- and $\tau$-independent bound, and hence also 
		$\|\mh^n\|_{W^{1,\infty}}$ for $n<\bar n$. Together with \eqref{eq:extrapolations L infty control}, this implies 
		\begin{equation}\label{mhn}
		\|\wmh^n\|_{W^{1,\infty}} \leqslant C
		\end{equation}
		and further
		\begin{equation}\label{wehn2}
		\|\weh^n\|_{L^\infty} \leqslant Ch.
		\end{equation}

		As in the Appendix, we aim to subtract $\eta_k$ times the error equation for time step $n-1$ 
		from the error equation for time 
		step  $n$, and then to test with $\bphih = \P_h(\wmh^n) \de^n_h \in \T_h(\wmh^n)$ 
		(similarly as in the proof of Lemma~\ref{lemma:stability-full - BDF 1 and 2}). 
		%
		%
		However, this is not possible directly due to the different test spaces at different time steps: 
		%
		\begin{subequations}
			\begin{align}
			\label{eq:error eqn - n}
			{}& \begin{aligned}
			{}& \alpha (\deh^n,\bphih) + (\weh^n \times \dmsh^n,\bphih) 
			\\ {}& 
			+ (\wmh^n \times \deh^n,\bphih) + (\nb \beh^n ,\nb \bphih) = - (\brh^n,\bphih) ,
			\end{aligned} 
			\intertext{for all $\bphih \in \T_h(\wmh^n)$, and }
			\label{eq:error eqn - (n-1)}
			{}& \begin{aligned}
			{}& \alpha (\deh^{n-1},\bpsih) + (\weh^{n-1} \times \dmsh^{n-1},\bpsih) \\
			{}& + (\wmh^{n-1} \times \deh^{n-1},\bpsih) + (\nb \beh^{n-1} ,\nb \bpsih)
			=- (\brh^{n-1},\bpsih) ,
			\end{aligned} 
			\end{align}
		\end{subequations}
		for all $\bpsih \in \T_h(\wmh^{n-1})$. 
		
		As in \eqref{eq:de_h^n perturbed - fully discrete}, we have
		\begin{equation}
		\label{eq:test function dot e_h fully discrete}
		\bphih = \P_h(\wmh^n) \deh^n = \deh^n + \bqh^n , \quad \textnormal{ with } \quad \bqh^n = - (\P_h(\wmh^n) - \P_h(\wmsh^n)) \dmsh^n ,
		\end{equation}
		where $\bqh^n$ is bounded by \eqref{eq:qhn bound}.
		
		In turn, the test function $\bpsih = \P_h(\wmh^{n-1})\deh^n \in \T_h(\wmh^{n-1})$ is a perturbation of $\bphih = \deh^n + \bqh^n$, 
		since using \eqref{eq:test function dot e_h fully discrete} we obtain
		\begin{align*}
		\bpsih {}& = \P_h(\wmh^{n-1})\deh^n \\
		{}& = \P_h(\wmh^{n})\deh^n - (\P_h(\wmh^{n}) -\P_h(\wmh^{n-1})) \deh^n \\
		{}& = \deh^n + \bqh^n + \bph^n \quad\text{ with }\quad \bph^n=- ( \P_h(\wmh^n) - \P_h(\wmh^{n-1}) ) \deh^n.
		\end{align*}
		The perturbation $\bph^n$ is estimated using the second bound in Lemma~\ref{lem:diff} ($i$) with $p=\infty$, $q=2$, and noting \eqref{mhn}. 
		We obtain
		\begin{equation*}
		\begin{aligned}
		\|\bph^n\|_{L^2}  
		\leqslant {}& \|( \P_h(\wmh^n) - \P_h(\wmh^{n-1}) ) \deh^n\|_{L^2} \\
		\leqslant {}& c \|\deh^n\|_{L^2} \|\wmh^n -\wmh^{n-1}\|_{L^\infty} \\
		\leqslant {}& c \|\deh^n\|_{L^2} \Big( \|\weh^n\|_{L^\infty} + \|\wmsh^n - \wmsh^{n-1}\|_{L^\infty}  + \|\weh^{n-1}\|_{L^\infty} \Big) \\
		\leqslant {}& c \|\deh^n\|_{L^2} \Big( \|\weh^n\|_{L^\infty} + 
		\sum_{j=0}^{k-1} |\gamma_j| \int_{t_{n-j-2}}^{t_{n-j-1}} \| \RRRh \partial_t \m(t) \|_{L^\infty} \,\d t
		+ \|\weh^{n-1}\|_{L^\infty} \Big) .
		\end{aligned}
		\end{equation*}
		We have $\|\mathbf{R}_h \partial_t \m(t) \|_{L^\infty}\leqslant c \| \partial_t \m(t) \|_{W^{1,\infty}} $ by \cite[Theorem~8.1.11]{BrennerScott}. 
		In view of \eqref{wehn2} we obtain, for $\tau\leqslant \bar C h$,
		\begin{equation}
		\label{eq:bound for p^n - alternative}
		\|\bph^n\|_{L^2}  \leqslant Ch \|\deh^n\|_{L^2},
		\end{equation}
		and by an inverse estimate,
		\begin{equation}
		\label{eq:bound for p^n - gradient}
		\|\nabla \bph^n\|_{L^2}  \leqslant C \|\deh^n\|_{L^2}.
		\end{equation}
		We also recall the bound \eqref{eq:rhn bound} for $\|\brh^n\|_{L^2}$.
		
		\medskip
		
		(b) \emph{Energy estimates.} By subtracting \eqref{eq:error eqn - n}$-\eta_k$\eqref{eq:error eqn - (n-1)} with 
		the above choice of test functions, we obtain
		\begin{equation}
		\label{eq:error eqn fully discrete}
		\begin{aligned}
		& \alpha (\deh^n - \eta_k \deh^{n-1},\deh^n + \bqh^n) 
		+ (\weh^n \times \dmsh^n - \eta_k \weh^{n-1} \times \dmsh^{n-1},\deh^n + \bqh^n) \\
		& + (\wmh^n \times \deh^n - \eta_k \wmh^{n-1} \! \times \! \deh^{n-1},\deh^n + \bqh^n) 
		+  (\nb \beh^n  - \eta_k \nb \beh^{n-1} ,\nb (\deh^n + \bqh^n)) \\
		& - \eta_k \big[ \alpha (\deh^{n-1},\bph^n) 
		+ (\weh^{n-1} \times \dmsh^{n-1},\bph^n) 
		\\ & \phantom{ \ - \eta_k \big[ }
		+ (\wmh^{n-1} \times \deh^{n-1},\bph^n) 
		+  (\nb \beh^{n-1} ,\nb \bph^n) \big] \\
		= & - (\brh^n - \eta_k\brh^{n-1} ,\deh^n + \bqh^n) - \eta_k (\brh^{n-1},\bph^n).
		\end{aligned}
		\end{equation}
		%
		We estimate the terms of the error equation \eqref{eq:error eqn fully discrete} separately and track carefully the 
		dependence on $\eta_k$ and $\alpha$.
		
		The term $\alpha (\deh^n - \eta_k \deh^{n-1},\deh^n)$ is bounded from below, using Young's inequality and absorptions, by
		\begin{equation*}
		\alpha (\deh^n - \eta_k \deh^{n-1},\deh^n) \geqslant \alpha \big( 1-\tfrac12\eta_k \big) \|\deh^n\|_{L^2}^2 - \tfrac{\alpha}{2}\eta_k \|\deh^{n-1}\|_{L^2}^2,
		\end{equation*} 
		while the term $ (\nb \beh^n  - \eta_k \nb \beh^{n-1} ,\nb \deh^n)$ is bounded from below, via the relation \eqref{multiplier} and \eqref{eq:s definition}, by
		\begin{equation*}
		(\nb \beh^n  - \eta_k \nb \beh^{n-1} ,\nb \deh^n) \geqslant  \frac{1}{\tau} \Big(\|\nb \bfEh^n\|_G^2 - \|\nb \bfEh^{n-1}\|_G^2 \Big) 
		+ (\nb \beh^n  - \eta_k \nb \beh^{n-1} ,\nb \bs_h^n),
		\end{equation*}
		with $\bfEh^n = (\beh^{n-k+1},\dotsc,\beh^n)$, and where the $G$-weighted semi-norm is 
		generated by the matrix $G=(g_{ij})$ from Lemma~\ref{lemma:Dahlquist} for the rational function $\delta(\zeta)/(1-\eta_k\zeta)$.
		
		The remaining terms outside the rectangular bracket are estimated using the Cauchy--Schwarz and Young inequalities 
		(the latter often with a sufficiently small but fixed $h$- and $\tau$-independent weighting factor $\mu>0$) and $\|\wmh^n\|_{L^\infty}=1$ 
		and orthogonality. We obtain, with varying constants $c$ (which depend on $\alpha$ and are inversely proportional to $\mu$)
		\begin{align*}
		{}& \alpha (\deh^n - \eta_k \deh^{n-1},\bqh^n) + (\weh^n \times \dmsh^n - \eta_k \weh^{n-1} \times \dmsh^{n-1},\deh^n + \bqh^n) \\
		{}& \quad + (\wmh^n \times \deh^n - \eta_k \wmh^{n-1} \times \deh^{n-1},\deh^n + \bqh^n) + 
		( \nb \be^n - \eta_k \nb \beh^{n-1} ,\nb \bqh^n) \\
		{}& \leqslant \bigl( \alpha \mu + \mu + \tfrac12 \eta_k \bigr) \|\deh^n\|_{L^2}^2 
		+ \bigl( \alpha\mu \eta_k +  \tfrac12 \eta_k \bigr) \|\deh^{n-1}\|_{L^2}^2
		\\
		{}& \quad + c \bigl(  \| \bqh^n\|_{L^2} + \|\weh^n\|_{L^2}^2 + \|\weh^{n-1}\|_{L^2}^2 \bigr)
		+\tfrac12 \bigl (\|\nb \beh^n\|_{L^2}^2 +\eta_k^2 \|\nb \beh^{n-1}\|_{L^2}^2 +  \| \nb\bqh^n\|_{L^2} \bigr)
		\\
		{}& \leqslant \bigl( \alpha \mu + \mu + \tfrac12 \eta_k \bigr) \|\deh^n\|_{L^2}^2 
		+ \bigl( \alpha\mu \eta_k +  \tfrac12 \eta_k \bigr) \|\deh^{n-1}\|_{L^2}^2
		+ c \sum_{j=0}^k \| \beh^{n-j-1} \|_{H^1}^2 ,
		%
		\end{align*}
		where in the last inequality we used \eqref{eq:hat e L2 estimate} and \eqref{eq:hat e H1 estimate} to estimate $\weh^n$.
		
		The terms inside the rectangular bracket are bounded similarly, using \eqref{eq:bound for p^n - alternative} and \eqref{eq:bound for p^n - gradient}
		and the condition $\tau\leqslant \bar C h$, by 
		\begin{align*}
		{}&  \alpha (\deh^{n-1},\bph^n) + (\weh^{n-1} \times \dmsh^{n-1},\bph^n) 
		+ (\wmh^{n-1} \times \deh^{n-1},\bph^n) 
		+  (\nb \beh^{n-1} ,\nb \bph^n) 
		\\
		{}& \leqslant
		\mu \| \deh^n \|_{L^2}^2 + ch \| \deh^{n-1} \|_{L^2}^2
		+c \bigl( \|\weh^{n-1}\|_{L^2}^2 +  \|\nb \beh^{n-1}\|_{L^2}^2 \bigr)
		\\
		{}& \leqslant
		\mu  \| \deh^n \|_{L^2}^2 + c \sum_{j=0}^k \| \beh^{n-j-1} \|_{H^1}^2.
		\end{align*}
		%
		Here $\mu$ is an arbitrarily small positive constant (independent of $\tau$ and $h$), and $c$ depends on the choice of $\mu$.
		
		In view of \eqref{eq:rhn bound}, the terms with the defects $\brh^n$ are bounded by
		\begin{align*}
		{}& 
		- (\brh^n - \eta_k\brh^{n-1} ,\deh^n + \bqh^n) - \eta_k (\brh^{n-1},\bph^n) 
		\\
		{}& \leqslant
		\mu  \| \deh^n \|_{L^2}^2 + c \bigl( \| \brh^n \|_{L^2}^2 +  \| \brh^{n-1} \|_{L^2}^2 +  \| \bqh^n \|_{L^2}^2 \bigr)
		\\
		{}& \leqslant
		\mu \| \deh^n \|_{L^2}^2 + c \sum_{j=0}^k \| \beh^{n-j-1} \|_{L^2}^2 + c \sum_{j=0}^1 \| \bdh^{n-j} \|_{L^2}^2.
		\end{align*}
		Combination of these inequalities yields
		\begin{align*}
		&\Bigl( \alpha(1-\tfrac12\eta_k) -\tfrac12\eta_k -\mu \Bigr)  \| \deh^n \|_{L^2}^2 -
		\Bigl( \tfrac\alpha 2 \eta_k + \tfrac12\eta_k +\mu\alpha\eta_k \Bigr) \|\deh^{n-1}\|_{L^2}^2
		\\
		&\quad +  \frac{1}{\tau} \Big(\|\nb \bfEh^n\|_G^2 - \|\nb \bfEh^{n-1}\|_G^2 \Big) 
		\\
		&\leqslant c \sum_{j=0}^k \| \beh^{n-j-1} \|_{H^1}^2 + c \sum_{j=0}^1 \| \bdh^{n-j} \|_{L^2}^2 + c \| \nabla \bs_h^n \|_{L^2}^2.
		\end{align*}
		Under condition \eqref{alpha-eta} we have
		\[
		\omega :=\alpha ( 1-\eta_k) - \eta_k > 0.
		\]
		Multiplying both sides by $\tau$ and summing up from $k$ to $n$ with $n\leqslant \bar n$ yields, for sufficiently small $\mu$,
		\begin{align*}
		& \tfrac12\omega \tau \sum_{j=k}^n \|\deh^j\|_{L^2}^2 +   \|\nb \bfEh^n\|_G^2 \\
		&\leqslant  c\tau \|\deh^{k-1}\|_{L^2}^2 +\|\nb\bfEh^{k-1}\|_G^2 + c \tau\sum_{j=0}^{n-1} \| \beh^{j} \|_{H^1}^2 + c \tau\sum_{j=k}^n \| \bdh^{j} \|_{L^2}^2 
		+ c \tau\sum_{j=k}^n\| \nabla \bs_h^j \|_{L^2}^2.
		\end{align*}
		The proof is then completed using exactly the same arguments as in
		the last part of the proof of Lemma~\ref{lemma:stability-full - BDF 1 and 2}, by establishing an estimate between $\|\beh^n\|_{L^2}^2$ 
		and $\tau \sum_{j=k}^n \|\deh^j\|_{L^2}^2$ and using a discrete Gronwall inequality, and completing the induction step for
		\eqref{eq:assumed bounds - h}.
	\end{proof}


	\section{Numerical experiments}\label{Se:numer-exp}
	To obtain significant numerical results, we prescribe the exact solution $\m$ on given  three-dimensional  domains $\Om:=[0,1]\times [0,1]\times [0,L]$ with $L\in\{1/100,1/4\}$.
	The discretizations of these domains will consist of a few layers of elements in $z$-direction (one layer for $L=1/100$ and ten layers for $L=1/4$) and a later specified number 
	of elements in $x$ and $y$ directions.  This mimics the common case of thin film alloys as for example in the standard problems of the Micromagnetic Modeling Activity 
	Group at NIST Center for Theoretical and Computational Materials Science (\texttt{ctcms.nist.gov}).  Moreover, this mesh structure helps to keep the computational 
	requirements reasonable and allow us to compute the experiments on a desktop PC.
	We are aware that these experiments are only of preliminary nature and are just supposed to confirm the theoretical results. A more thorough investigation 
	of the numerical properties of the developed method is needed. This will require us to incorporate preconditioning, parallelization of the computations, 
	as well as lower order energy contributions in the effective field~\eqref{Heff} to be able to compare to benchmark results from computational physics. 
	This, however, is beyond the scope of this paper,  and will be the topic of a subsequent work.  
	
	\vspace*{0.2cm}
	
	We consider the time interval $[0,\bar t\,]$ with $\bar t=0.2$ and define two different exact solutions. Since within our computational budget 
	either the time discretization error or the space discretization error dominates, we construct the solutions such that the first one is harder to 
	approximate in space, while the second one is harder to approximate in time. Both solutions are constant in $z$-direction as is often observed in thin-film applications. 
	
	\subsection{Implementation}
	The numerical experiments were conducted using  the finite element package FEniCS (\texttt{www.fenicsproject.org}) on a desktop computer. 
	As already discussed in Section~\ref{SSe:LLG-discr-full}, there are several ways to implement the tangent space restriction. We decided to solve 
	a saddle point problem (variant~(a) in Section~\ref{SSe:LLG-discr-full}) for simplicity of implementation. For preconditioning, we used the black-box AMG preconditioner that comes 
	with FEniCS. Although this might not be the optimal solution, it keeps the number of necessary iterative solver steps within reasonable bounds. 
	Assuming perfect preconditioning, the cost per time-step is then proportional to the number of mesh-elements. We  observed this behavior approximately, 
	although further research beyond the scope of this work is required to give a definite conclusion.
	\subsection{Exact solutions}
	We choose the damping parameter $\alpha=0.2$ and define $g(t):=(\bar t+0.1)/(\bar t+0.1-t)$ as well as $d(x):=(x_1-1/2)^2 + (x_2-1/2)^2$, 
	which is the squared distance of the projection 
	of $x$ to $[0,1]\times[0,1]$ and the point $(1/2,1/2)$. For some constant $C=400$  (a choice made to have pronounced effects),  define
	\begin{equation}\label{eq:exactsol1}
	\m(x,t):=\begin{pmatrix}
	C \e^{-\frac{g(t)}{1/4-d(x)}}(x_1-1/2)\\
	C \e^{-\frac{g(t)}{1/4-d(x)}}(x_2-1/2)\\
	\sqrt{ 1- C^2 \e^{-2\frac{g(t)}{1/4-d(x)}}d(x)}
	\end{pmatrix} \text{if } d(x)\leqslant \frac 14 \text{ and } \m(x,t):=\begin{pmatrix}0\\0\\1\end{pmatrix}\text{else.}
	\end{equation}
	It is easy to check that $|\m(x,t)|=1$ for all $(x,t)\in \varOmega \times [0,\bar t\,]$. Moreover, $\partial_n \m(x,t)=0$ for all $x\in \partial\varOmega$.
	We may calculate the time derivative of $\m$ in a straightforward fashion, i.e., $\partial_t\m(x,t) =0 $ for $d(x)> 1/4$ and
	\begin{equation*}
	\partial_t\m(x,t)= \begin{pmatrix}
	\frac{-g^\prime(t)}{1/4-d(x)} C \e^{-\frac{g(t)}{1/4-d(x)}}(x_1-1/2)\\
	\frac{-g^\prime(t)}{1/4-d(x)} C \e^{-\frac{g(t)}{1/4-d(x)}}(x_2-1/2)\\
	\frac{g^\prime(t)}{1/4-d(x)} C^2 \e^{-2\frac{g(t)}{1/4-d(x)}}\frac{d(x)}{\m_3(x,t)}
	\end{pmatrix}\quad\text{if }d(x)\leqslant \frac 14.
	\end{equation*}
	Here, $\m_3$ denotes the third component of $\m$ as defined above.
	
	\vspace*{0.2cm}
	
	The second exact solution is defined via
	\begin{equation}\label{eq:exactsol2}
	\widetilde \m(x,t):=\begin{pmatrix}
	-(x_1^3-3x_1^2/2 +1/4)\sin(3\pi t/\bar t)\\
	\sqrt{1-(x_1^3-3x_1^2/2 +1/4)^2}\\
	-(x_1^3-3x_1^2/2 +1/4)\cos(3\pi t/\bar t)
	\end{pmatrix}.
	\end{equation}
	Due to the polynomial nature in the first and the third component, and the well-behaved square-root, the space approximation error does not dominate the time approximation.
	
	\subsection{The experiments}
	
	We now may compute the corresponding forcings 
	$\bm{H}$ resp. $\widetilde{\bm{H}}$ to obtain the prescribed solutions by inserting into~\eqref{llg-projection}, i.e.,
	\begin{equation*}
	\bm{H}=\alpha \partial_t \m + \m\times \partial_t\m -  \varDelta \m.
	\end{equation*}  
	(Note that we may disregard the projection $\mathbf{P}(\m)$ from~\eqref{llg-projection} since we solve in the tangent space anyway.)
	We compute $\bm{H}$ numerically by first interpolating $\m$ and $\partial_t\m$ and then computing the derivatives. 
	This introduces an additional error which is not accounted for in the theoretical analysis. However, the examples below confirm the expected 
	convergence rates and hence conclude that this additional perturbation is negligible.
	Figure~\ref{fig:exact} shows slices of the exact solution at different time steps. Figure~\ref{fig:convtau} shows the convergence with respect 
	to the time step size $\tau$, while Figure~\ref{fig:convh} shows convergence with respect to the spatial mesh size $h$. 
	All the experiments confirm the expected rates for smooth solutions.

	\begin{figure}
		\includegraphics[width=0.33\textwidth]{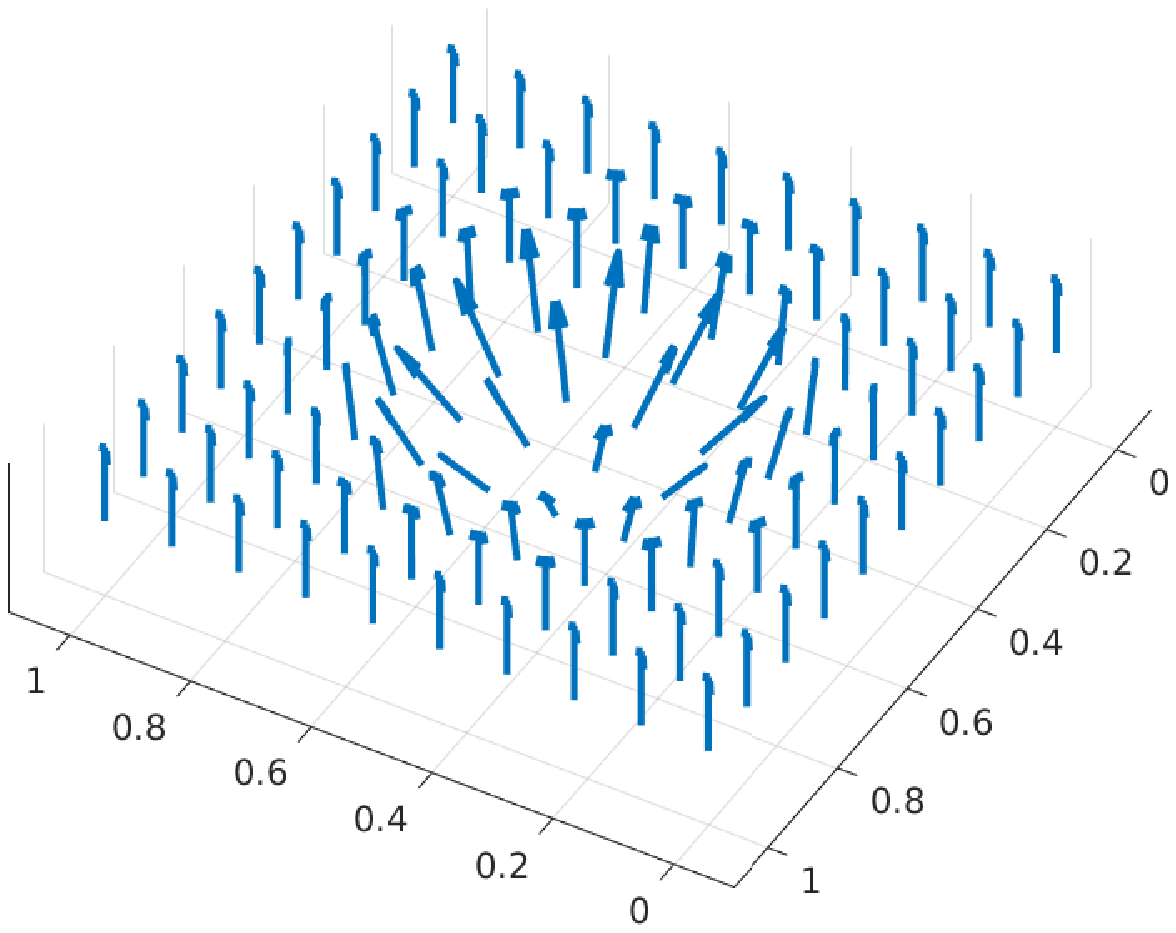}\includegraphics[width=0.33\textwidth]{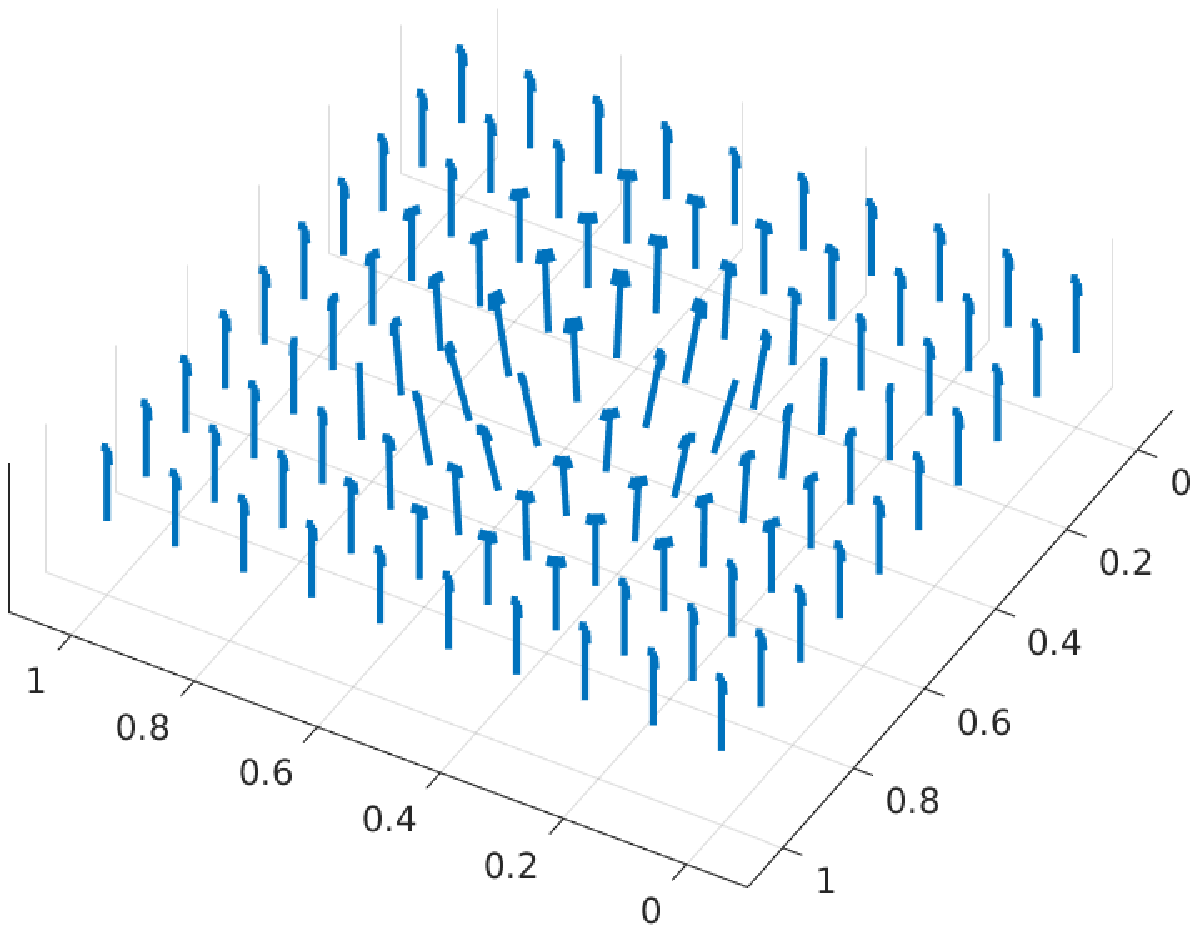}%
		\includegraphics[width=0.33\textwidth]{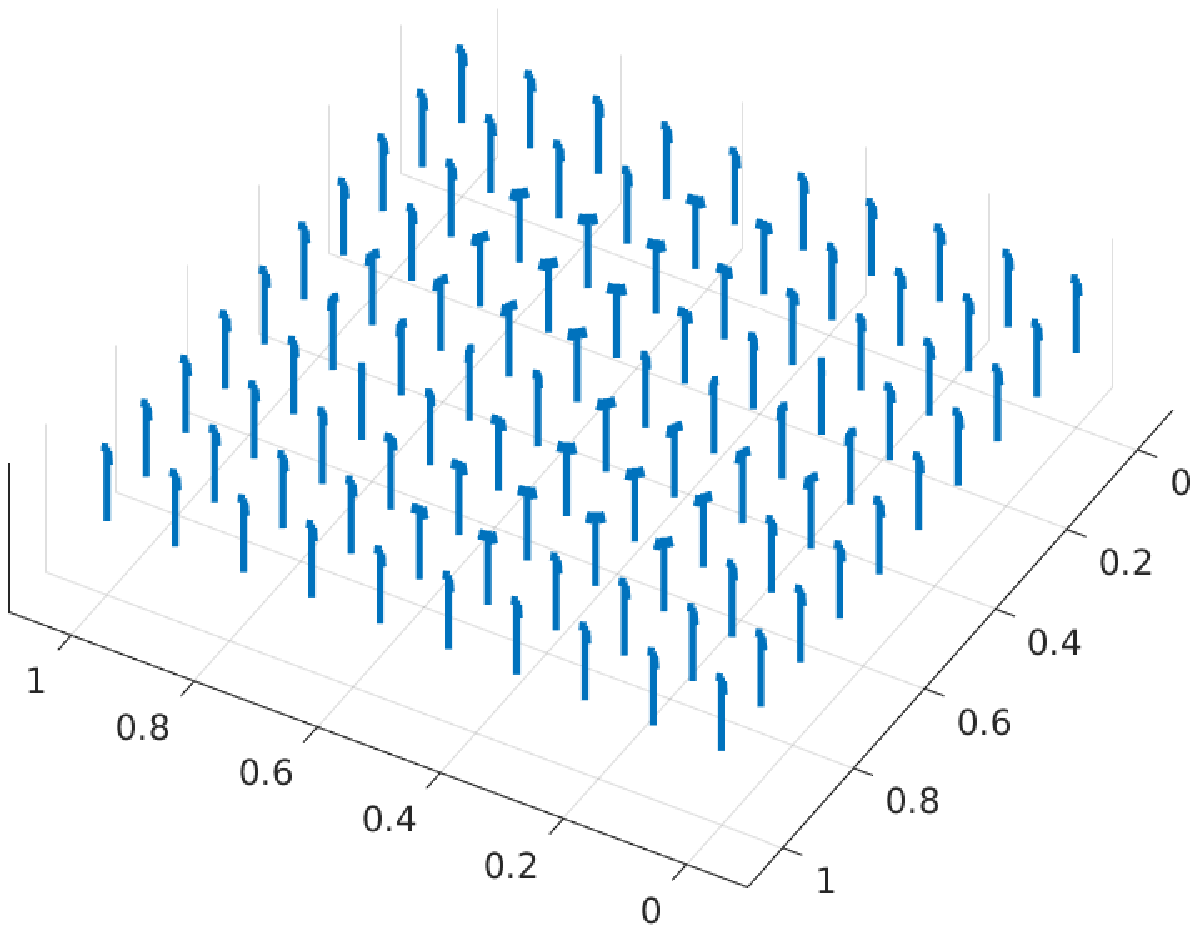}\\
		\includegraphics[width=0.33\textwidth]{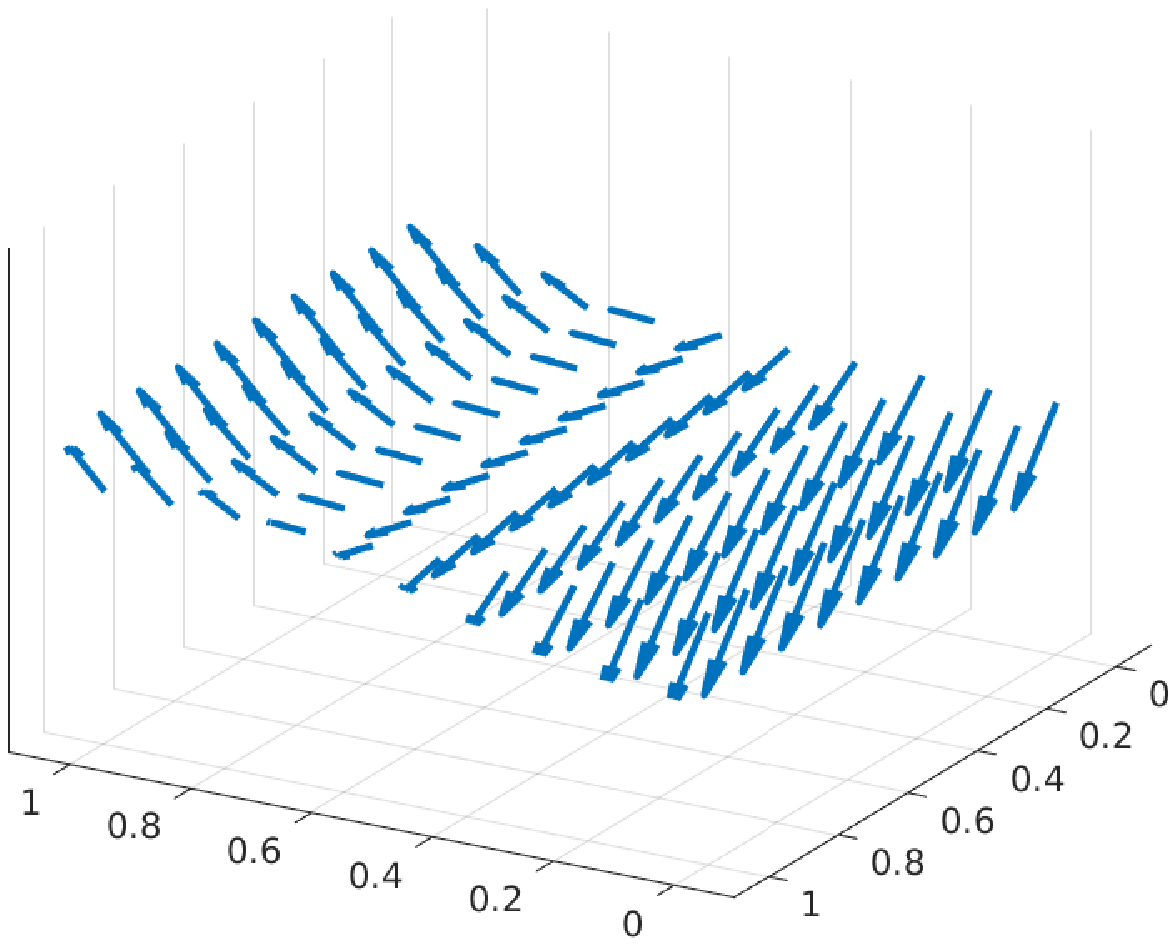}\includegraphics[width=0.33\textwidth]{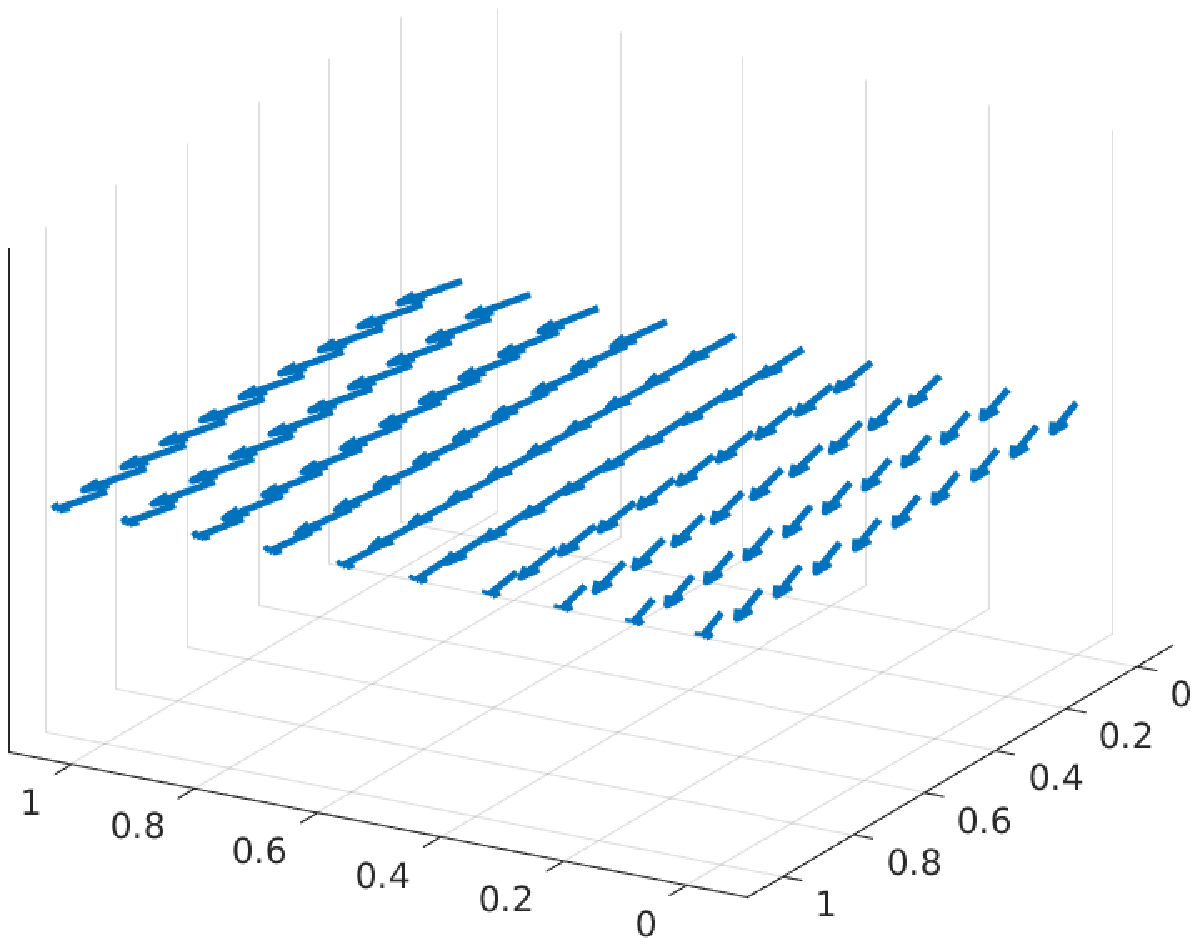}%
		\includegraphics[width=0.33\textwidth]{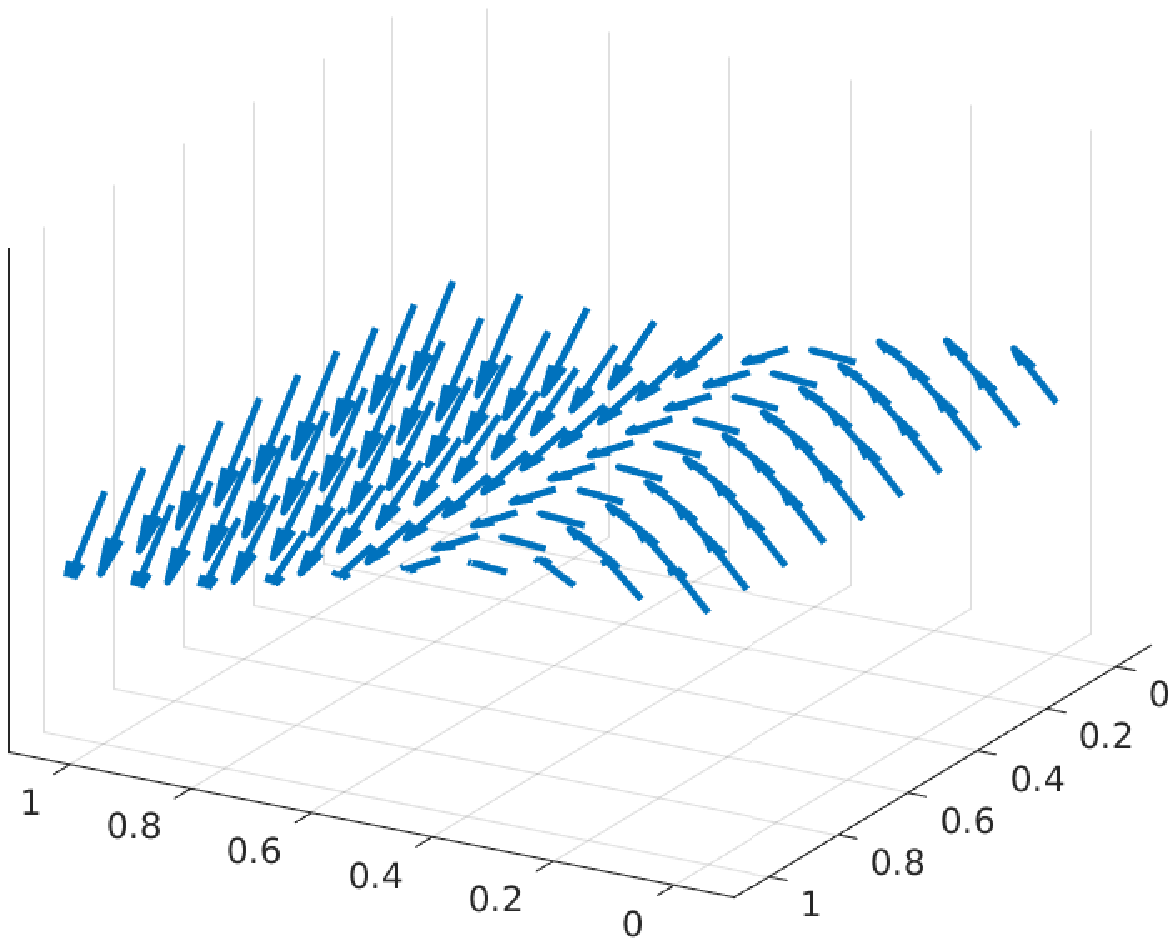}
		\caption{The first row shows the exact solution $\m(x,t)$ from~\eqref{eq:exactsol1} for $x\in [0,1]\times[0,1]\times\{0\}$ and $t\in\{0,0.05,\bar t\}$ 
			(from left to right), whereas the second row shows the exact solution $\widetilde\m(x,t)$ from~\eqref{eq:exactsol2} for $x\in [0,1]\times[0,1]\times\{0\}$ 
			and $t\in\{0,0.2/6,0.2/3\}$ (from left to right). While the problems are three-dimensional, the solutions are constant in $z$-direction and we only show one slice of the solution.}
		\label{fig:exact}
	\end{figure}  
	
	\begin{figure} 
		{\psfrag{v8}[][][.65]{\raisebox{0.2cm}{$\,\,10^{-8}$}}
			\psfrag{v6}[][][.65]{\raisebox{0.2cm}{$\,\,10^{-6}$}}
			\psfrag{v4}[][][.65]{\raisebox{0.2cm}{$\,\,10^{-4}$}}
			\psfrag{v2}[][][.65]{\raisebox{0.2cm}{$\,\,10^{-2}$}}
			\psfrag{v0}[][][.65]{\raisebox{0.2cm}{$\,\,\,10^0$}}
			\psfrag{h3}[][][.65]{\raisebox{0.2cm}{$\,\,\,\,\,\,\,\,10^{-3}$}}
			\psfrag{h2}[][][.65]{\raisebox{0.2cm}{$\,\,\,\,\,\,\,\,10^{-2}$}}
			\psfrag{h1}[][][.65]{\raisebox{0.2cm}{$\,\,\,\,\,\,\,\,10^{-1}$}}
			\psfrag{first}{\tiny{$\,k=1$}}
			\psfrag{second}{\tiny{$\,k=2$}}
			\psfrag{third}{\tiny{$\,k=3$}}
			\psfrag{fourth}{\tiny{$\,k=4$}}
			\psfrag{tau}{\tiny{timestep $\tau$}}
			\includegraphics[width=0.45\textwidth]{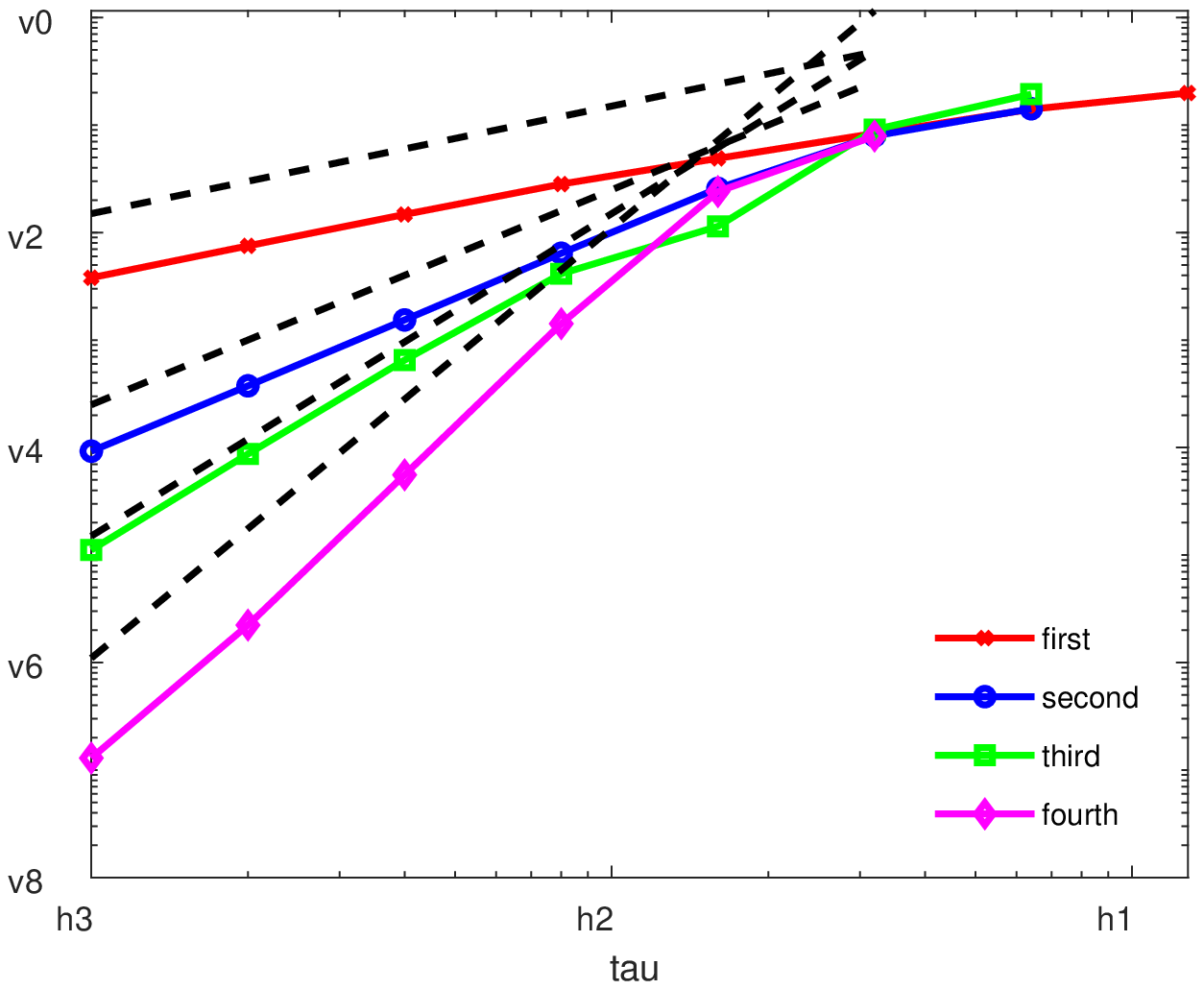}}\hspace{3.2mm}
		{
			\psfrag{v6}[][][.65]{\raisebox{0.2cm}{$\,10^{-6}$}}
			\psfrag{v5}[][][.65]{\raisebox{0.2cm}{$\,10^{-5}$}}
			\psfrag{v4}[][][.65]{\raisebox{0.2cm}{$\,10^{-4}$}}
			\psfrag{v3}[][][.65]{\raisebox{0.2cm}{$\,10^{-3}$}}
			\psfrag{v2}[][][.65]{\raisebox{0.2cm}{$\,10^{-2}$}}
			\psfrag{v1}[][][.65]{\raisebox{0.2cm}{$\,10^{-1}$}}
			\psfrag{v0}[][][.65]{\raisebox{0.2cm}{$\,\,10^0$}}
			\psfrag{v1plus}[][][.65]{\raisebox{0.2cm}{$\!\!\!\!\!10^1$}}
			\psfrag{h3}[][][.65]{\raisebox{0.2cm}{$\,\,\,\,\,\,\,\,10^{-3}$}}
			\psfrag{h2}[][][.65]{\raisebox{0.2cm}{$\,\,\,\,\,\,\,\,10^{-2}$}}
			\psfrag{h1}[][][.65]{\raisebox{0.2cm}{$\,\,\,\,\,\,\,\,10^{-1}$}}
			\psfrag{first}{\tiny{$\,k=1$}}
			\psfrag{second}{\tiny{$\,k=2$}}
			\psfrag{third}{\tiny{$\,k=3$}}
			\psfrag{fourth}{\tiny{$\,k=4$}}
			\psfrag{tau}{\tiny{timestep $\tau$}}
			\includegraphics[width=0.45\textwidth]{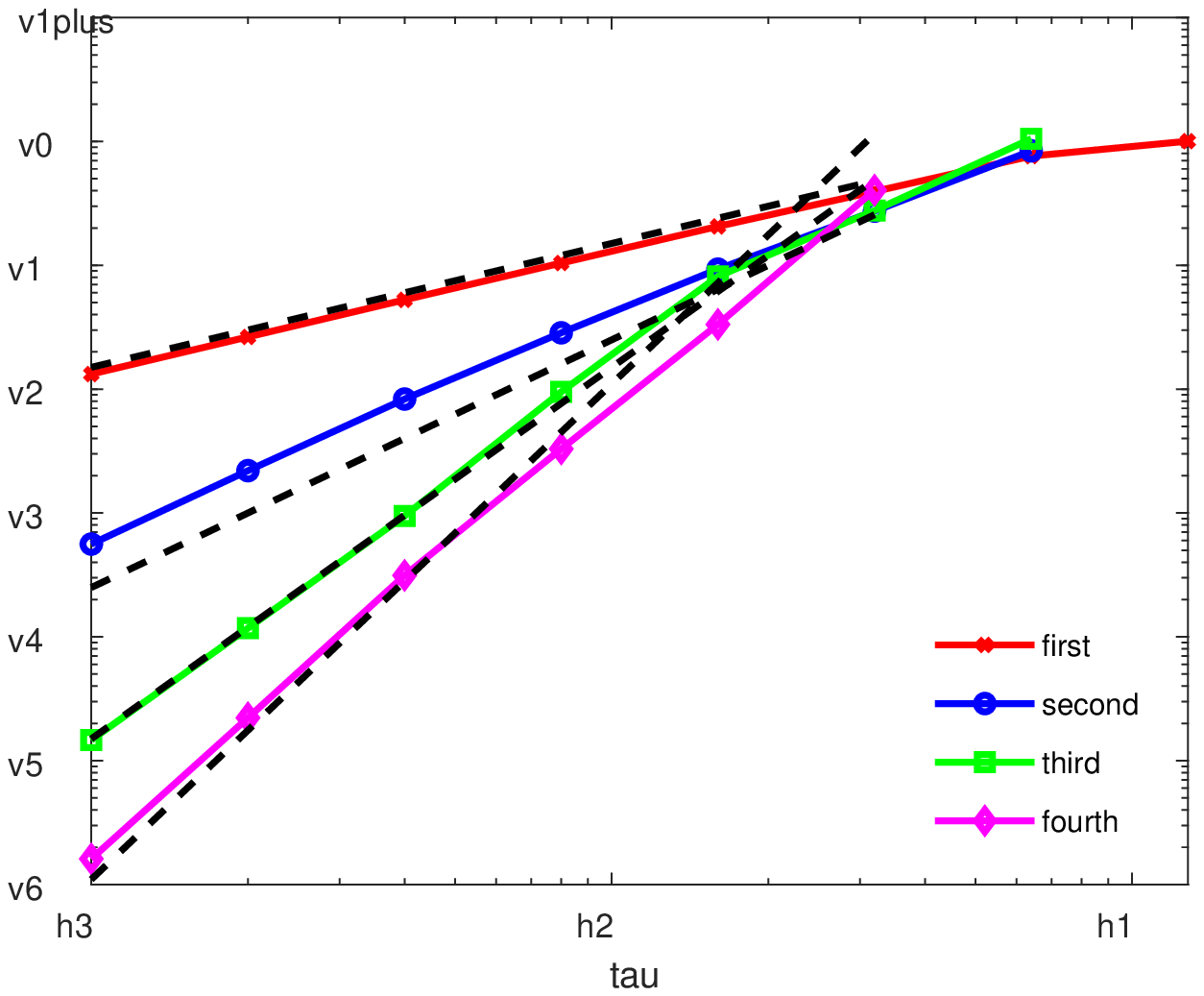}}
		\caption{The plots show the error between computed solutions and exact solution $\widetilde \m$ for a given time stepsize 
			with a spatial polynomial degree of $r=2$ and a spatial mesh size $1/40$ which results in $\approx 6\cdot 10^4$ degrees 
			of freedom per time step in the left plot. In the right plot we use a thicker domain $D=[0,1]\times [0,1]\times [0,1/4]$ with $10$ 
			elements in $z$-direction. This results in $\approx 4\cdot 10^5$ degrees of freedom per timestep. 
			We use the $k$-step methods of order $k\in\{1,2,3,4\}$ and observe the expected rates $\mathcal{O}(\tau^k)$ indicated by the dashed lines. 
			The coarse levels of the higher order methods are missing because the $k$th step is already beyond the final time $\bar t$.} 
		\label{fig:convtau}
	\end{figure}

	Finally, we consider an example with nonsmooth initial data and constant right-hand side. The initial data are given by
	\begin{equation}\label{eq:nonsmoothinit}
	\m_0(x):=\begin{pmatrix}
	x_1-1/2\\
	x_2-1/2\\
	\sqrt{ 1- d(x)}
	\end{pmatrix}\text{ if } d(x)\leqslant \frac 14\text{ and } \m_0(x):=\begin{pmatrix}0\\0\\1\end{pmatrix} \text{else.}
	\end{equation}
	With the constant forcing field $\bm{H}:=(0,1,1)^T$ we compute a numerical approximation to the unknown exact solution. 
	Note that we do not expect any smoothness of the solution (even the initial data is not smooth). Figure~\ref{fig:energy} 
	nevertheless shows a physically
	consistent decay of the energy $\|\nabla \m(t)\|_{L^2(\varOmega)^3}$ over time as well as a good agreement between 
	different orders of approximation. Moreover, the computed approximation shows little deviation from unit length as would 
	be expected for smooth solutions.
	\begin{figure} 
		\psfrag{first}{\tiny{$\ r=1$}}
		\psfrag{second}{\tiny{$\ r=2$}}
		\psfrag{third}{\tiny{$\ r=3$}}
		\psfrag{fourth}{\tiny{$\ r=4$}}
		\psfrag{h}{\tiny{meshsize $h$}} 
		\psfrag{v4}[][][.65]{\raisebox{0.2cm}{$\,\,\,\,10^{-4}$}}
		\psfrag{v3}[][][.65]{\raisebox{0.3cm}{$\,\,\,\,10^{-3}$}}
		\psfrag{v2}[][][.65]{\raisebox{0.3cm}{$\,\,\,\,10^{-2}$}}
		\psfrag{v1}[][][.65]{\raisebox{0.3cm}{$\,\,\,\,10^{-1}$}}
		\psfrag{v0}[][][.65]{\raisebox{0.3cm}{$\,\,\, \,10^0$}}
		\psfrag{h1}[][][.65]{\raisebox{0.2cm}{$\ \ \ \, 10^{-1}$}}
		\includegraphics[width=0.6\textwidth]{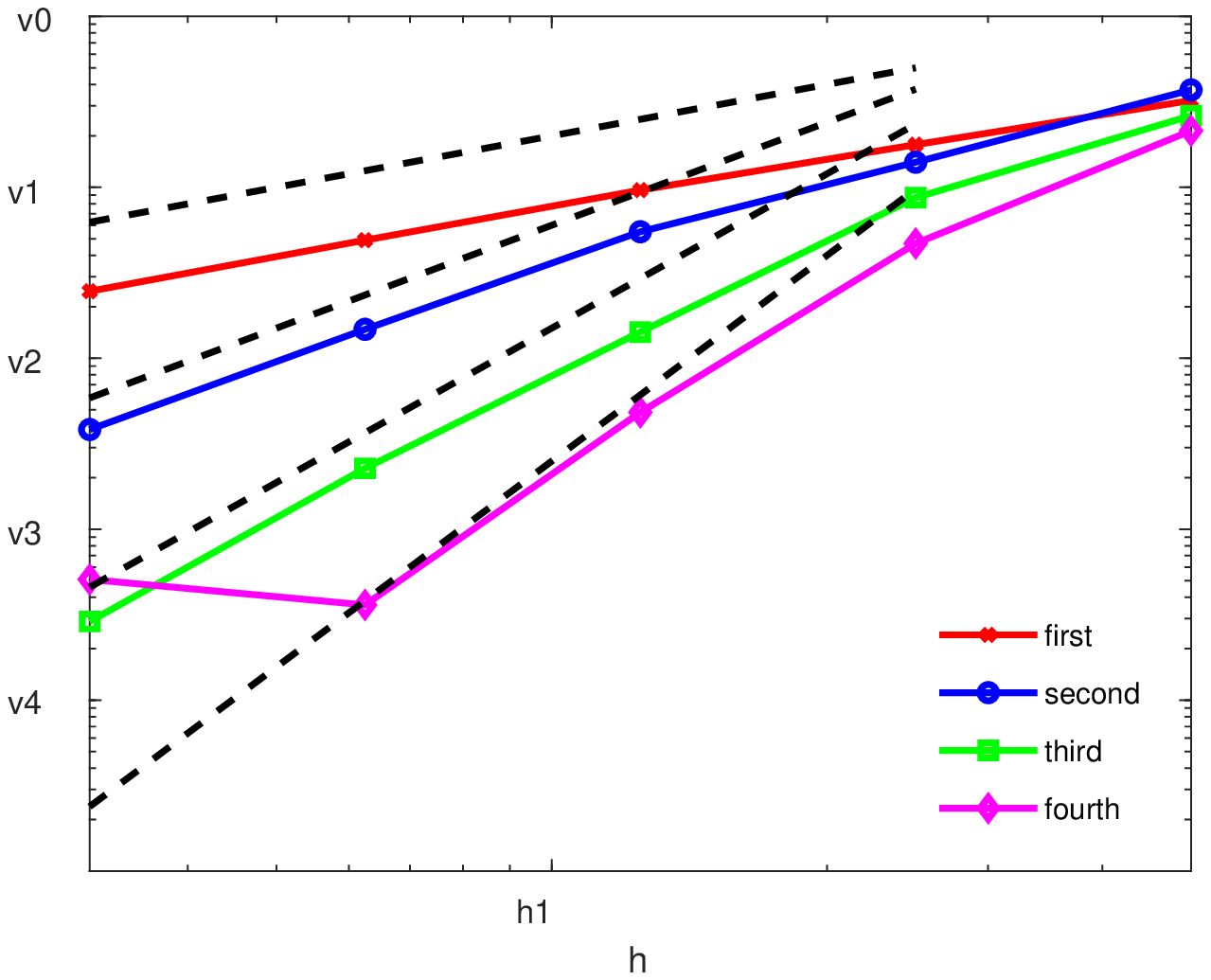}
		\caption{The plot shows convergence in meshsize $h$ with respect to the exact solution $\m$ from~\eqref{eq:exactsol1} 
			on the domain $D=[0,1]\times [0,1]\times [0,1/100]$ with one layer of elements in $z$-direction. We used the second order 
			BDF method with $\tau=10^{-3}$ and spatial polynomial degrees $r\in\{1,2,3,4\}$. The mesh sizes range from $1/2$ to $1/32$. 
			We observe the expected rates $\mathcal{O}(h^r)$ indicated by the dashed lines. The finest mesh-size for $r=4$ does reach 
			the expected error level. This is due to the fact that the time-discretization errors start to dominate in that region.}
		\label{fig:convh}
	\end{figure}
	\begin{figure}
		{
			\psfrag{v15}[][][.65]{$\,\,0.15$}
			\psfrag{v14}[][][.65]{$\,\,0.14$}
			\psfrag{v13}[][][.65]{$\,\,0.13$}
			\psfrag{v12}[][][.65]{$\,\,0.12$}
			\psfrag{v11}[][][.65]{$\,\,0.11$}
			\psfrag{v1}[][][.65]{$\,\,\,0.1$}
			\psfrag{v09}[][][.65]{$\,\,0.09$}
			\psfrag{v08}[][][.65]{$\,\,0.08$}
			\psfrag{v07}[][][.65]{$\,\,0.07$}
			\psfrag{v06}[][][.65]{\raisebox{0.1cm}{$\,\,0.06$}}
			\psfrag{v05}[][][.65]{\raisebox{0.1cm}{$\,\,0.05$}}
			\psfrag{h0}[][][.65]{\raisebox{0.1cm}{$0$}}
			\psfrag{h05}[][][.65]{\raisebox{0.1cm}{$\,\,\,\,\,0.05$}}
			\psfrag{h1}[][][.65]{\raisebox{0.1cm}{$\,\,\,\,0.1$}}
			\psfrag{h15}[][][.65]{\raisebox{0.1cm}{$\,\,\,\,\,0.15$}}
			\psfrag{h2}[][][.65]{\raisebox{0.1cm}{$\,\,\,0.2$}}
			\psfrag{time}{\tiny{time}}
			\psfrag{1}{\tiny{$\ r=k=1$}}
			\psfrag{2}{\tiny{$\ r=k=2$}}
			\psfrag{3}{\tiny{$\ r=k=3$}}
			\psfrag{4}{\tiny{$\ r=k=4$}}
			\psfrag{firstfirstfirstfirst}{\tiny{$\ r=k=1$}}
			\psfrag{second}{\tiny{$\ r=k=2$}}
			\psfrag{third}{\tiny{$\ r=k=3$}}
			\psfrag{fourth}{\tiny{$\ r=k=4$}}
			\includegraphics[width=0.45\textwidth]{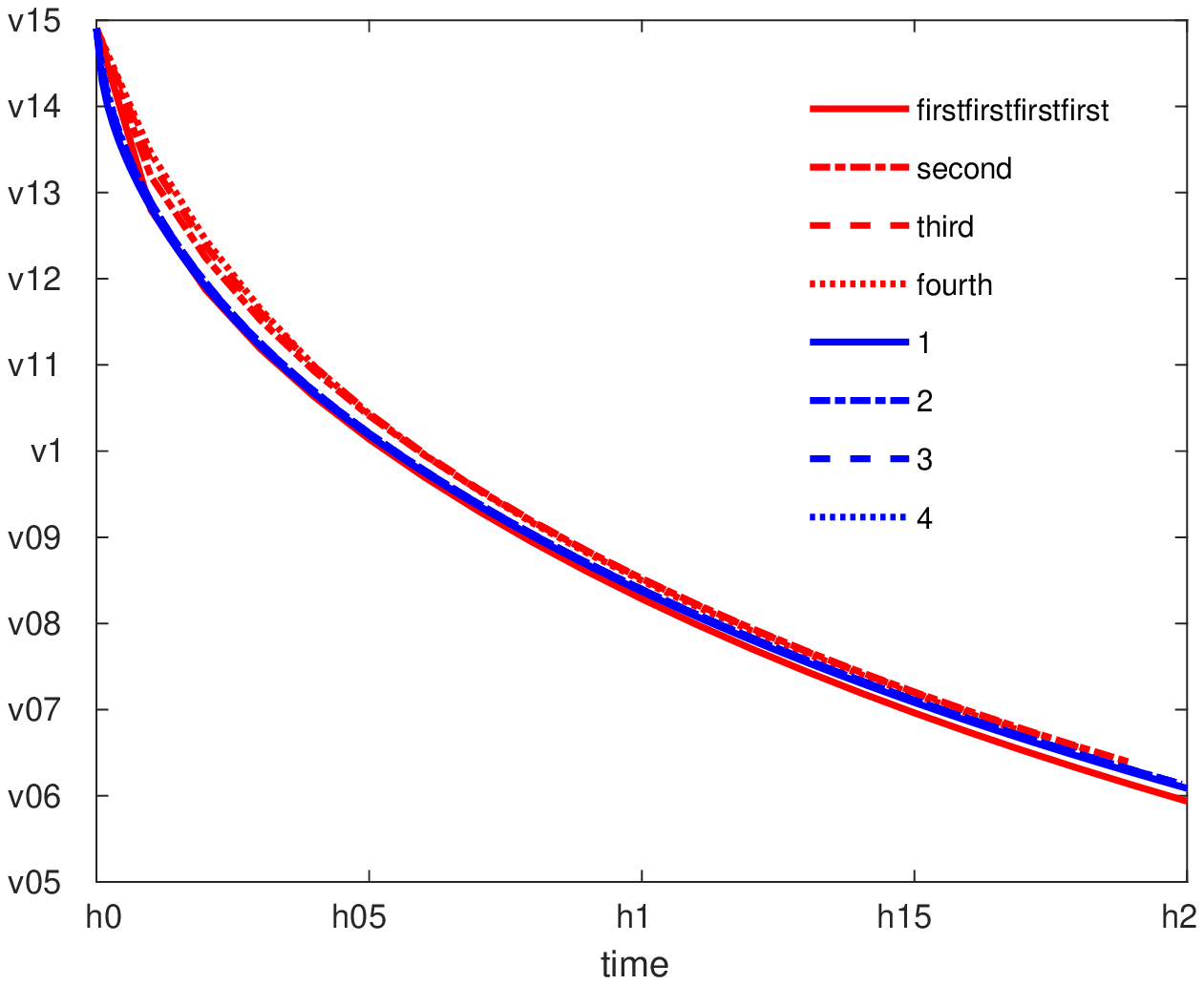}}\hspace{5mm}
		{ \psfrag{v1}[][][.65]{\raisebox{0.2cm}{$\,\,\,10^{-1}$}}
			\psfrag{v2}[][][.65]{\raisebox{0.2cm}{$\,\,\,10^{-2}$}}
			\psfrag{v3}[][][.65]{\raisebox{0.2cm}{$\,\,\,10^{-3}$}}
			\psfrag{h0}[][][.65]{\raisebox{0.1cm}{$0$}}
			\psfrag{h05}[][][.65]{\raisebox{0.1cm}{$\,\,\,\,0.05$}}
			\psfrag{h1}[][][.65]{\raisebox{0.1cm}{$\,\,\,\,0.1$}}
			\psfrag{h15}[][][.65]{\raisebox{0.1cm}{$\,\,\,\,\,0.15$}}
			\psfrag{h2}[][][.65]{\raisebox{0.1cm}{$\,\,\,0.2$}}
			\psfrag{time}{\tiny{time}}
			\includegraphics[width=0.45\textwidth]{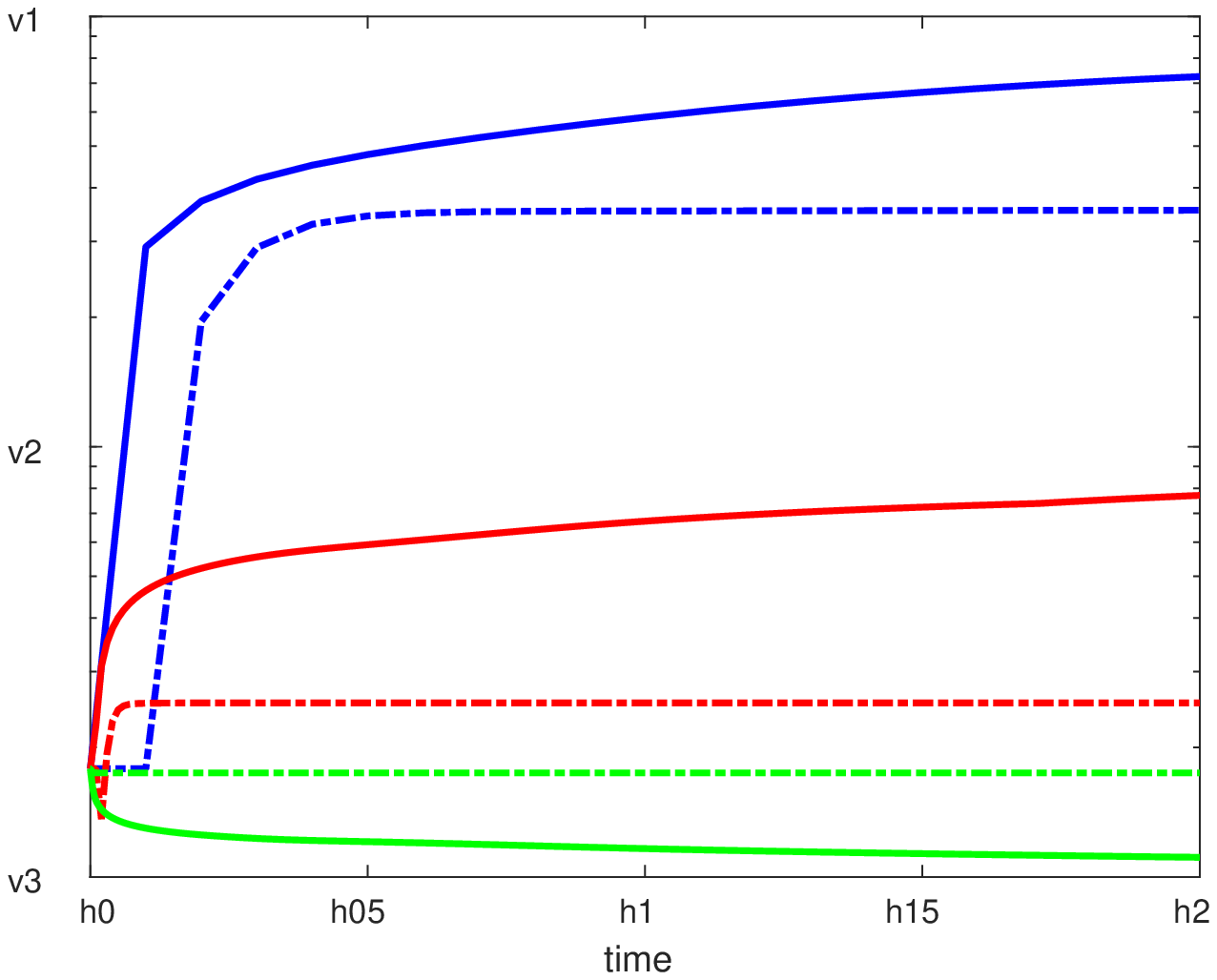}}
		\caption{Left plot: Decay of energies $\|\nabla \m(t)\|_{L^2(\varOmega)^3}$ for the approximations to the unknown solution with $\m_0$ 
			and $\bm{H}$  given in \eqref{eq:nonsmoothinit} and one line after~\eqref{eq:nonsmoothinit}. We plot four approximations of the $k$-step 
			method with polynomial degree $r$ for $r=k\in\{1,2,3,4\}$. 
			The spatial mesh-size is $1/40$ and the size of the timesteps is $10^{-3}$ (blue) and $10^{-2}$ (red). Right plot: Deviation from unit 
			length $\|1-|\m(t)|^2\|_{L^\infty(\Omega)}$ plotted over time for step sizes $\tau =10^{-2}$ (blue), $\tau=10^{-3}$ (red), and $\tau = 10^{-4}$ (green). 
			The solid lines indicate $k=1$, whereas the dashed lines indicate $k=2$. The spatial mesh-size is $1/40$ with $r=1$.}
		\label{fig:energy}
	\end{figure}

	\section{Appendix: Energy estimates for backward difference formulae}
	
	The stability proofs of this paper rely on energy estimates, that is, on the use of positive definite bilinear forms to bound 
	the error $e$ in terms of the defect $d$. This is, of course, a basic technique for studying the time-continuous problem 
	and also for backward Euler and Crank--Nicolson time discretizations (see, e.g., Thom\'ee \cite{Thomee}), but energy 
	estimates  still appear to be not well known for backward difference formula (BDF) time discretizations of order up to $5$, 
	which are widely used for solving stiff ordinary differential equations. To illustrate the basic mechanism, we here just consider 
	the prototypical linear parabolic evolution equation in its weak formulation, given by two positive definite symmetric bilinear 
	forms $(\cdot,\cdot)$ and $a(\cdot,\cdot)$ on Hilbert spaces $H$ and $V$ with induced norms $|\cdot|$ and $\|\cdot\|$, 
	respectively, and with $V$ densely and continuously embedded in $H$. The problem then is to find $u(t)\in V$ such that
	\begin{equation}\label{lin-par}
	(\partial_t u,v) + a(u,v) = (f,v) \qquad\forall v\in V,
	\end{equation}
	with initial condition $u(0)=u_0$. If $u^\star$ is a function that satisfies the equation up to a defect $d$, that is,
	\[
	(\partial_t u^\star,v) + a(u^\star,v) = (f,v) +(d,v)\qquad\forall v\in V,
	\]
	then the error $e=u-u^\star$ satisfies, in this linear case, an equation of the same form,
	\[
	(\partial_t e,v) + a(e,v) = (d,v) \qquad\forall v\in V,
	\]
	with initial value $e_0=u_0-u_0^\star$. Testing with $v=e$ yields 
	\[
	\frac12 \frac{\d}{\d t} |e|^2 + \|e\|^2 = (d,e).
	\]
	Estimating the right-hand side by $(d,e)\leqslant \|d\|_\star \, \|e\|\leqslant \tfrac12 \|d\|_\star^2 + \tfrac12 \|e\|^2$, with the dual norm $\|\cdot\|_\star$, 
	and integrating from time $0$ to $t$ results in the error bound
	\[
	|e(t)|^2 \leqslant |e(0)|^2 + \int_0^t \| d(s) \|_\star^2 \, \d s.
	\]
	On the other hand, testing with $v=\partial_t e$ yields
	\[
	|\partial_t e|^2 + \frac12 \frac{\d}{\d t} \| e \|^2 =(d,\partial_t e),
	\]
	which leads similarly to the error bound
	\[
	\|e(t)\|^2 \leqslant \|e(0)\|^2 + \int_0^t | d(s) |^2 \, \d s.
	\]
	This procedure is all-familiar, but it is not obvious how to extend it to time discretizations beyond the backward Euler 
	and Crank--Nicolson methods. The use of energy estimates for BDF methods relies on the following remarkable results.

	\begin{lemma}{\upshape (Dahlquist \cite{D}; see also \cite{BC} and \cite[Section V.6]{HW})}
		\label{lemma:Dahlquist}
		Let $\delta(\zeta) =\delta_k\zeta^k+\dotsb+\delta_0$ and 
		$\mu(\zeta)=  \mu_k\zeta^k+\dotsb+\mu_0$ be polynomials of degree at 
		most $k\ ($and at least one of them of degree $k)$
		that have no common divisor. 
		Let $(\cdot,\cdot)$ be an inner product with associated norm $|\cdot|.$
		If
		\[\Real \frac {\delta(\zeta)}{\mu(\zeta)}>0\quad\text{for }\, |\zeta|<1,\]
		then there exists a positive definite symmetric matrix $G=(g_{ij})\in \R^{k\times k}$ 
		such that for $v_0,\dotsc,v_k$ in 
		the real inner product space,
		\begin{equation*}
		\Big (\sum_{i=0}^k\delta_iv_{k-i},\sum_{j=0}^k\mu_jv_{k-j}\Big ) \geqslant
		\sum_{i,j=1}^kg_{ij}(v_{i},v_{j})
		-\sum_{i,j=1}^kg_{ij}(v_{i-1},v_{j-1}).
		\end{equation*}
	\end{lemma}

	In combination with the preceding result for the multiplier $\mu(\zeta)=1-\eta_k\zeta,$ 
	the following property of BDF methods up to order $5$ becomes important.
	
	\begin{lemma}{\upshape (Nevanlinna \& Odeh \cite{NO})}\label{lemma:NO}
		For $k\leqslant 5,$  there exists 
		$0\leqslant \eta_k<1$ such that 
		for
		$\delta (\zeta)= \sum_{ \ell=1}^k \frac 1 \ell  (1-\zeta)^ \ell$, 
		\[\Real \frac {\delta(\zeta)}{1-\eta_k\zeta}>0\quad\text{for }\, |\zeta|<1.\]
		The smallest possible values of $\eta_k$ are
		\[\eta_1=\eta_2=0,\ \eta_3=0.0836,\ \eta_4=0.2878,\ \eta_5=0.8160.
		\]
	\end{lemma}
	
	Precise expressions for the optimal multipliers for the BDF methods of orders $3, 4$ and $5$ are given by Akrivis \& Katsoprinakis \cite{AK}.
	
	An immediate consequence of Lemma \ref{lemma:NO} and Lemma \ref{lemma:Dahlquist} is the
	relation
	\begin{equation}
	\label{multiplier}
	\Big (\sum_{i=0}^k\delta_iv_{k-i},v_k-\eta_k v_{k-1}\Big )\geqslant
	\sum_{i,j=1}^kg_{ij}(v_{i},v_{j})
	-\sum_{i,j=1}^kg_{ij}(v_{i-1},v_{j-1})
	\end{equation}
	with a positive definite symmetric matrix $G=(g_{ij})\in \R^{k\times k}$;  
	it is this inequality  that plays a crucial role in our energy estimates, and the same inequality for the inner product $a(\cdot,\cdot)$. \medskip
	
	The error equation for the BDF time discretization of the linear parabolic problem \eqref{lin-par} reads
	\[
	(\dot e^n, v ) + a(e^n,v) = (d^n,v) \qquad\forall v\in V,\qquad\text{where }\quad \dot e^n = \frac1\tau\sum_{j=0}^k \delta_j e^{n-j},
	\]
	with starting errors $e^0,\dotsc,e^{k-1}$. When we test with $v=e^n-\eta_k e^{n-1}$, the first term can be estimated 
	from below by \eqref{multiplier}, the second term is bounded from below by $ (1-\tfrac12\eta_k) \|e^n\|^2 - \tfrac12\eta_k \|e^{n-1}\|^2 $, 
	and the right-hand term is estimated from above by the Cauchy-Schwarz inequality. Summing up from $k$ to $n$
	then yields the error bound
	\begin{equation}
	\label{e-test}
	| e^n |^2 + \tau \sum_{j=k}^n \| e^j \|^2 \leqslant C_k \Bigl(\sum_{i=0}^{k-1} \bigl(|e^i|^2 + \tau \| e^i \|^2 \bigr) + \tau \sum_{j=k}^n \|d^j\|_\star^2\Bigr),
	\end{equation}
	where $C_k$ depends only on the order $k$ of the method. 
	This kind of estimate for the BDF error has recently been used for a variety of linear and nonlinear parabolic problems \cite{LMV13,AL15,ALL17,KL17}.  
	
	On the other hand, when we first subtract $\eta_k$ times the error equation for $n-1$ from the error equation with $n$ and then test 
	with $\dot e^n$, we obtain
	\[
	(\dot e^n -\eta_k \dot e^{n-1}, \dot e^n) + a( e^n -\eta_k e^{n-1},\dot e^n) = (d^n -\eta_k d^{n-1}, \dot e^n).
	\]
	Here, the second term is bounded from below by \eqref{multiplier} with the $a(\cdot,\cdot)$ inner product, the first term 
	is bounded from below by $ (1-\tfrac12\eta_k) |\dot e^n|^2 - \tfrac12\eta_k |\dot e^{n-1}|^2 $, and the right-hand term is 
	estimated from above by the Cauchy--Schwarz inequality. Summing up from $k$ to $n$
	then yields the error bound
	\begin{equation}\label{edot-test}
	\| e^n \|^2 + \tau \sum_{j=k}^n | \dot e^j |^2 \leqslant C_k \Bigl(\sum_{i=0}^{k-1} \|e^i\|^2 + \tau \sum_{j=k}^n |d^j|^2\Bigr).
	\end{equation}
	It is this type of estimate that we use in the present paper for the nonlinear problem considered here. It has previously been used in \cite{KLL18}.

	\subsection*{Acknowledgment}
	The work of Michael Feischl, Bal\'azs Kov\'acs and Christian Lubich is supported by Deutsche Forschungsgemeinschaft -- Project-ID 258734477 -- SFB 1173.

	\bibliographystyle{amsplain}

\end{document}